\documentclass[11pt,handout]{article}%
\usepackage{geometry}                % See geometry.pdf to learn the layout options. There are lots.
\usepackage{tikz} 
\usetikzlibrary{positioning}
\usepackage{graphicx}
\usepackage{amssymb}
\usepackage{epstopdf}
\usepackage{subfigure}  % use for side-by-side figures
\usepackage[sans]{dsfont}
\usepackage[english]{babel}
\usepackage{latexsym}
\usepackage{mathrsfs}
\usepackage{graphicx}
\usepackage{color}
\usepackage{float}
\usepackage{mathtools}
\usepackage{amsmath}
\usepackage{amsfonts}
\numberwithin{equation}{section}
\usepackage{enumerate}
\usepackage{amsthm}
\usepackage[title]{appendix}

\usepackage[bookmarks=true,colorlinks=true,linkcolor={blue},urlcolor={blue}, citecolor={blue},pdfstartview={XYZ null null 1.22}]{hyperref}%
	\newcommand{\R}{\mathbb R}

	\newcommand{\abs}[1]{\left|#1\right|}
	
	\def\be#1\ee{\begin{equation}#1\end{equation}}
	
	\newtheorem{theorem}{Theorem}[section]
	\newtheorem{proposition}{Proposition}

	\theoremstyle{definition}
	\newtheorem{algorithm}{Algorithm}[section]
	\newtheorem{remark}{Remark}

	\def\RR{\mathbb R}

	\providecommand{\R}{\mathbb{R}}
	\providecommand{\Pcal}{\mathcal{P}}
	\providecommand{\Ecal}{\mathcal{E}}
	\providecommand{\Ccal}{\mathcal{C}}
	\providecommand{\Scal}{\mathcal{S}}
	\providecommand{\Fcal}{\mathcal{F}}
	\providecommand{\xE}{\bar{x}_{\Ecal}}
	\providecommand{\dd}{\,\mathrm{d}}
	\providecommand{\supp}{\operatorname{supp}}
	\providecommand{\Id}{\mathrm{Id}}
	\providecommand{\Vpe}{V_p^{\eta}}
	\providecommand{\thetaq}[2]{\ensuremath{\Theta_{#1}(#2)}}
	\newtheorem{assumption}{Assumption}
	\newtheorem{lemma}{Lemma} 
	
	\newtheorem{definition}{Definition}

	\title{Consensus-based optimization with $\alpha$-stable jump processes}

	\author{ 
Pedro  Aceves-Sanchez\thanks{Department of Mathematics, University of Arizona, Tucson, Arizona, USA 
	(pedroas@arizona.edu)},  \and Giacomo Albi\thanks{Department of Computer Science, University of Verona, ITALY (giacomo.albi@univr.it)}, \and F. Ferrarese \thanks{Department of Mathematics and Computer Science \& Center for Modeling, Computing and Statistics (CMCS), University of Ferrara, via Machiavelli 30, 44121 Ferrara, ITALY (federica.ferrarese@unife.it)}, \and Michael Herty \thanks{Institute for Geometry and Applied Mathematics (IGPM), RWTH Aachen University, GERMANY(herty@igpm.rwth-aachen.de) and
	Extraordinary Professor, Department of Mathematics and Applied Mathematics,  University of Pretoria, South Africa.}  
	}
	
	\date{}
	
\begin{document}

\maketitle

% REQUIRED
\begin{abstract} 
	In this paper, we introduce a CBO method that incorporates jumps driven by an $\alpha$-stable stochastic process within a kinetic framework. This approach naturally induces nonlocal stochastic effects, enhancing the exploration capabilities of the method. We formulate the dynamics at the particle level, providing a detailed description of the associated stochastic system and its asymptotic behavior. In particular, by means of a Fourier-based representation, we derive the corresponding fractional Fokker–Planck equation, which captures the nonlocal diffusion mechanisms generated by $\alpha$-stable processes. As a central result, we establish a rigorous convergence analysis for the proposed method. Finally, we assess its performance through a series of numerical experiments, highlighting the effectiveness of $\alpha$-stable jump dynamics and their advantages over standard diffusion-based strategies, especially in complex optimization landscapes.
\end{abstract}

	\section{Introduction}\label{sec:intro} 
Gradient-free optimization methods have attracted significant interest in recent years, especially in settings where the objective function is nonconvex, nonsmooth, noisy, or prohibitively expensive to evaluate \cite{blum2003metaheuristics,borghi2024kinetic,kennedy1995particle}. These approaches are particularly effective in high-dimensional and multi-modal landscapes, where gradient-based methods may fail due to the presence of multiple local minima, lack of regularity, or unreliable derivative information \cite{borghi2023adaptive,ferrarese2025localized}.
Classical examples of gradient-free optimization techniques include particle swarm optimization (PSO) \cite{grassi2021particle,kennedy1995particle}, genetic algorithms (GA) \cite{michalewicz1996genetic,mitchell1995genetic}, and more general evolutionary strategies, that relies on population-based heuristics
such as in \cite{blum2003metaheuristics}, typically inspired by biological, and social behaviors \cite{bellomo2021life,bellomo2022crowds}. In PSO, candidate solutions evolve by combining individual memory with social interactions, while GA employs selection, mutation, and recombination mechanisms to iteratively improve a population of individuals.
More recently, consensus-based optimization (CBO) has emerged as a mathematically grounded alternative within the class of population-based, gradient-free methods \cite{carrillo2018analytical,fornasier2020consensus,fornasier2021consensus,franceschi2026superlinear}. In CBO, a swarm of agents evolves according to stochastic dynamics that drive the particles toward a weighted consensus point biased by the objective function, while random perturbations ensure sufficient exploration of the search space. Closely related are kinetic-based optimization (KBO) methods, which can be interpreted as the kinetic (mesoscopic) description of a time-discretized version of the CBO dynamics, typically obtained via the Euler–Maruyama scheme \cite{benfenati2022binary,albi2023kinetic}. More precisely, given an objective function $\mathcal{E}:\RR^d\to \RR$, which has a unique global minimum at $x^\star$, the goal is to solve
\begin{equation}
	\min_x \mathcal{E}(x).
\end{equation}
Hence, in KBO approach a population of $N$ agents is considered, where each particle at position $x_i$, for any $i=1, \ldots ,N$, evolves according to the update rule
\begin{equation}\label{eq:KBO}
	x_i' = x_i + \nu   (\bar{x}_{\mathcal{E}}(t) - x_i) + \sigma  D(\bar{x}_{\mathcal{E}}(t), x_i)   z_i,
\end{equation}
with $x'_i$ denoting the post-interaction position. Here, $\nu$ and $\sigma$ are positive parameters controlling the relative strength of drift and diffusion, $z_i \sim \mathcal{N}(0,1)$, for any $i=1,\ldots, N$, represents a stochastic perturbation, and $D(\bar{x}_{\mathcal{E}}(t), x_i)$ is a diffusion matrix. The quantity $\bar{x}_{\mathcal{E}}(t)$ represents an estimate of the global minimizer, computed as a weighted average of the particle positions \cite{PTTM2017}.
At the mesoscopic level, the evolution of the particle distribution $f(x,t)$ is governed by the Boltzmann-type equation
\begin{equation}\label{eq:boltzmann_kbo}
	\partial_t f(x,t) = \eta\mathcal{Q}(f)(x,t),
\end{equation}
where $\eta>0$ is the interaction frequency and $\mathcal{Q}(\cdot)$ is an interaction operator accounting for the mass gain and loss resulting from the interaction dynamics \eqref{eq:KBO}. In kinetic optimization frameworks, agent dynamics can be modeled via either multiple or binary interactions \cite{benfenati2022binary}. In multiple interaction schemes such as \eqref{eq:KBO}, particles update their states simultaneously using the global consensus estimator $\bar{x}_{\mathcal{E}}(t)$. Conversely, binary frameworks rely on pairwise interactions between randomly selected agents. While binary rules naturally yield the Boltzmann-type collision operator $\mathcal{Q}(f)$ in \eqref{eq:boltzmann_kbo} and enable efficient simulation via Direct Simulation Monte Carlo methods \cite{nanbu1980direct,pareschi2001introduction,pareschi2013interacting}, they recover the exact same mean-field Fokker--Planck description in the grazing interaction limit as all-to-all dynamics. This kinetic setting bridges discrete particle updates and continuous mean-field models, allowing the long-time behavior and convergence of the system to be rigorously analyzed through Fokker--Planck approximations. Extensions incorporating local interaction mechanisms, hybrid strategies combining KBO with genetic algorithms, and hydrodynamic limits of fractional Fokker--Planck equations have been investigated in \cite{albi2023kinetic,benfenati2022binary,pasCesbron2019}. 
% At the mesoscopic level, the evolution of the particle distribution $f(x,t)$ can be described according to
% \begin{equation}\label{eq:boltzmann_kbo}
	% 	\partial_t f(x,t) = \eta \, \mathcal{Q}(f)(x,t),
	% \end{equation}
% which corresponds to a Boltzmann-type equation, where $\eta>0$ is the interaction frequency and the operator $\mathcal{Q}(\cdot
% )$ accounts for the gain and loss of mass at position $x$ and time $t>0$, resulting from the interaction dynamics \eqref{eq:KBO}.
% This kinetic formulation provides a mesoscopic description of the particle dynamics via a Boltzmann equation, which can be efficiently simulated using Direct Simulation Monte Carlo methods \cite{nanbu1980direct,pareschi2001introduction,pareschi2013interacting}. Furthermore, its long-time behavior characterizes the convergence of the particle system and, under suitable asymptotic regimes, yields Fokker--Planck approximations that enable a tractable convergence analysis. Extensions incorporating local interaction mechanisms between particles, as well as hybrid strategies combining KBO with genetic algorithms, have been proposed in \cite{albi2023kinetic,benfenati2022binary}. Finally, the hydrodynamic limits of fractional Fokker-Planck equations have been recently studied. See, for instance \cite{pasCesbron2019}.

The aim of this paper is to introduce a novel KBO model incorporating, in addition to classical Brownian diffusion, stochastic perturbations driven by an $\alpha$-stable L\'evy process \cite{applebaum2009levy,ken1999levy,kyprianou2022stable,meerschaert2019stochastic}. This combined framework enhances exploration by allowing agents to perform both local random walks and nonlocal jumps, effectively coupling diffusive and jump-driven dynamics. The proposed approach extends the models introduced in \cite{kalise2023consensus}, where jump-diffusion CBO dynamics were constructed by combining Wiener processes with compound Poisson noise. We also refer to \cite{borghi2025swarm} for an alternative approach to account for large jumps in CBO-dynamics based on BGK-type kinetic equations.

In particular, the standard jump-diffusion KBO model at the particle level reads
\begin{equation}\label{eq:jump_diffusion}
	x_i' = x_i + \nu(t)\big(\bar{x}_\mathcal{E}(t) - x_i\big) 
	+ \sigma(t) D\big(\bar{x}_\mathcal{E}(t), x_i\big) z_i 
	+ \gamma(t) D\big(\bar{x}_\mathcal{E}(t), x_i\big) \tilde{z}_i,
\end{equation}
for $i=1,\ldots,N$, where
\begin{equation}\label{eq:finite_poisson_measure}
	z_i \sim \mathcal{N}(0,1), 
	\qquad 
	\tilde{z}_i \sim \sum_{j=1}^{N^i(t)} w_j.
\end{equation}
Here, $N^i(t)$ are independent Poisson processes with intensity $\lambda$, and $w_j$ are i.i.d.\ $d$-dimensional random variables with zero mean and probability density $\rho(w)$, representing the jump amplitudes. The coefficients $\nu(t)$ and $\sigma(t)$ are continuous, positive, and non-decreasing functions such that $\nu(t)\to \nu>0$ and $\sigma(t)\to \sigma>0$ as $t\to\infty$, while $\gamma(t)$ is continuous, non-negative, and non-increasing with $\gamma(t)\to \gamma\geq 0$. The matrix $D\big(\bar{x}_\mathcal{E}(t), x\big)$ denotes the diffusion tensor, and $\bar{x}_\mathcal{E}(t)$ is the consensus point defined as a weighted average of the particle positions.
The jump component can be equivalently expressed in terms of a Poisson random measure as
\begin{equation}\label{eq:jump_diffusion_continuos}
	\tilde{z}_i \sim \int_{\mathbb{R}} w \, \mathcal{N}^i(dt,dw),
\end{equation}
where $\mathcal{N}^i(dt,dw)$ are independent Poisson random measures with intensity $\lambda \pi(dw)\,dt$, with $\lambda>0$ denoting the jump intensity and $\pi(dw)$ the probability distribution of the jump amplitudes \cite{applebaum2009levy,ken1999levy}. 
Several variants of this model can be considered. For instance, the Poisson intensity $\lambda$ may depend on time, with $\lambda(t)\to 0$ as $t\to\infty$, or the jump noise may be shared across particles, leading to simultaneous jumps with independent amplitudes. 

Here, the main objective is to formally derive the corresponding Fokker-Planck equation for the particle system, which naturally incorporates a fractional diffusion operator to account for the $\alpha$-stable jumps. A central challenge of these models is that $\alpha$-stable processes generally have infinite variance, which complicates both analytical treatment and interpretation of the concentration and convergence properties. Hence, a complete convergence proof will be provided, relying on the use of the Wasserstein $p$-distance between the evolving particle distribution and a Dirac measure centered at the global minimizer, inspired by the approach in \cite{fornasier2024consensus}.

The rest of the paper is organized as follows. In Section \ref{sec:particles_representation} we introduce the KBO method with $\alpha$-stable process, focusing on the particle evolution. In Section \ref{sec:FPE} we formally derive the corresponding fractional Fokker-Planck equation. In Section \ref{sec:convergence} we present a detailed convergence proof for the KBO method with $\alpha$ stable process. In Section \ref{sec:numerical_methods}, we introduce the numerical method used to simulate the particle dynamics, and in Section \ref{sec:numerical_experiments}, we propose different numerical experiments to validate the effectiveness of the method, and compare it to the classical KBO method. Finally, in Section \ref{sec:conclusions} we present some future research lines. 
\section{Kinetic framework for CBO with $\alpha$-stable processes}\label{sec:particles_representation}
We introduce a KBO model in which stochastic exploration is driven by a  rotationally  symmetric $\alpha$-stable L\'evy process, extending the jump-diffusion formulation in \eqref{eq:jump_diffusion}. This choice allows particles to perform nonlocal jumps and enhances exploration beyond standard Brownian diffusion. In contrast to compound Poisson noise, $\alpha$-stable processes exhibit infinite activity and heavy-tailed increments, capturing long-range effects in the particle dynamics.
To this aim, we consider $f(x,t)
\geq 0$ and $x\in \mathbb{R}^d$ the distribution of particles in position $x$ at time $t
\geq 0$, assuming $f(x,
\cdot)$ a probability density function.
We suppose that each particle changes its position according to the following updating rules 
\begin{subequations}\label{eq:KBO_alfa_stable}
	\begin{align}
		x'   &= x + \nu (\bar{x}_\mathcal{E}(t) - x), \\
		x''  &= x' + \sigma D(\bar{x}_\mathcal{E}(t),x')\, z  + \gamma D(\bar{x}_\mathcal{E}(t),x')\, \tilde{z},
	\end{align}
\end{subequations}
where $\nu,\sigma,\gamma \geq 0$, $z \sim \mathcal{N}(0,I_d)$, and $\tilde{z}$ is rotationally symmetric $\alpha$-stable random vector with probability density $g_\alpha^{\mathrm{RS}}$  for $\alpha\in (0,2]$. Such process introduces an infinite-activity jump mechanism, in contrast to compound Poisson noise with discrete jumps. In particular, for $\alpha \in [0,2)$ the displacements follow a heavy-tailed distribution, with increasingly pronounced jump behavior as $\alpha$ decreases.

In \eqref{eq:KBO_alfa_stable}, $D(\bar{x}_\mathcal{E}(t),x)$ is a diffusion matrix, defined as either 
\begin{equation}\label{eq:D_isotropic}
	D(\bar{x}_\mathcal{E}(t),x) = \vert \bar{x}_\mathcal{E}(t)-x \vert 
	\textrm{Id}_d, 
\end{equation}
in the case of isotropic diffusion, where $\textrm{Id}_d$ is the identity matrix of size $d \times d$, or 
\begin{equation}\label{eq:D_anisotropic}
	D(\bar{x}_\mathcal{E}(t),x)  = diag\{(\bar{x}_\mathcal{E}(t)-x)_1,\ldots,(\bar{x}_\mathcal{E} (t)-x)_d\},
\end{equation} 
in the case of anisotropic diffusion. 
Furthermore, the term $\bar{x}_\mathcal{E} (t)$ is the global estimate of the best position of the minimizer, which is computed as a convex combination of particle locations weighted by the cost function:
\begin{equation}\label{eq:x_tot}
	\bar{x}_\mathcal{E} (t) = \frac{\int_{\mathbb{R}^d}x e^{-\beta \mathcal{E}(x)}f(x,t) \dd x}{\int_{\mathbb{R}^d} e^{-\beta\mathcal{E}(x)}f(x,t) \dd x} \, .
\end{equation} 
The choice of the consensus point is made according to the Laplace principle \cite{dembo2009large,miller2006applied}, which yields
\begin{equation}\label{eq:laplace}
	\lim_{\beta \to \infty} 
	\left(
	-\frac{1}{\beta} \log \int_{\mathbb{R}^d} e^{-\beta \mathcal{E}(x)} f(x,t)\, \dd x
	\right)
	=
	\inf_{x \in \mathrm{supp}\, f(x,t)} \mathcal{E}(x).
\end{equation}
\begin{remark}
	The probability density $g_\alpha^{\mathrm{RS}}$ introduced above can be
	interpreted as the unit-time law of a rotationally symmetric
	$\alpha$-stable L\'evy process, generally denoted by
	$\mathcal{L}_t^{\alpha}$.
	We recall that, by the L\'evy representation theorem, the process $\mathcal{L}_t^\alpha$ admits the representation
	\begin{equation}
		\mathcal{L}_t^\alpha = \int_{0}^t w \,\tilde{N}(\dd s, \dd w),
		\qquad 
		\tilde{N}(\dd s,\dd w) = N(\dd s, \dd w) - \pi(\dd w)\, \dd s, 
	\end{equation}
	where $\tilde{N}(ds,dw)$ is a compensated Poisson random measure with L\'evy measure
	\begin{equation*}
		\pi(dw) = C_{\alpha} |w|^{-d-\alpha}\, \dd w.
	\end{equation*} 
	For futher details we refer to~\cite{applebaum2009levy,ken1999levy}.
\end{remark} 

\paragraph{Boltzmann description.}
To describe the evolution of the particle density we rely on a kinetic approach, where $f(x,t)$ is governed by the Boltzmann-type equation
\begin{align}\label{eq:Boltz}
	\partial_t f(x,t) = \eta_1 \mathcal{Q}_1(f)(x,t) 
	+ \eta_2 \mathcal{Q}_2(f)(x,t),
\end{align}
with $\eta_i>0$, $i=1,2$, denote the interaction frequencies associated with \eqref{eq:KBO_alfa_stable}, namely the drift and stochastic exploration mechanisms, respectively.
In weak form, for a given test function $\psi \in C_c^\infty(\mathbb{R}^d)$, the Boltzmann equation 
\eqref{eq:Boltz} reads
\begin{align}\label{eq:weak_Boltz_splitting}
	&\int_{\mathbb{R}^d} \partial_t f(x,t)\,\psi(x)\, \dd x
	= \eta_1 \int_{\mathbb{R}^d} \big[\psi(x') - \psi(x)\big] f(x,t)\, \dd x \cr
	&\quad + \eta_2 \int_{\mathbb{R}^{3d}} \big[\psi(x'') - \psi(x')\big] f(x',t)\, g_2(z)\,g_\alpha(\tilde z)\, \dd x\, \dd z\, \dd \tilde z.
	%&\qquad + \eta_3 \int_{\mathbb{R}^{2d}} \big[\psi(x''') - \psi(x)\big] f(x,t)\, g_\alpha(\tilde z)\,dx\,d\tilde z.
\end{align}
Here, and later in this section, we denote by $g_\alpha$ the probability density $g^{\textrm{RS}}_\alpha(\cdot
)$ to simplify the notation. Such density encodes the distribution of the jump increments, and it is defined as the inverse Fourier transform of the characteristic function of a rotationally symmetric $\alpha$-stable distribution, namely,
\begin{equation*}
	g^{\textrm{RS}}_\alpha(z) = \mathcal{F}^{-1}_{\kappa \to z}\big(\tilde{\phi}_\alpha(\kappa)\big)(z),
	\qquad 
	\tilde{\phi}_\alpha(\kappa) = \exp\big(-|\kappa|^\alpha\big),
\end{equation*}
where $\alpha \in (0,2)$. Special cases include $\alpha=1$, which corresponds to the Cauchy distribution, and $\alpha=2$, which corresponds to Gaussian noise given under standard unit-variance normalization by $\tilde{\phi}_2(k) = \exp\left(-\vert k\vert^2\right)$. This should be distinguished from the standard Gaussian distribution
$\mathcal{N}(0,I_d)$, used for $z$ in \eqref{eq:KBO_alfa_stable}, whose characteristic function is
$\phi_2(k) = \exp(-|\kappa|^2/2)$.  For the sake of completeness, we recall that the Fourier transform of a sufficiently regular function $f$ is given by
\[
\hat{f}(\xi) = \mathcal{F}(f)(\xi) := \int_{\mathbb{R}^d} e^{-\imath \xi \cdot x} f(x)\, \dd x,
\qquad \xi \in \mathbb{R}^d,
\]
with inverse transform
\[
\mathcal{F}^{-1}(h)(x) := \frac{1}{(2\pi)^d} \int_{\mathbb{R}^d} e^{\imath \xi \cdot x} h(\xi)\, \dd \xi.
\]
For notational simplicity, in what follows, we omit the dot product symbol in the exponential when no ambiguity arises.

Hence,  following \cite{pareschi2013interacting}, in \eqref{eq:weak_Boltz_splitting} we select the test function $\psi(x) = e^{-\imath \xi x}$, which corresponds to the Fourier representation of the Boltzmann-type equation 
\eqref{eq:Boltz} as
\begin{equation}\label{eq:weak_Boltz_fourier}
	\partial_t \hat{f}(\xi,t) 
	= \eta_1 \hat{\mathcal{Q}}_1(f)(\xi,t) 
	+ \eta_2 \hat{\mathcal{Q}}_2(f)(\xi,t),
	%+ \eta_3 \hat{\mathcal{Q}}_3(f)(\xi,t),
\end{equation}
where each term $\hat{\mathcal{Q}}_i(f)$, $i=1,2$ denotes the Fourier transform of the corresponding interaction operator.

In the following section, we formally derive the fractional Fokker-Planck model describing the asymptotic behavior of the optimization process, restricting for simplicity to the isotropic diffusion matrix setting, which is also considered in the subsequent convergence analysis. 

\subsection{Formal derivation of fractional Fokker-Planck}\label{sec:FPE} 
We now investigate the asymptotic behavior of \eqref{eq:weak_Boltz_fourier} under a suitable interaction scaling, analyzing separately the contribution of each operator in order to derive a Fokker--Planck-type approximation. Hence, we introduce the following scaling 
\begin{equation}\label{eq:grazing_scaling}
	\nu \to \nu \varepsilon, \qquad
	\sigma \to \sigma \sqrt{2 \varepsilon}, \qquad
	\gamma \to \gamma \varepsilon^{1/\alpha}, \qquad
	\eta_i \to {1}/{\varepsilon},
\end{equation} 
\noindent which corresponds to a regime of frequent interactions with small increments, see for example
\cite{pareschi2013interacting} for further details. 
\paragraph{Scaled drift term.}
The first operator of \eqref{eq:weak_Boltz_splitting} write as follows
\begin{equation}\label{eq:operator_Q1_fourier}
	\hat{\mathcal{Q}}_1(f) 
	= \int_{\mathbb{R}^d} \big[e^{-\imath \xi x'} - e^{-\imath \xi x}\big] f(x,t)\, \dd x,
\end{equation}
with scaled interaction $x' = x + \varepsilon\nu(\bar{x}_\mathcal{E}(t)-x)$. Performing a Taylor expansion around $\varepsilon=0$, we obtain
\begin{equation}
	e^{-\imath \xi \nu( \bar{x}_\mathcal{E}(t) -x) \varepsilon}= 1 - \imath \xi \nu ( \bar{x}_\mathcal{E}(t) -x) \varepsilon + \mathcal{O}(\varepsilon^2),  %\frac{1}{2}\xi^2 \nu^2  \vert \bar{x}_\mathcal{E}(t) -x\vert^2 e^{-\imath \nu\xi (\bar{x}_\mathcal{E}(t)-x) \bar{\varepsilon}} \varepsilon^2,
\end{equation}
%with $\bar{\varepsilon} \in (0,\varepsilon)$. 
Considering the scaling \eqref{eq:grazing_scaling} for $\eta_1$  we have
\begin{equation}
	\eta_1 \hat{\mathcal{Q}}_1(f) = \frac{1}{\varepsilon} \left[ -\imath \xi \nu \varepsilon \int_{\RR^d} (\bar{x}_\mathcal{E}(t) -x)  e^{-\imath \xi x} f(x,t) dx  + \mathcal{O}(\varepsilon^2) \right].%R_1(\bar{\varepsilon}) \varepsilon^2 \right],   
\end{equation}
Retaining only the leading-order contribution, as $\varepsilon\to 0$ we formally have up to order $\mathcal{O}(\varepsilon)$
\begin{align}\label{eq:drift_fourier_final}
	\eta_1 \hat{\mathcal{Q}}_1(f)
	&= -\imath \xi \nu \int_{\mathbb{R}^d} (\bar{x}_\mathcal{E}(t) - x) e^{-\imath \xi x} f(x,t)\, \dd x\cr
	&\quad= -\nu\, \mathcal{F}\big(\nabla_x \cdot ((\bar{x}_\mathcal{E}(t)-x) f(x,t))\big).
\end{align}

\paragraph{Scaled diffusion term.}
Recalling \eqref{eq:weak_Boltz_splitting}, we write the second operator as
\begin{align}\label{eq:operator_Q2}
	\hat{\mathcal{Q}}_2(f)
	&= \int_{\mathbb{R}^{2d}} f(x',t) e^{-\imath \xi x'} 
	\int_{\mathbb{R}^d} \left(e^{-\imath \xi D(\bar{x}_\mathcal{E}(t),x') ( \sigma z + \gamma \tilde z)} - 1
	\right)g_2(z)g_\alpha(\tilde z) \,\dd z\, \dd \tilde z \, \dd x',\cr
	%\\&\quad - \int_{\mathbb{R}^d} f(x,t) e^{-\imath \xi x}\,dx \int_{\mathbb{R}^d} g_2(z)\,dz.
	&\quad= \int_{\mathbb{R}^d} e^{-\imath \xi x'} f(x',t)\, 
	\phi_2\big(\sigma D(\bar{x}_\mathcal{E}(t),x')\xi\big)\, \tilde{\phi}_\alpha\big(\gamma D(\bar{x}_\mathcal{E}(t),x')\xi\big) \dd x'
	- \hat{f}(\xi,t),
\end{align}
where we used that $\int_{\mathbb{R}^d} g_2(z)\, \dd z = 1$, $\int_{\mathbb{R}^d} g_\alpha(z)\, \dd z = 1$ and the Fourier representation of $g_2(\cdot)$ as
$\phi_2(\xi)=e^{-|\xi|^2/2}$  and of $g_\alpha(z)$ as $\tilde{\phi}_\alpha(\xi) = e^{-|\xi|^\alpha}$. Therefore, according to \eqref{eq:grazing_scaling} and using Taylor around $\varepsilon=0$, we have

\begin{equation}
	\begin{split}
		&\phi_2(\sigma \sqrt{2 \varepsilon} D(\bar{x}_\mathcal{E}(t),x') \xi)
		=1 -\sigma^2 D^2(\bar{x}_\mathcal{E}(t),x') \xi^2 \varepsilon + \mathcal{O}(\varepsilon^2),\\%\frac{1}{2} \sigma^4 \xi^4 D^4(\bar{x}_\mathcal{E}(t),x) e^{-\sigma^2 D^2(\bar{x}_\mathcal{E}(t),x) \xi^2 \bar{\varepsilon}_1} \varepsilon^2
		&\tilde{\phi}_\alpha(\gamma \varepsilon^{1/\alpha} D(\bar{x}_\mathcal{E}(t),x') \xi) 
		%&= e^{-\gamma^\alpha \vert D(\bar{x}_\mathcal{E}(t),x)\vert^\alpha \vert\xi\vert^\alpha \varepsilon} \cr
		= 1 -\gamma^\alpha \vert D (\bar{x}_\mathcal{E}(t),x') \vert^\alpha \vert \xi\vert^\alpha \varepsilon + \mathcal{O}(\varepsilon^2),
	\end{split}
\end{equation}
where we recall that the isotropic setting is assumed, allowing the factorization $\vert D(\bar{x}_\mathcal{E}(t),x')^T \xi\vert^\alpha = \vert D(\bar{x}_\mathcal{E}(t),x')\vert ^\alpha \vert \xi\vert^\alpha$.
Therefore, scaling $\eta_2\rightarrow 1/\varepsilon$ we obtain,
\begin{align*}
	\eta_2 \hat{\mathcal{Q}}_2(f) = &\frac{1}{\varepsilon} \Big[-\sigma^2 \xi^2 \varepsilon \int_{\RR^d} D^2(\bar{x}_\mathcal{E}(t),x') e^{-\imath \xi x'} f(x',t) \dd x'  \Big. \\
	& \Big. - \gamma^\alpha \vert \xi\vert^\alpha \varepsilon \int_{\RR^d} \vert D(\bar{x}_\mathcal{E}(t),x')\vert^\alpha e^{-\imath \xi x'} f(x',t) \dd x' + \mathcal{O}(\varepsilon^2) \Big]% \varepsilon^2 R_2(\bar{\varepsilon}_1) \Big],
\end{align*}
Recalling that $x' = x + \mathcal{O}(\varepsilon)$, as $\varepsilon\to 0$, we formally have up to order $\mathcal{O}(\varepsilon)$
\begin{equation}\label{eq:diffusion_fourier_final}
	\begin{split}
		\eta_2 \hat{\mathcal{Q}}_2(f)
		&= -\sigma^2 \xi^2 
		\int_{\mathbb{R}^d} D^2(\bar{x}_\mathcal{E}(t),x) e^{-\imath \xi x} f(x,t)\, \dd x\\ & \qquad - \gamma^\alpha \vert\xi\vert^\alpha 
		\int_{\mathbb{R}^d} \vert D(\bar{x}_\mathcal{E}(t),x)\vert^\alpha e^{-\imath \xi x} f(x,t)\, \dd x\cr 
		&=\sigma^2 \mathcal{F}\big(\Delta(D^2(\bar{x}_\mathcal{E}(t),x) f(x,t))\big) - \gamma^\alpha \vert \xi \vert^\alpha  \mathcal{F}\big(\vert D(\bar{x}_\mathcal{E}(t),x)\vert^\alpha f(x,t))\big).
	\end{split}
\end{equation}

Collecting the leading-order contributions in \eqref{eq:drift_fourier_final}, and \eqref{eq:diffusion_fourier_final}, we formally obtain the  following fractional Fokker--Planck equation
\begin{equation}\label{eq:fFP}
	\begin{aligned}
		\partial_t f(x,t) = 
		& - \nabla_x \cdot \big(\nu(\bar{x}_\mathcal{E}(t)-x) f(x,t)\big) + \sigma^2 \Delta\big(D^2(\bar{x}_\mathcal{E}(t),x) f(x,t)\big) \\
		& 
		\qquad - \gamma^{\alpha} (-\Delta)^{\alpha/2}\big(|D(\bar{x}_\mathcal{E}(t),x)|^\alpha f(x,t)\big),
	\end{aligned}
\end{equation}
where the fractional Laplacian encodes the nonlocal behavior induced by $\alpha$-stable jumps, and is defined as follows
\[
(-\Delta)^{\alpha/2} g(x) = \mathcal{F}^{-1}\big(|\xi|^\alpha \mathcal{F}(g)(\xi)\big).
\]

In Subsection \eqref{subsec:prelim} we give another equivalent representation of the fractional Laplacian.

\section{Convergence to the global minimizer}\label{sec:convergence}
The convergence analysis of consensus-type global optimizers typically proceeds in three steps: proving that a suitable measure of dispersion (usually the variance) decays exponentially in time, showing that the mean converges to a limit point, and finally arguing that this limit is a global minimizer of the objective \cite{PTTM2017,benfenati2022binary}. For the kinetic model driven by a symmetric $\alpha$-stable process, this strategy remains natural but requires two key adaptations. First, because the second moment of $f$ is infinite for $\alpha < 2$, variance estimates are replaced by bounds on the $p$-th moment ($1 < p < \alpha$) or, equivalently, on the Wasserstein $p$-distance between $f(\cdot, t)$ and the Dirac mass at the minimizer \cite{fornasier2024consensus}. Second, the fractional Laplacian of the natural test function $\vert{}x - x^\star\vert{}^p$ is singular at $x^\star$. When the consensus point $\xE(t)$ differs from $x^\star$, this singularity is not compensated by the diffusion coefficient $\vert{}\xE(t) - x\vert{}^\alpha$. This produces a moment of negative order that diverges precisely for measures concentrated near $x^\star$—the very measures the dynamics are expected to approach. We resolve this by working with a \emph{regularized} moment, obtained by mollifying the test function at a scale $\eta$ tied to the prescribed accuracy (Section~\ref{subsec:regmoment}). The resulting regularization error is explicit and is absorbed in the same manner as the drift cross term. This approach avoids requiring any additional regularity properties of $f$ beyond $\int_{\R^d} f \dd x = 1$ and the finiteness of the $p$-th moment. See Remark~\ref{rem:comparison} for a comparison with alternative strategies.

Recall that for $f_1,f_2\in\Pcal(\R^d)$, the space of probability measures on $\R^d$, and $p\in[1,\infty)$ the Wasserstein
p-distance is
\[
W_p(f_1,f_2) = \inf_{\Pi\in\Gamma( f_1, f_2)} \Big(\int_{\R^d\times\R^d} |x-y|^p\dd\Pi( x, y)\Big)^{1/p},
\]
where $\Gamma( f_1, f_2)$ is the set of transport plans, i.e.\ probability
measures on $\R^d\times\R^d$ with marginals $f_1$ and $f_2$ \cite{ambrosio2008gradient}.
When $f_2 = \delta_{x^\star}$ the only transport plan is $f_1\otimes\delta_{x^\star}$,
so that
\begin{equation}\label{eq:Vp}
	W_p^p\big(f(\cdot,t),\delta_{x^\star}\big) = \int_{\R^d}|x-x^\star|^p f(x,t) \dd x =: V_p(t).
\end{equation}
The goal is to prove that $V_p(t)$ decays exponentially fast in a certain time window, yielding convergence up to the prescribed accuracy $\varepsilon$ in $W_p$ to the Dirac mass at the global minimizer.

Throughout this section we consider \eqref{eq:fFP} with $\sigma=0$ and
isotropic diffusion, $D(\xE,x)=|\xE-x|\,\Id_d$, that is
\begin{equation}\label{eq:FFP0}
	\partial_t f=-\nabla_x\cdot\big(\nu(\xE(t)-x)f\big)-\gamma^\alpha(-\Delta)^{\alpha/2}\big(|\xE(t)-x|^\alpha f\big),
\end{equation}
with $\xE( t)$ given by \eqref{eq:x_tot}, and we fix
\[
1< p < \alpha <2,\qquad s := \alpha-p\in( 0, 1),\qquad w := x - x^\star, \qquad \delta( t) := |\xE( t)-x^\star|.
\]
We write $B_r( x)$ for the open ball of radius $r$ centred at $x \in \R^d$, and $\omega_d := |B_1 ( 0)|$ for the volume of the unit ball in $\R^d$ with centre at the origin.

\subsection{Preliminaries}\label{subsec:prelim}

\begin{assumption}\label{ass:E}
	The objective $\Ecal\in C(\R^d)$ satisfies:
	\begin{itemize}
		\item[(A1)] there exists $x^\star\in\R^d$ such that
		$\Ecal(x^\star)=\inf_{x\in\R^d}\Ecal(x)=:\underline{\Ecal}$;
		\item[(A2)] there exist $\Ecal_\infty,R_0,\zeta>0$ and $\mu\in(0,\infty)$ such that
		\begin{equation}\label{eq:assE}
			\begin{split}
				& |x-x^\star|\le\frac{(\Ecal(x)-\underline{\Ecal})^{\mu}}{\zeta}\quad\text{for all }x\in B_{R_0}(x^\star), \\
				\qquad
				& \Ecal(x)-\underline{\Ecal}>\Ecal_\infty\quad\text{for all }x\notin B_{R_0}(x^\star).
			\end{split}
		\end{equation}
	\end{itemize}
\end{assumption}
\noindent Since $\xE$ is unchanged if a constant is added to $\Ecal$, we shall
assume that $\underline{\Ecal} = 0$.

\paragraph{The fractional Laplacian.}
For $\varphi\in C^2(\R^d)$ with $|\varphi(x)|\le C(1+|x|)^{q}$ for some $q<\alpha$
we use the singular-integral representation
\begin{equation}\label{eq:fraclap}
	(-\Delta)^{\alpha/2}\varphi(x):=\frac{c_{d,\alpha}}{2}\int_{\R^d}\frac{2\varphi(x)-\varphi(x+h)-\varphi(x-h)}{|h|^{d+\alpha}}\dd h,
\end{equation}
where
\[
c_{d,\alpha} = \frac{2^{\alpha} \, \Gamma(\frac{d+\alpha}{2})}{\pi^{d/2} \, |\Gamma(-\frac{\alpha}{2})|},
\]
which converges absolutely. With this normalization
\eqref{eq:fraclap} coincides with the Fourier definition used in
Section~2 on smooth functions of growth $q<\alpha$
\cite{kwasnicki2017ten,DNPV2012}.

\paragraph{Weak solutions.}
Let
\begin{equation}\label{eq:classCp}
	\begin{split}
		& \Ccal_p := \Big\{\varphi\in C^2(\R^d) : \ |\nabla\varphi(x)|\le C(1+|x|)^{p-1}, \\
		& \hspace{4cm} \big|(-\Delta)^{\alpha/2} \varphi( x) \big |\le C(1+|x|)^{-s}\ \text{ for some } C>0 \Big\}.
	\end{split}
\end{equation}
We use the standard notion of measure-valued solution, tested against
compactly supported functions. The extension to the unbounded test functions
in $\Ccal_p$ needed below is a consequence of the moment bound
(Lemma~\ref{lem:extension}), not an assumption. 

\begin{definition}\label{def:weak}
	Let $T>0$ and $f_0\in\Pcal(\R^d)$. A map $f\in C([0,T];\Pcal(\R^d))$ with
	$f(\cdot,0)=f_0$ is a \emph{weak solution} of \eqref{eq:FFP0} on $[0,T]$ if
	$\sup_{t\in[0,T]}V_p(t)<\infty$ and, for every $\varphi\in C^2_c(\R^d)$ and every
	$t\in[0,T]$,
	\begin{equation}\label{eq:weakc}
		\begin{aligned}
			\int_{\R^d}\varphi (x) f( x, t) \dd x-\int_{\R^d}\varphi ( x) f_0 ( x) \dd x
			= & \int_0^t \Big[ \nu \! \int_{\R^d} \! \langle \nabla \varphi ( x), \xE( \tau) - x \rangle f( x, \tau)\dd x\\
			& \hspace{-1.5cm} -\gamma^\alpha \! \int_{\R^d} \! \big(( -\Delta)^{\alpha/2} \varphi (x) \big) \, | \xE( \tau) - x|^\alpha f( x, \tau)\dd x \Big] \dd \tau .
		\end{aligned}
	\end{equation}
\end{definition}

Notice that no separate growth condition on $\varphi$ itself is needed in \eqref{eq:classCp} because
the bound on the gradient implies
\begin{equation}\label{eq:growthCp}
	|\varphi(x)|\le|\varphi(0)|+\int_0^1|\nabla\varphi(tx)|\,|x|\dd t
	\le|\varphi(0)|+\frac{C}{p}\big[ (1+|x|)^p-1 \big]\le C'(1+|x|)^p \, ,
\end{equation}
for all  $x \in \R^d$, with $C'$ depending only on $C$, $p$ and $|\varphi(0)|$. Thus every $\varphi\in\Ccal_p$
grows at most like $|x|^p$, which is exactly the order at which the standing
assumption $\sup_{t \in [ 0, T]} V_p(t)<\infty$ provides integrability against $f(\cdot,t)$, so
that $\int\varphi f\dd x$ is finite. Since moreover $p<\alpha$, the growth
\eqref{eq:growthCp} also ensures that the singular integral \eqref{eq:fraclap}
defining $(-\Delta)^{\alpha/2}\varphi$ converges absolutely for every
$\varphi\in C^2(\R^d)$ satisfying the gradient bound, so that the second
condition in \eqref{eq:classCp} is meaningful. In this sense the two
conditions in \eqref{eq:classCp} express a single requirement, namely that
$\varphi$ behaves like a function of order $p$ under both operations appearing
in the fractional Fokker-Planck operator, i.e. right-hand-side of \eqref{eq:FFP0}. In other words, differentiation lowers the order by one, and the diffusion
coefficient $|\xE-x|$ raises it by one, whereas the fractional Laplacian lowers the
order by $\alpha$, and the coefficient $|\xE-x|^\alpha$ raises it by $\alpha$.
Consequently $|\nabla\varphi|\,|\xE-x|$ and $|(-\Delta)^{\alpha/2}\varphi|\,|\xE-x|^\alpha$
are both of order $p$, and both integrals on the right-hand side of
\eqref{eq:weakc} below are finite whenever $V_p<\infty$.

Since $\Ecal$ is continuous and $p>1$, weak continuity of $t\mapsto f(\cdot,t)$
together with the uniform bound on $V_p$ implies that $t\mapsto\xE(t)$ is
continuous. In particular $\bar B:=\sup_{[0,T]}|\xE(t)|<\infty$ and the
right-hand side of \eqref{eq:weakc} is well defined for $\varphi\in C^2_c$,
both integrands being bounded. Only $f\ge0$ and $\int_{\R^d} f \dd x = 1$ are used below. In addition, 
$f(\cdot,t)$ is identified with its density when convenient.

\begin{lemma}[Extension of the weak formulation to $\Ccal_p$]\label{lem:extension}
	Let $f$ be a weak solution in the sense of Definition~\ref{def:weak}. Then
	\eqref{eq:weakc} holds for every $\varphi\in\Ccal_p$, and
	$t\mapsto\int\varphi f(\cdot,t)\dd x$ is Lipschitz on $[0,T]$. In particular, for
	every $\varphi\in\Ccal_p$ and a.e.\ $t\in[0,T]$,
	\begin{equation}\label{eq:weak}
		\frac{\dd}{\dd t}\int_{\R^d}\varphi\, f\dd x
		=\nu\int_{\R^d}\big\langle\nabla\varphi,\xE(t)-x\big\rangle f\dd x
		-\gamma^\alpha\int_{\R^d}\big((-\Delta)^{\alpha/2}\varphi\big)\,|\xE(t)-x|^\alpha f\dd x.
	\end{equation}
\end{lemma}
The proof of this result can be found in the Appendix. Note that $|x-x^\star|^p\notin\Ccal_p$, since its fractional Laplacian is
unbounded at $x^\star$ (Lemma~\ref{lem:power}), this is why the test function
is regularized in Section~\ref{subsec:regmoment}. For $\varphi\in\Ccal_p$ both
integrals in \eqref{eq:weak} are finite because
$|\nabla\varphi|\,|\xE-x|\le C(1+|x|)^p$ and
$|(-\Delta)^{\alpha/2}\varphi| \, |\xE - x|^\alpha \le C(1+|x|)^{-s+\alpha} = C(1+|x|)^{p}$.

\begin{remark}[On existence and uniqueness]\label{rem:existence}
	Definition~\ref{def:weak} is the standard notion of measure-valued solution,
	with the moment requirement placed at the natural order $p < \alpha$. The driving
	noise has finite moments exactly of order less than $\alpha$, so no higher
	moment can be asked for. In the Gaussian setting the corresponding convergence
	results require fourth moments \cite{fornasier2024consensus}. On the other hand, here the decay
	$|(-\Delta)^{\alpha/2}\varphi|\le C(1+|x|)^{-s}$ of admissible test functions
	compensates the growth of the diffusion coefficient $|\xE-x|^\alpha$, and only
	$V_p$ enters. The existence and uniqueness of weak solutions to \eqref{eq:FFP0} is not
	addressed in this paper. 
\end{remark}

\begin{lemma}\label{lem:power}
	Let $0 < p < \alpha <2$ and set
	\begin{equation}\label{eq:Bpa}
		B_{ p, \alpha} := \frac{2^{\alpha} \, \Gamma \! \big(\frac{d+p}{2}\big) \, \Gamma \! \big(\frac{\alpha - p}{2} \big)}
		{\big|\Gamma \! \big(-\frac{p}{2}\big)\big|\;\Gamma\!\big(\frac{d+p-\alpha}{2}\big)}.
	\end{equation}
	Then $B_{p,\alpha}>0$ and, for every $x\neq x^\star$, the integral
	\eqref{eq:fraclap} with $\varphi (x) = | x - x^\star|^p$ converges absolutely and
	\begin{equation}\label{eq:powerid}
		(-\Delta)^{\alpha/2}\big(|x-x^\star|^{p}\big)=-B_{p,\alpha}\,|x-x^\star|^{\,p-\alpha}.
	\end{equation}
	Moreover there is $C_\ast=C_\ast(d,p,\alpha)$ such that, with $g:=|\cdot|^p$,
	\begin{equation}\label{eq:absconv}
		\frac{c_{d,\alpha}}{2}\int_{\R^d}\frac{|2g(w)-g(w+h)-g(w-h)|}{|h|^{d+\alpha}}\dd h\le C_\ast\,|w|^{-s} \, ,
		\qquad\text{for all }w\neq0 .
	\end{equation}
\end{lemma}

\begin{proof}
	For $\lambda>-d$, $\lambda\notin2\mathbb N_0$, the Fourier transform of the
	tempered distribution $|x|^\lambda$ is the homogeneous distribution
	$\Fcal(|\cdot|^\lambda)(\xi)=D_\lambda|\xi|^{-d-\lambda}$ with
	$D_\lambda=2^{d+\lambda}\pi^{d/2}\,\Gamma(\frac{d+\lambda}{2})/\Gamma(-\frac{\lambda}{2})$
	\cite{stein1970singular,landkof1972foundations}. Both $\lambda=p$ and $\lambda=p-\alpha$ are admissible
	($p-\alpha>-1\ge-d$, and neither is an even nonnegative integer), hence
	\[
	\Fcal\big((-\Delta)^{\alpha/2}|\cdot|^p\big)(\xi)=|\xi|^\alpha D_p|\xi|^{-d-p}
	=\frac{D_p}{D_{p-\alpha}}\,\Fcal\big(|\cdot|^{p-\alpha}\big)(\xi) \, ,
	\]
	where
	\[
	\frac{D_p}{D_{p-\alpha}}=2^{\alpha}\frac{\Gamma(\frac{d+p}{2})\,\Gamma(\frac{\alpha-p}{2})}
	{\Gamma(-\frac p2)\,\Gamma(\frac{d+p-\alpha}{2})}.
	\]
	Since $-\frac p2\in(-1,0)$ we have $\Gamma(-\frac p2)<0$, while the remaining
	Gamma factors have positive arguments; therefore $D_p/D_{p-\alpha}=-B_{p,\alpha}$
	with $B_{p,\alpha}>0$, and \eqref{eq:powerid} holds in $\Scal'(\R^d)$. Both sides
	of \eqref{eq:powerid} are continuous on $\R^d\setminus\{x^\star\}$, and on such
	a set the distributional and the pointwise definition agree
	\cite{kwasnicki2017ten}, which gives \eqref{eq:powerid} pointwise.
	
	For \eqref{eq:absconv} fix $w\ne0$ and write $\Delta_hg(w):=2g(w)-g(w+h)-g(w-h)$.
	We split the integral at $|h|=|w|/2$.
	
	\emph{Near field, $|h|<|w|/2$.} For $v\neq0$ one has
	$D^2g(v)=p|v|^{p-2}\big(\Id_d+(p-2)\,\tfrac{v}{|v|}\otimes\tfrac{v}{|v|}\big)$, hence
	$\|D^2g(v)\|\le p(3-p)|v|^{p-2}$. If $|v-w|\le|h|<|w|/2$ then $|v|>|w|/2$, and since
	$p-2<0$ we get $\|D^2g(v)\|\le p(3-p)(|w|/2)^{p-2}$ on the segment $[w-h,w+h]$.
	Taylor's formula with integral remainder,
	$\Delta_hg(w)=-\int_0^1(1-\theta)\big[\langle D^2g(w+\theta h)h,h\rangle+\langle D^2g(w-\theta h)h,h\rangle\big]\dd\theta$,
	therefore gives $|\Delta_hg(w)|\le p(3-p)(|w|/2)^{p-2}|h|^2$, and in polar
	coordinates
	\begin{align*}
		\int_{\{|h|<|w|/2\}}\frac{|\Delta_hg(w)|}{|h|^{d+\alpha}}\dd h
		&\le p(3-p)\Big(\frac{|w|}{2}\Big)^{p-2}\,d\omega_d\int_0^{|w|/2}r^{1-\alpha}\dd r\\
		&=\frac{p(3-p)\,d\omega_d}{2-\alpha}\Big(\frac{|w|}{2}\Big)^{p-\alpha}
		=\frac{p(3-p)\,d\omega_d\,2^{s}}{2-\alpha}\,|w|^{-s}.
	\end{align*}
	
	\emph{Far field, $|h|\ge|w|/2$.} Here $|w|\le2|h|$, so that
	$|\Delta_hg(w)|\le2|w|^p+|w+h|^p+|w-h|^p\le2|w|^p+2(|w|+|h|)^p\le(2^{p+1}+2\cdot3^p)\,|h|^p$,
	and in polar coordinates, using $p-\alpha=-s<0$ for the convergence at infinity,
	\begin{align*}
		\int_{\{|h|\ge|w|/2\}}\frac{|\Delta_hg(w)|}{|h|^{d+\alpha}}\dd h
		&\le(2^{p+1}+2\cdot3^p)\,d\omega_d\int_{|w|/2}^{\infty}r^{-s-1}\dd r\\
		&=\frac{(2^{p+1}+2\cdot3^p)\,d\omega_d}{s}\Big(\frac{|w|}{2}\Big)^{-s}
		=\frac{(2^{p+1}+2\cdot3^p)\,d\omega_d\,2^{s}}{s}\,|w|^{-s}.
	\end{align*}
	
	Adding the two contributions and multiplying by $c_{d,\alpha}/2$ proves
	\eqref{eq:absconv} with
	\[
	C_\ast=\frac{c_{d,\alpha}}{2}\,d\omega_d\,2^{s}\Big[\frac{p(3-p)}{2-\alpha}+\frac{2^{p+1}+2\cdot3^p}{s}\Big].
	\]
\end{proof}

\begin{remark}
	As a consistency check, for $\alpha=2$ one has
	$\Gamma(\frac{d+p}{2}) = \frac{d+p-2}{2}\Gamma(\frac{d+p-2}{2})$ and
	$\Gamma(1-\frac p2)=\frac p2|\Gamma(-\frac p2)|$, so \eqref{eq:Bpa} gives
	$B_{p,2}=p(d+p-2)$, in agreement with $\Delta |x|^p = p(p+d-2)|x|^{p-2}$. Note also
	that $B_{p,\alpha} \to \infty$ as $p \nearrow \alpha$, in agreement with the fact
	that the $\alpha$-moment of an $\alpha$-stable law is infinite.
\end{remark}

We finally conclude this subsection with two auxiliary inequalities.

\begin{lemma}\label{lem:elem}
	\begin{itemize}
		\item[(i)] Let $q>1$. For all $a,b \ge 0$ and $\tau > 0$, it holds
		\begin{equation}\label{eq:E}
			(a+b)^q\le(1+\tau)\,a^q+\thetaq{q}{\tau}\,b^q,
			\quad \text{where }
			\thetaq{q}{\tau} := \Big(1-(1+\tau)^{-\frac{1}{q-1}}\Big)^{1-q}.
		\end{equation}
		\item[(ii)] If $g \ge 0$ is a probability density and $\Psi \ge 0$ is measurable,
		then
		\begin{equation}\label{eq:holder}
			\int_{\R^d} \Psi^{\frac{p-1}{p}} \, g \dd x \le \Big( \int_{\R^d} \Psi \, g \dd x \Big)^{\frac{p-1}{p}} .
		\end{equation}
	\end{itemize}
\end{lemma}

\begin{proof}
	(i) For $\vartheta\in(0,1)$, convexity of $t\mapsto t^q$ on $[0,\infty)$ yields
	$(a+b)^q=\big((1-\vartheta)\frac{a}{1-\vartheta}+\vartheta\frac{b}{\vartheta}\big)^q
	\le(1-\vartheta)^{1-q}a^q+\vartheta^{1-q}b^q$. We conclude by choosing $\vartheta$ such that
	$(1-\vartheta)^{1-q}=1+\tau$. 
	
	(ii) Applying H\"older's inequality with the conjugate exponents
	$\frac{p}{p-1}$ and $p$, yields
	\[
	\int_{\R^d} \Psi^{\frac{p-1}{p}}g \dd x = \int_{\R^d} (\Psi g)^{\frac{p-1}{p}}g^{\frac1p} \dd x 
	\le \bigg( \int_{\R^d} \Psi g \dd x \bigg) ^{\frac{p-1}{p}} \bigg( \int_{\R^d} g \dd x \bigg)^{\frac1p}.
	\]
	Hence, we conclude the  proof of inequality \eqref{eq:holder}.
\end{proof}

%==============================================================================
\subsection{The regularized moment}\label{subsec:regmoment}
%==============================================================================

Fix a mollifier $\rho\in C^\infty_c(\R^d)$ with
\begin{equation}\label{eq:mollifier}
	\rho \ge 0, \qquad\int_{\R^d}\rho\dd z = 1,\qquad \rho( -z) = \rho( z),\qquad\supp\rho\subset\overline{B_1(0)},
\end{equation}
and set $\rho_\eta(y):=\eta^{-d}\rho(y/\eta)$ for $\eta>0$, so that
$\supp\rho_\eta\subset\overline{B_\eta(0)}$, $\int_{\R^d} \rho_\eta \dd y = 1$, and
$\int y\,\rho_\eta(y)\dd y=0$. Define
\begin{equation}\label{eq:psi}
	\begin{split}
		& \psi_\eta( x) := \int_{\R^d}|x - x^\star - y|^{p} \rho_\eta( y) \dd y, \quad
		k_\eta( x) := \int_{\R^d}| x - x^\star - y|^{-s} \rho_\eta( y) \dd y, \\
		& \Vpe( t) := \int_{\R^d} \psi_\eta(x) f(x,t) \dd x,
	\end{split}
\end{equation}
that is, $\psi_\eta=|\cdot|^p*\rho_\eta$ and $k_\eta=|\cdot|^{-s}*\rho_\eta$
evaluated at $w=x-x^\star$. Finally let
\begin{equation}\label{eq:Crho}
	C_\rho:=\|\rho\|_{\infty}\,\frac{d\,\omega_d}{d-s},
\end{equation}
which is finite because $s<1\le d$.

In the following result, we state some regularity properties of the function $\psi_\eta$.

\begin{lemma}\label{lem:regularity}
	For $p>1$ the function $|\cdot|^p$ belongs to $C^1(\R^d)$, with gradient
	$p|v|^{p-2}v$ for $v\ne0$ and $0$ at $v=0$. Moreover
	$\psi_\eta\in C^\infty(\R^d)$ and
	\begin{equation}\label{eq:gradpsi}
		\nabla \psi_\eta(x) = p \int_{\R^d}|w-y|^{p-2}(w-y) \, \rho_\eta( y) \dd y .
	\end{equation}
\end{lemma}

\begin{proof}
	For $v\ne 0$ the gradient is $p | v|^{p-2}v$, of modulus $p|v|^{p-1}\to0$ as
	$v\to0$. Now at $v=0$, $(|te|^p-0)/t=t^{p-1}|e|^p\to0$ for every unit vector $e$. Therefore,
	$|\cdot|^p$ is differentiable at $0$ with zero gradient, and the gradient is
	continuous. Writing $\psi_\eta(x)=\int|z|^p\rho_\eta(w-z)\dd z$, all derivatives
	may be placed on $\rho_\eta\in C^\infty_c$ (the integrand is dominated locally
	uniformly by $\|\partial^k\rho_\eta\|_\infty|z|^p\mathbf 1_{\{|z-w_0|\le\eta+1\}}$),
	whence $\psi_\eta\in C^\infty$. For \eqref{eq:gradpsi} differentiate under the
	integral in \eqref{eq:psi}. Hence, the integrand is $C^1$ in $w$ and its gradient is
	bounded by $p(|w|+\eta)^{p-1}$ on $\supp\rho_\eta$, locally uniformly in $w$.
\end{proof}

\begin{lemma}\label{lem:comparison}
	For all $x\in\R^d$ and $\tau>0$,
	\begin{equation}\label{eq:comparison}
		|w|^p\ \le\ \psi_\eta(x)\ \le\ (1+\tau)\,|w|^p+\thetaq{p}{\tau}\,\eta^p .
	\end{equation}
	Consequently, for every probability density $f$,
	\begin{equation}\label{eq:Vcomparison}
		V_p\ \le\ \Vpe\ \le\ (1+\tau)\,V_p+\thetaq{p}{\tau}\,\eta^p .
	\end{equation}
\end{lemma}

\begin{proof} 
	\emph{Lower bound.} $\Phi:=|\cdot|^p$ is convex and $C^1$
	(Lemma~\ref{lem:regularity}), so its graph lies above every tangent plane,
	$\Phi(b)\ge\Phi(a)+\langle\nabla\Phi(a),b-a\rangle$. With $a=w$, $b=w-y$ this
	reads $|w-y|^p\ge|w|^p-\langle p|w|^{p-2}w,\,y\rangle$ (for $w=0$ it reads
	$|y|^p\ge0$). Multiply by $\rho_\eta(y)\ge0$ and integrate: the left side gives
	$\psi_\eta(x)$, the first term on the right gives $|w|^p$, and the last term
	vanishes because $y\rho_\eta(y)$ is odd. 
	
	\emph{Upper bound.} On
	$\supp \rho_\eta$ we have $| y| \le \eta$, hence $|w - y| \le |w|+\eta$ and
	$\psi_\eta( x) \le (| w| + \eta)^p$. Now apply \eqref{eq:E} with $q = p$. 
	
	Finally, the inequalities
	\eqref{eq:Vcomparison} follows by integrating \eqref{eq:comparison} against $f$.
\end{proof}

The following result gathers several useful properties of $\psi_\eta$ and $k_\eta$. In particular, we show that the fractional Laplacian of order $\alpha$ of $\psi_\eta$ is proportional to $k_\eta$.

\begin{lemma}\label{lem:fraclappsi}
	The following properties hold:
	\begin{itemize}
		\item[(i)] For all $x \in \mathbb{R}^d$, $(-\Delta)^{\alpha/2}\psi_\eta(x) = -B_{p,\alpha} k_\eta(x)$, where $k_\eta(x) > 0$.
		\item[(ii)] For all $x \in \mathbb{R}^d$, we have $k_\eta(x) \le C_\rho \eta^{-s}$, where $C_\rho$ is defined in \eqref{eq:Crho}.
		\item[(iii)] For all $R > 1$ and $x \in \mathbb{R}^d$ such that $|x - x^\star| \ge R\eta$, we have 
		\[
		k_\eta(x)\le\bigg( 1 - \frac1R \bigg)^{-s}| x - x^\star|^{-s} \, .
		\]
		\item[(iv)] The function $\psi_\eta$ belongs to $\mathcal{C}_p$.
	\end{itemize}
	
\end{lemma}

\begin{proof} We will start with the proof of item (ii) since item (i) is based on it.
	
	(ii) Substituting $y=\eta z$ and writing $u:=w/\eta$,
	\begin{align*}
		k_\eta(x) & = \eta^{-s} \int_{\R^d} |u-z|^{-s}\rho(z)\dd z \\
		& \le\eta^{-s} \|\rho\|_\infty \int_{B_1(u)}|v|^{-s}\dd v \\
		& \le\eta^{-s} \|\rho\|_\infty \int_{B_1(0)}|v|^{-s}\dd v \\
		& = \eta^{-s} \|\rho\|_\infty \frac{d\,\omega_d}{d-s},
	\end{align*}
	where the first inequality uses $0\le\rho\le\|\rho\|_\infty$ and
	$\supp\rho\subset\overline{B_1(0)}$ followed by the substitution $v=u-z$, and the
	last integral is finite since $s<d$. For the second inequality, note that
	$B_1(u)$ and $B_1(0)$ have the same volume, so
	$|B_1(u)\setminus B_1(0)|=|B_1(0)\setminus B_1(u)|$, while $|v|^{-s}\le1$ on the
	former set and $|v|^{-s}>1$ on the latter. Therefore,
	\[
	\int_{B_1(u)\setminus B_1(0)}|v|^{-s}\dd v\le|B_1(u)\setminus B_1(0)|
	=|B_1(0)\setminus B_1(u)|\le\int_{B_1(0)\setminus B_1(u)}|v|^{-s}\dd v \, ,
	\]
	and adding $\int_{B_1(u)\cap B_1(0)}|v|^{-s}\dd v$ to both sides gives
	$\int_{B_1(u)}|v|^{-s}\dd v\le\int_{B_1(0)}|v|^{-s}\dd v$. 
	
	(i) Let $g:=|\cdot|^p$ and $\Delta_hg(v):=2g(v)-g(v+h)-g(v-h)$. Since $\Delta_h$
	is linear, $\Delta_h\psi_\eta(w)=\int\Delta_hg(w-y)\rho_\eta(y)\dd y$. By
	\eqref{eq:absconv} and (ii),
	\[
	\int\rho_\eta(y)\int\frac{|\Delta_hg(w-y)|}{|h|^{d+\alpha}}\dd h\dd y
	\le\frac{2C_\ast}{c_{d,\alpha}}\int\rho_\eta(y)|w-y|^{-s}\dd y
	=\frac{2C_\ast}{c_{d,\alpha}}\,k_\eta(x)<\infty
	\]
	(the point $y=w$ is a null set). Hence Fubini's theorem applies and, by
	\eqref{eq:powerid},
	\begin{align*}
		(-\Delta)^{\alpha/2}\psi_\eta(x) 
		& =\int\rho_\eta(y)\Big[\frac{c_{d,\alpha}}{2}\int\frac{\Delta_hg(w-y)}{|h|^{d+\alpha}}\dd h\Big]\dd y \\
		& =-B_{p,\alpha}\int\rho_\eta(y)|w-y|^{-s}\dd y \\
		& =-B_{p,\alpha}k_\eta(x).
	\end{align*}
	Positivity of $k_\eta$ is obvious.
	
	\medskip
	(iii) If $|w| \ge R\eta$ and $|y|\le\eta$ then $|w-y|\ge|w|-\eta\ge(1-\frac1R)|w|$,
	and since $-s < 0$, $|w - y|^{-s}\le(1 - \frac1R)^{-s}|w|^{-s}$. We conclude by integrating against
	$\rho_\eta$.
	
	\medskip
	
	(iv) $\psi_\eta \in C^\infty$ by Lemma 3.7, and $|\psi_\eta(x)| \le (|w| + \eta)^p \le C(1 + |x|)^p$ with $p < \alpha$. By \eqref{eq:gradpsi} and \eqref{eq:holder} applied to the probability density $\rho_\eta$ yields 
	\[
	|\nabla \psi_\eta| \le p \int |w - y|^{p-1} \rho_\eta (y) \dd y \le p \psi_\eta^{(p-1)/p} \le p(|w| + \eta)^{p-1} \le C(1 + |x|)^{p-1}. 
	\]
	
	Finally, by (i), $|(-\Delta)^{\alpha/2}\psi_\eta| = B_{p,\alpha} k_\eta$, and (ii) on $\{|w| < 2\eta\}$ together with (iii) for $R = 2$ on $\{|w| \ge 2\eta\}$ give
	\[
	k_\eta(x) \le C_\rho\,\eta^{-s}\,\mathbf{1}_{\{|w| < 2\eta\}} + 2^s |w|^{-s}\,\mathbf{1}_{\{|w| \ge 2\eta\}} \le C(1 + |x|)^{-s},
	\]
	where the last inequality holds because $1 + |x| \le 1 + |x^\star| + 2\eta$ on the first set, while $1 + |x| \le 1 + |x^\star| + |w| \le \bigl(1 + \frac{1+|x^\star|}{2\eta}\bigr)|w|$ on the second. Notice that $C$ depends on $\eta$ and $x^\star$, as permitted in \eqref{eq:classCp}.
	
\end{proof}

The following result provides bounds on the gradient of $\psi_\eta$.

\begin{lemma}\label{lem:gradient}
	For all $x\in\R^d$, it holds
	\begin{align}
		|\nabla\psi_\eta(x)|&\le p\,\psi_\eta(x)^{\frac{p-1}{p}}, \label{eq:D1}\\
		p\,\psi_\eta(x)-p\,\eta\,\psi_\eta(x)^{\frac{p-1}{p}}
		&\le\big\langle\nabla\psi_\eta(x),w\big\rangle
		\le p\,\psi_\eta(x) + p \, \eta \,\psi_\eta( x)^{\frac{p-1}{p}}. \label{eq:D2}
	\end{align}
\end{lemma}

\begin{proof} Inequality
	\eqref{eq:D1} was shown in the proof of Lemma~\ref{lem:fraclappsi}(iv). For
	\eqref{eq:D2} write $w = (w-y) + y$ inside the inner product in \eqref{eq:gradpsi}, and notice that
	\begin{align*}
		\langle\nabla\psi_\eta(x),w\rangle
		& =p\int|w-y|^{p-2}\big(|w-y|^2+\langle w-y,y\rangle\big)\rho_\eta(y)\dd y \\
		& =p\,\psi_\eta(x)+p\int|w-y|^{p-2}\langle w-y,y\rangle\rho_\eta(y)\dd y .
	\end{align*}
	On $\supp\rho_\eta$, $|\langle w-y,y\rangle|\le\eta|w-y|$, so the last integral is
	bounded in modulus by $\eta\int|w-y|^{p-1}\rho_\eta \dd y \le \eta \, \psi_\eta^{(p-1)/p}$
	by \eqref{eq:holder}.
\end{proof}

The following lemma gives two upper bounds for certain moments of $f$.

\begin{lemma}\label{lem:moments}
	Recall that $w = x - x^\star$. For every probability density $f$ with $V_p < \infty$ and every $R > 1$,
	\begin{align}
		\int_{\R^d}k_\eta\,|w|^\alpha f\dd x &\le \Big(1 - \frac1R\Big)^{-s}V_p + C_\rho R^\alpha \eta^p, \label{eq:K1}\\
		\int_{\R^d}k_\eta \, f \dd x & \le C_\rho \, \eta^{-s}, \label{eq:K2}
	\end{align}
	where $C_\rho$ is given by \eqref{eq:Crho}.
\end{lemma}

\begin{proof}
	For \eqref{eq:K1} split $\R^d=\{|w|\ge R\eta\}\cup\{|w|<R\eta\}$. Now, on the first set
	Lemma~\ref{lem:fraclappsi}(iii) gives $k_\eta|w|^\alpha\le(1-\frac1R)^{-s}|w|^{\alpha-s}=(1-\frac1R)^{-s}|w|^p$. Whereas
	on the second set, Lemma~\ref{lem:fraclappsi}(ii) gives
	$k_\eta|w|^\alpha\le C_\rho\eta^{-s}(R\eta)^\alpha=C_\rho R^\alpha\eta^p$. Integrating
	against $f$ yields the desired inequality.
	
	On the other hand, \eqref{eq:K2} follows from Lemma~\ref{lem:fraclappsi}(ii) and $\int_{\R^d} f \dd x = 1$. 
\end{proof}

\subsection{Evolution of the regularized moment and auxiliary estimates}
\label{subsec:evolution}

We start by setting
\begin{equation}\label{eq:Cpa}
	C_{p,\alpha}:=\nu p-\gamma^\alpha B_{p,\alpha},
	\qquad
	\Lambda( \tau', R) := ( 1 + \tau')\bigg(1 - \frac1R \bigg)^{-s} \, ,
\end{equation}
where $\tau'>0$, $R>1$
and noticing that $\Lambda(\tau',R) \searrow 1$, as $\tau' \to 0$ and $R \to \infty$.

We now establish an upper bound on the rate of change of $\Vpe$.

\begin{proposition}\label{prop:upper}
	Let $f$ be a weak solution of \eqref{eq:FFP0} on $[0,T]$. For every $\eta>0$,
	$\tau'>0$, $R>1$ and for a.e.\ $t\in[0,T]$,
	
	\begin{equation} \label{eq:upper}
		\begin{split}
			\frac{\dd}{\dd t}\Vpe  & \le -\Big[\nu p-\Lambda(\tau',R)\,\gamma^\alpha B_{p,\alpha}\Big]\Vpe
			+\nu p\,(\delta+\eta)\,(\Vpe)^{\frac{p-1}{p}} \\
			& \hspace{3cm} +\gamma^\alpha B_{p,\alpha}C_\rho\Big[(1+\tau')R^\alpha\eta^p+\thetaq{\alpha}{\tau'}\,\delta^\alpha\eta^{-s}\Big].
		\end{split}
	\end{equation}
	
\end{proposition}

\begin{proof}
	By Lemma~\ref{lem:fraclappsi}(iv) and Lemma~\ref{lem:extension} we may take
	$\varphi=\psi_\eta$ in \eqref{eq:weak}.
	
	\emph{Drift term.} Decompose $\xE-x=(\xE-x^\star)-w$. By \eqref{eq:D1} and the
	lower inequality in \eqref{eq:D2},
	\[
	\langle\nabla\psi_\eta,\xE-x\rangle\le|\nabla\psi_\eta|\,\delta-\langle\nabla\psi_\eta,w\rangle
	\le p\,\delta\,\psi_\eta^{\frac{p-1}{p}}-p\,\psi_\eta+p\,\eta\,\psi_\eta^{\frac{p-1}{p}} .
	\]
	Integrating against $f$ and using \eqref{eq:holder} with $g = f$, and $\Psi = \psi_\eta$, yields
	\[
	\nu\int\langle\nabla\psi_\eta,\xE-x\rangle f\dd x\le-\nu p\,\Vpe+\nu p\,(\delta+\eta)\,(\Vpe)^{\frac{p-1}{p}} .
	\]
	
	\emph{Fractional term.} By Lemma~\ref{lem:fraclappsi}(i),
	$-\gamma^\alpha\int((-\Delta)^{\alpha/2}\psi_\eta)|\xE-x|^\alpha f \dd x = \gamma^\alpha B_{p,\alpha}\int k_\eta|\xE-x|^\alpha f \dd x$.
	Since $|\xE-x|\le\delta+|w|$, \eqref{eq:E} with $q=\alpha$ gives
	$|\xE-x|^\alpha\le(1+\tau')|w|^\alpha+\thetaq{\alpha}{\tau'}\delta^\alpha$, and then
	\eqref{eq:K1}--\eqref{eq:K2} yield
	\[
	\begin{split}
		\gamma^\alpha B_{p,\alpha}\int k_\eta|\xE-x|^\alpha f\dd x
		& \le \\
		& \hspace{-3cm} \gamma^\alpha B_{p,\alpha}\Big[(1+\tau')\Big(1-\frac1R\Big)^{-s}V_p+(1+\tau')C_\rho R^\alpha\eta^p+\thetaq{\alpha}{\tau'}C_\rho\,\delta^\alpha\eta^{-s}\Big].
	\end{split}
	\]
	Finally $V_p\le\Vpe$ by \eqref{eq:Vcomparison}. Adding the two contributions
	gives \eqref{eq:upper}.
\end{proof}

\begin{remark}[Sharpness of the rate]\label{rem:sharp}
	If $\xE\equiv x^\star$, then $\delta\equiv0$ and every inequality in the above
	proof is an equality up to the regularization. Letting formally $\eta \to 0$,
	$\tau'\to0$, and $R\to\infty$, in such a way that $R^\alpha \eta^p \to 0$, one obtains
	\[
	\frac{\dd}{\dd t}V_p=-\big(\nu p-\gamma^\alpha B_{p,\alpha}\big)V_p=-C_{p,\alpha}V_p \, ,
	\]
	implying
	\[
	V_p(t)=V_p(0)\,e^{-C_{p,\alpha}t} \, .
	\]
	Thus the rate $C_{p,\alpha}$ and the condition $\nu p>\gamma^\alpha B_{p,\alpha}$
	in Theorem~\ref{thm:main} below are sharp for this argument. Thus the only losses in
	\eqref{eq:upper} are the factor $\Lambda \searrow 1$ and the $\eta$-remainders. 
\end{remark}

\begin{proposition}[Lower bound]\label{prop:lower}
	Under the assumptions of Proposition~\ref{prop:upper}, for a.e.\ $t\in[0,T]$,
	\begin{equation}\label{eq:lower}
		\frac{\dd}{\dd t}\Vpe\ \ge\ -\nu p\,\Vpe-\nu p\,(\delta+\eta)\,(\Vpe)^{\frac{p-1}{p}} .
	\end{equation}
\end{proposition}

\begin{proof}
	{\it Drift term}: by \eqref{eq:D1} and the upper inequality in \eqref{eq:D2},
	$\langle\nabla\psi_\eta,\xE-x\rangle\ge-p\,\delta\,\psi_\eta^{(p-1)/p}-p\,\psi_\eta-p\,\eta\,\psi_\eta^{(p-1)/p}$. We conclude by
	integrating and using \eqref{eq:holder}. {\it Fractional term}: by
	Lemma~\ref{lem:fraclappsi}(i) it equals
	$\gamma^\alpha B_{p,\alpha} \int k_\eta|\xE-x|^\alpha f \dd x \ge 0$, since $k_\eta > 0$ and
	$f \ge 0$, therefore we can discard it.
\end{proof}

We now establish the following quantitative Laplace principle, adapted from \cite{fornasier2024consensus} to our $V_p$ functional.

\begin{lemma}\label{lem:laplace}
	Let $\Ecal$ satisfy Assumption~\ref{ass:E} with $\underline{\Ecal}=0$, let
	$f\in\Pcal(\R^d)$ with $V_p<\infty$, and let $\xE$ be given by
	\eqref{eq:x_tot}. Let $r\in(0,R_0]$, $E_r:=\sup_{x\in B_r(x^\star)}\Ecal(x)$,
	and $q>0$ with $q+E_r\le \Ecal_\infty$. Then
	\begin{equation}\label{eq:laplace}
		|\xE-x^\star|\le C_r+\frac{e^{-\beta q}}{\rho(B_r(x^\star))}\,V_p^{1/p},
		\text{ where } C_r:=\frac{(q+E_r)^{\mu}}{\zeta},
	\end{equation}
	and $\rho(B_r(x^\star)):=\int_{B_r(x^\star)}f\dd x$.
\end{lemma}

\begin{proof}
	Let $\omega_\beta(x):=e^{-\beta\Ecal(x)}$ and $\|\omega_\beta\|_{L^1(f)}:=\int\omega_\beta f$.
	Then
	\[
	|\xE-x^\star|=\Big|\int(x-x^\star)\frac{\omega_\beta(x)}{\|\omega_\beta\|_{L^1(f)}}f\dd x\Big|
	\le\int|x-x^\star|\frac{\omega_\beta(x)}{\|\omega_\beta\|_{L^1(f)}}f\dd x,
	\]
	and
	\[
	\|\omega_\beta\|_{L^1(f)}\ge e^{-\beta E_r}\rho(B_r(x^\star)).
	\]
	Set $\tilde r:=C_r=(q+E_r)^\mu/\zeta$ and split the integral over
	$B_{\tilde r}(x^\star)$ and its complement. On $B_{\tilde r}(x^\star)$ we have
	$|x-x^\star|\le\tilde r$ and the normalized weights integrate to at most one, so
	this part is less or equal than $\tilde r$. For $x\notin B_{\tilde r}(x^\star)$ two cases occur:
	if $x\in B_{R_0}(x^\star)$, then by \eqref{eq:assE}
	$\Ecal(x)\ge(\zeta|x-x^\star|)^{1/\mu}\ge(\zeta\tilde r)^{1/\mu}=q+E_r$; if
	$x\notin B_{R_0}(x^\star)$, then $\Ecal(x)>\Ecal_\infty\ge q+E_r$. In both cases
	$\omega_\beta(x)\le e^{-\beta(q+E_r)}$, and therefore
	\[
	\int_{B_{\tilde r}(x^\star)^c}|x-x^\star|\frac{\omega_\beta}{\|\omega_\beta\|_{L^1(f)}}f\dd x
	\le\frac{e^{-\beta(q+E_r)}}{e^{-\beta E_r}\rho(B_r(x^\star))}\int|x-x^\star|f\dd x
	\le\frac{e^{-\beta q}}{\rho(B_r(x^\star))}V_p^{1/p},
	\]
	by Jensen's inequality. This concludes the proof of \eqref{eq:laplace}.
\end{proof}

We now establish a lower bound for the mass around the minimizer. Let
\begin{equation}\label{eq:bump}
	\begin{split}
		& \phi_1(u):=\begin{cases}\exp\Big(1-\dfrac{1}{1-|u|^2}\Big),&|u|<1,\\[1mm] 0,&|u|\ge1,\end{cases}
		\qquad
		\phi_r(x):=\phi_1\Big(\frac{x-x^\star}{r}\Big), \\
		& g_1:=-(-\Delta)^{\alpha/2}\phi_1 .
	\end{split}
\end{equation}
Then $0\le\phi_r\le\mathbf 1_{B_r(x^\star)}$, $\phi_r\in C^\infty_c\subset\Ccal_p$,
and by the change of variables $y=x^\star+rz$, in
\eqref{eq:fraclap},
\begin{equation}\label{eq:scaling}
	\begin{split}
		& (-\Delta)^{\alpha/2}\phi_r(x)=r^{-\alpha}\big((-\Delta)^{\alpha/2}\phi_1\big)(u),\qquad \text{where } u:=\frac{x-x^\star}{r}, \\
		& \nabla\phi_r(x)=-\frac{2r^2\,w}{(r^2-|w|^2)^2}\,\phi_r(x)\ \ (|w|<r).
	\end{split}
\end{equation}

\begin{lemma}[Sign of $(-\Delta)^{\alpha/2}\phi_1$ near the boundary]\label{lem:sign}
	Let $\mu_0:=\min_{|y|\le1/2}\phi_1(y)=e^{-1/3}$, $M_2:=\sup_{\R^d}\|D^2\phi_1\|$ and
	$\kappa_0:=\mu_0\,\omega_d\,2^{-d}\,(3/2)^{-d-\alpha}$. Choose
	$\varrho_\ast\in(0,\frac14]$ such that $M_2\,d\omega_d\,\frac{\varrho_\ast^{2-\alpha}}{2-\alpha}\le\frac{\kappa_0}{3}$,
	and then $\varsigma\in[\frac34,1)$ such that
	$\frac{2d\omega_d}{\alpha}\varrho_\ast^{-\alpha}\exp\big(1-\frac{1}{1-\varsigma^2}\big)\le\frac{\kappa_0}{3}$.
	Then
	\begin{equation}\label{eq:gmin}
		g_1(u)\ \ge\ g_{\min}:=\frac{c_{d,\alpha}\,\kappa_0}{6}>0\quad\text{for }\varsigma\le|u|<1,
		\qquad
		g_1(u)>0\quad\text{for }|u|\ge1 .
	\end{equation}
\end{lemma}

\begin{proof}
	For $|u|\ge1$, $\phi_1(u)=0$ and \eqref{eq:fraclap} gives
	$(-\Delta)^{\alpha/2}\phi_1(u)=-c_{d,\alpha}\int\phi_1(y)|u-y|^{-d-\alpha}\dd y<0$.
	
	Let $\varsigma\le|u|<1$ and split the integral in \eqref{eq:fraclap} at
	$|h|=\varrho_\ast$. \emph{Near field:} by Taylor's theorem
	$|2\phi_1(u)-\phi_1(u+h)-\phi_1(u-h)|\le M_2|h|^2$, so the near-field integral is
	bounded in modulus by $M_2\,d\omega_d\,\varrho_\ast^{2-\alpha}/(2-\alpha)\le\kappa_0/3$.
	\emph{Far field:} dropping $-\phi_1(u+h)\le0$,
	\[
	\int_{|h|\ge\varrho_\ast}\frac{2\phi_1(u)-\phi_1(u+h)-\phi_1(u-h)}{|h|^{d+\alpha}}\dd h
	\le2\phi_1(u)\,d\omega_d\,\frac{\varrho_\ast^{-\alpha}}{\alpha}-\int_{|h|\ge\varrho_\ast}\frac{\phi_1(u-h)}{|h|^{d+\alpha}}\dd h .
	\]
	Since $\phi_1$ is radially decreasing, $\phi_1(u)\le\exp(1-\frac{1}{1-\varsigma^2})$,
	so the first term is less than or equal to $\kappa_0/3$. For the second, substitute $y=u-h$ and
	restrict to $y\in B_{1/2}(0)$: there $|h|=|u-y|\ge|u|-\frac12\ge\frac14\ge\varrho_\ast$
	and $|u-y|\le|u|+\frac12<\frac32$, whence
	$\int_{|h|\ge\varrho_\ast}\phi_1(u-h)|h|^{-d-\alpha}\dd h\ge\mu_0|B_{1/2}|(3/2)^{-d-\alpha}=\kappa_0$.
	Altogether
	\[
	(-\Delta)^{\alpha/2}\phi_1(u)\le\frac{c_{d,\alpha}}{2}\big[\frac{\kappa_0}{3}+\frac{\kappa_0}{3}-\kappa_0\big]=-\frac{c_{d,\alpha}\kappa_0}{6}.
	\] 
\end{proof}

With $\varsigma$ and $g_{\min}$ as in Lemma~\ref{lem:sign} we define
\begin{gather}
	m_\varsigma:=\min_{|u|\le\varsigma}\phi_1(u)=e^{1-\frac{1}{1-\varsigma^2}},\qquad
	M_\varsigma:=\max_{|u|\le\varsigma}\max\big\{0,(-\Delta)^{\alpha/2}\phi_1(u)\big\},
	\label{eq:massconst1}\\
	\Sigma_\varsigma:=\sup_{\varsigma\le|u|<1}\frac{2|u|}{(1-|u|^2)^2}\Big(\frac{2|u|\,\phi_1(u)}{(1-|u|^2)^2}\Big)^{\frac{1}{\alpha-1}},
	\label{eq:massconst1b}\\
	P_{\rm in}(r,B):=\nu\,\frac{2\varsigma}{(1-\varsigma^2)^2}\,\frac{B+\varsigma r}{r}
	+\gamma^\alpha\,\frac{M_\varsigma}{m_\varsigma}\Big(\frac{B+\varsigma r}{r}\Big)^{\alpha},
	\label{eq:massconst2}\\
	P_{\rm bd}:=\nu\Big(\frac{\nu}{\gamma^\alpha g_{\min}}\Big)^{\frac{1}{\alpha-1}}\Sigma_\varsigma,
	\qquad
	P(r,B):=\max\{P_{\rm in}(r,B),P_{\rm bd}\}.
	\label{eq:massconst3}
\end{gather}
Here $M_\varsigma<\infty$ because $(-\Delta)^{\alpha/2}\phi_1$ is continuous,
and $\Sigma_\varsigma<\infty$ because $\phi_1(u)^{1/(\alpha-1)}$ vanishes faster
than any power of $1-|u|^2$ as $|u|\to1$. Note that $P_{\rm bd}$ depends neither
on $r$ nor on $B$.

In the next result we establish a lower bound for the mass around the minimizer.

\begin{proposition}\label{prop:mass}
	Let $\nu,\gamma>0$, $r>0$, $T>0$, and let $f$ be a weak solution of
	\eqref{eq:FFP0} on $[0,T]$ with $\sup_{t\in[0,T]}|\xE(t)-x^\star|\le B$. Then for
	all $t\in[0,T]$,
	\begin{equation}\label{eq:mass}
		\rho_t\big(B_r(x^\star)\big)\ \ge\ \int_{\R^d}\phi_r\,f \dd x\ \ge\ \Big(\int_{\R^d}\phi_r\,f_0\dd x\Big)\,e^{-P(r,B)\,t}.
	\end{equation}
\end{proposition}

\begin{proof}
	Taking $\varphi=\phi_r\in C^2_c(\R^d)\subset\Ccal_p$ in \eqref{eq:weak},
	$\frac{\dd}{\dd t}\int \phi_r f \dd x=\int(T_1+T_2)f\dd x$ with
	\[
	T_1(x):=\nu\langle\nabla\phi_r(x),\xE-x\rangle
	=-\nu\frac{2r^2\langle w,\xE-x\rangle}{(r^2-|w|^2)^2}\phi_r(x),
	\]
	and
	\[
	T_2(x):=-\gamma^\alpha|\xE-x|^\alpha(-\Delta)^{\alpha/2}\phi_r(x),
	\]
	and we observe that
	\begin{equation}\label{eq:outward}
		\langle w,\xE-x\rangle=\langle w,\xE-x^\star\rangle-|w|^2\le|w|\,(\delta-|w|).
	\end{equation}
	It suffices to show $T_1+T_2\ge-P\phi_r$ pointwise, since then Gr\"onwall's
	inequality and $\phi_r\le\mathbf 1_{B_r(x^\star)}$ give \eqref{eq:mass}. We split
	$\R^d$ into $\Omega_{\rm out}:=\{|w|\ge r\}$, $\Omega_{\rm in}:=\{|w|\le\varsigma r\}$
	and $\Omega_{\rm bd}:=\{\varsigma r<|w|<r\}$.
	
	\emph{Region $\Omega_{\rm out}$.} Here $\phi_r=0$ and $\nabla\phi_r=0$, so
	$T_1=0$; by \eqref{eq:scaling} and Lemma~\ref{lem:sign},
	$(-\Delta)^{\alpha/2}\phi_r<0$, so $T_2\ge0$.
	
	\emph{Region $\Omega_{\rm in}$.} Here $(r^2-|w|^2)^2\ge(1-\varsigma^2)^2r^4$ and
	$|\xE-x|\le\delta+|w|\le B+\varsigma r$, so by \eqref{eq:outward}
	\[
	T_1\ge-\nu\frac{2r^2|w|\,|\xE-x|}{(r^2-|w|^2)^2}\phi_r\ge-\nu\frac{2\varsigma\,(B+\varsigma r)}{(1-\varsigma^2)^2\,r}\,\phi_r .
	\]
	By \eqref{eq:scaling}, $(-\Delta)^{\alpha/2}\phi_r\le r^{-\alpha}M_\varsigma$ on
	$\Omega_{\rm in}$, while $\phi_r\ge m_\varsigma$ there; hence
	$T_2\ge-\gamma^\alpha(B+\varsigma r)^\alpha r^{-\alpha}M_\varsigma\ge-\gamma^\alpha\frac{M_\varsigma}{m_\varsigma}\big(\frac{B+\varsigma r}{r}\big)^\alpha\phi_r$.
	Adding, $T_1+T_2\ge-P_{\rm in}(r,B)\,\phi_r$.
	
	\emph{Region $\Omega_{\rm bd}$.} Set $g_r(x):=-(-\Delta)^{\alpha/2}\phi_r(x)=r^{-\alpha}g_1(u)\ge r^{-\alpha}g_{\min}>0$
	by \eqref{eq:scaling} and Lemma~\ref{lem:sign}; then $T_2=\gamma^\alpha|\xE-x|^\alpha g_r\ge0$.
	If $\langle w,\xE-x\rangle\le0$ then $T_1\ge0$ and we are done. Otherwise (this
	requires $|w|<\delta$ by \eqref{eq:outward}, which is possible on the layer when
	$\delta>\varsigma r$) write
	\[
	|T_1|\le\nu\,A_r(x)\,|\xE-x|,\quad
	A_r(x):=\frac{2r^2|w|\,\phi_r(x)}{(r^2-|w|^2)^2}=\frac1r\,A_1(u), \quad
	A_1(u):=\frac{2|u|\,\phi_1(u)}{(1-|u|^2)^2},
	\]
	and define the crossover length $\vartheta(x):=\big(\nu A_r(x)/(\gamma^\alpha g_r(x))\big)^{1/(\alpha-1)}$.
	If $|\xE-x|\ge\vartheta$, then, since $\alpha>1$,
	$T_2=\gamma^\alpha g_r|\xE-x|^{\alpha-1}|\xE-x|\ge\gamma^\alpha g_r\vartheta^{\alpha-1}|\xE-x|=\nu A_r|\xE-x|\ge|T_1|$,
	so $T_1+T_2\ge0$. If $|\xE-x|<\vartheta$, then $T_1+T_2\ge T_1\ge-\nu A_r\vartheta$.
	Using $A_r=r^{-1}A_1(u)$ and $g_r=r^{-\alpha}g_1(u)$ one finds
	$\vartheta=r\big(\nu A_1(u)/(\gamma^\alpha g_1(u))\big)^{1/(\alpha-1)}$, so that the
	powers of $r$ cancel and
	\[
	\nu A_r\vartheta=\phi_1(u)\cdot\frac{2\nu|u|}{(1-|u|^2)^2}\Big(\frac{2\nu|u|\,\phi_1(u)}{\gamma^\alpha g_1(u)(1-|u|^2)^2}\Big)^{\frac{1}{\alpha-1}}
	\le\phi_r(x)\,P_{\rm bd},
	\]
	where $g_1(u)\ge g_{\min}$ and the definition of $\Sigma_\varsigma$ were used.
	Hence $T_1+T_2\ge-P_{\rm bd}\phi_r$ on $\Omega_{\rm bd}$, and
	$T_1+T_2\ge-P(r,B)\phi_r$ on $\R^d$.
\end{proof}

\begin{remark}
	The boundary layer is where the fractional term is indispensable: for $\gamma=0$
	and $\delta>r$ all characteristics of the drift leave $B_r(x^\star)$ in finite
	time, so no bound of the form \eqref{eq:mass} can hold. The constant
	$P_{\rm bd} \propto \nu(\nu/\gamma^\alpha)^{1/(\alpha-1)}$ is the fractional
	counterpart of the term $\lambda^2/\sigma^2$ appearing in the corresponding
	estimate for Brownian diffusion \cite[Prop.~4.6]{fornasier2024consensus}, to which it reduces
	in structure as $\alpha\to2$. The condition $\alpha>1$ is essential for this step.
\end{remark}

\subsection{The convergence theorem}\label{subsec:theorem}

\begin{theorem}\label{thm:main}
	Let $1<p<\alpha<2$ and $\nu,\gamma>0$ with $C_{p,\alpha}=\nu p-\gamma^\alpha B_{p,\alpha}>0$.
	Let $\Ecal\in C(\R^d)$ satisfy Assumption~\ref{ass:E} with $\underline{\Ecal} = 0$ and with constants $\Ecal_\infty$, $R_0$, $\mu$ as in \eqref{eq:assE},
	and let $f_0\in\Pcal(\R^d)$ with $V_p(0)<\infty$ and $x^\star\in\supp f_0$. Fix
	$\varepsilon\in(0,V_p(0))$ and $\theta\in(0,1)$, and choose successively
	\begin{itemize}
		\item[(i)] $\tau'>0$ and $R>1$ such that
		$\big(\Lambda(\tau',R)-1\big)\gamma^\alpha B_{p,\alpha}\le\frac{\theta}{4}C_{p,\alpha}$;
		\item[(ii)] $\kappa\in(0,1)$ such that
		$\nu p\,\kappa+\gamma^\alpha B_{p,\alpha}C_\rho(1+\tau')R^\alpha\kappa^p\le\frac{\theta}{4}C_{p,\alpha}$;
		\item[(iii)] $c\in(0,1)$ such that
		$\nu p\,c+\gamma^\alpha B_{p,\alpha}C_\rho\,\thetaq{\alpha}{\tau'}\,c^\alpha\kappa^{-s}\le\frac{\theta}{2}C_{p,\alpha}$.
	\end{itemize}
	Set $\eta:=\kappa\,\varepsilon^{1/p}$ and
	\begin{gather*}
		T^\star:=\frac{1}{(1-\theta)C_{p,\alpha}}\log\frac{\Vpe(0)}{\varepsilon},
		\qquad
		q_\varepsilon:=\frac12\min\Big\{\Big(\frac{\zeta c\,\varepsilon^{1/p}}{2}\Big)^{1/\mu},\,\Ecal_\infty\Big\},\\
		r_\varepsilon:=\max\Big\{r\in(0,R_0]:\ \max_{\overline{B_r(x^\star)}}\Ecal\le q_\varepsilon\Big\},
		\qquad
		B_\varepsilon := c\,\varepsilon^{1/p}, \\
		m_0 := \int\phi_{r_\varepsilon}f_0\dd x,\qquad
		m_\star := m_0\,e^{-P(r_\varepsilon,B_\varepsilon)\,T^\star},\qquad
		\beta_0 := \frac{1}{q_\varepsilon}\log \bigg( \frac{2}{m_\star c}\bigg( \frac{ \Vpe(0)}{\varepsilon}\bigg)^{1/p} \bigg),
	\end{gather*}
	with $P(\cdot,\cdot)$ as in \eqref{eq:massconst3}. Let $f$ be a weak solution of
	\eqref{eq:FFP0} on $[0,T^\star]$ with initial datum $f_0$. Then for every
	$\beta>\beta_0$ there exists
	\begin{equation}\label{eq:Twindow}
		T \in \bigg[\frac{(1-\theta)\,C_{p,\alpha}}{\nu p\,(1+c+\kappa)}\,T^\star,\ T^\star \bigg]
		\quad\text{such that}\quad \Vpe(T)=\varepsilon,
	\end{equation}
	and $\Vpe$ decays at least exponentially fast on $[0,T]$:
	\begin{equation}\label{eq:decay}
		\Vpe(t)\le\Vpe(0)\,e^{-(1-\theta)C_{p,\alpha}t}\qquad\text{for all }t\in[0,T].
	\end{equation}
	In particular $W_p^p(f(\cdot,T),\delta_{x^\star})=V_p(T)\le\varepsilon$ and, for
	every $\tau>0$,
	\begin{equation}\label{eq:decayWp}
		W_p^p\big(f(\cdot,t),\delta_{x^\star}\big)\le\Big[(1+\tau)V_p(0)+\thetaq{p}{\tau}\,\kappa^p\varepsilon\Big]e^{-(1-\theta)C_{p,\alpha}t}\qquad\text{for all }t\in[0,T].
	\end{equation}
\end{theorem}

Before the proof, we comment on the definitions. The choices (i)--(iii) are
possible because $\Lambda\to1$ and each condition is a smallness condition on a
single parameter. In other words, $c$ is chosen after $\kappa$ because $\kappa^{-s}\to\infty$ as
$\kappa\to0$. Since $\Ecal$ is continuous with $\Ecal(x^\star)=0$, the map
$r\mapsto\max_{\overline{B_r(x^\star)}}\Ecal$ is continuous and nondecreasing,
and $\max_{\overline{B_r(x^\star)}}\Ecal<q_\varepsilon$ for $r$ small since
$\Ecal(x^\star)=0<q_\varepsilon$. Therefore, the set defining $r_\varepsilon$ is a nonempty
interval, closed in $(0,R_0]$, and $r_\varepsilon\in(0,R_0]$ is well defined.
Finally $m_0>0$ because $x^\star\in\supp f_0$ and $\phi_{r_\varepsilon}>0$ on
$B_{r_\varepsilon}(x^\star)$, and $\Vpe(0)\ge V_p(0)>\varepsilon$ by
\eqref{eq:Vcomparison}, so $T^\star>0$.

\begin{proof}
	
	\emph{Step 0: the role of $q_\varepsilon$ and $r_\varepsilon$.} By the definition of
	$r_\varepsilon$, $E_{r_\varepsilon}=\sup_{B_{r_\varepsilon}(x^\star)}\Ecal\le\max_{\overline{B_{r_\varepsilon}(x^\star)}}\Ecal\le q_\varepsilon$,
	and by the definition of $q_\varepsilon$, $2q_\varepsilon\le\Ecal_\infty$. Hence
	$q_\varepsilon+E_{r_\varepsilon}\le2q_\varepsilon\le\Ecal_\infty$, so Lemma~\ref{lem:laplace}
	applies with $r=r_\varepsilon\le R_0$ and $q=q_\varepsilon$. The first term in \eqref{eq:laplace} satisfies
	
	\begin{equation}\label{eq:Creps}
		C_{r_\varepsilon}=\frac{(q_\varepsilon+E_{r_\varepsilon})^\mu}{\zeta}\le\frac{(2q_\varepsilon)^\mu}{\zeta}\le\frac{c}{2}\,\varepsilon^{1/p},
	\end{equation}
	the last inequality holds thanks to the definition of $q_\varepsilon$, which was chosen
	precisely for this purpose, and $\beta > \beta_0$ gives
	$\frac{e^{-\beta q_\varepsilon}}{m_\star}\,\Vpe(0)^{1/p}<\frac{c}{2}\,\varepsilon^{1/p}$.
	Consequently, at every $t\in[0,T^\star]$ at which both
	$\rho_t(B_{r_\varepsilon}(x^\star))\ge m_\star$ and $V_p(t)\le\Vpe(0)$ hold,
	Lemma~\ref{lem:laplace} yields
	\begin{equation}\label{eq:deltabound}
		\delta(t)\le\frac{c}{2}\,\varepsilon^{1/p}+\frac{e^{-\beta q_\varepsilon}}{m_\star}\,V_p(t)^{1/p}
		\le\frac{c}{2}\,\varepsilon^{1/p}+\frac{e^{-\beta q_\varepsilon}}{m_\star}\,\Vpe(0)^{1/p}
		<c\,\varepsilon^{1/p}.
	\end{equation}
	The remaining steps consist in identifying the times at which the two
	hypotheses of \eqref{eq:deltabound} hold.

	\emph{Step 1: the stopping time.} Define
	\[
	T_\beta := \sup\Big\{t\in[0,T^\star] : \ \Vpe(t')>\varepsilon\ \text{ and }\ \delta(t')< c\,\varepsilon^{1/p} \ \text{ for all }t' \in [0,t] \Big\}.
	\]
	Both $t\mapsto\Vpe(t)$ (Lipschitz by Lemma~\ref{lem:extension}, since
	$\psi_\eta\in\Ccal_p$ by Lemma~\ref{lem:fraclappsi}(iv)) and $t\mapsto\delta(t)$
	are continuous. At $t=0$ we have $\Vpe(0)>\varepsilon$ and both hypotheses of
	\eqref{eq:deltabound} hold, since
	$\rho_0(B_{r_\varepsilon}(x^\star))\ge\int\phi_{r_\varepsilon}f_0\dd x=m_0\ge m_\star$ and
	$V_p(0)\le\Vpe(0)$, hence $\delta(0)<c\,\varepsilon^{1/p}$.
	By continuity both strict inequalities persist on some $[0,t_0]$, so
	$T_\beta\ge t_0>0$. The set in the definition of $T_\beta$ is an interval
	containing $[0,T_\beta)$, so both strict inequalities hold on $[0,T_\beta)$. Thus
	by continuity, $\Vpe(T_\beta)\ge\varepsilon$ and
	$\delta( T_\beta) \le c \, \varepsilon^{1/p}$, and if $T_\beta < T^\star$ at least one
	of these two inequalities is an equality.
	
	\emph{Step 2: differential inequalities on $[0,T_\beta)$.} For $t<T_\beta$ we
	have $\Vpe(t)>\varepsilon$ and $\delta(t)<c\,\varepsilon^{1/p}$. In particular,
	$\delta+\eta<(c+\kappa)\,\varepsilon^{1/p}<(c+\kappa)\,\Vpe(t)^{1/p}$ and
	$\eta^p=\kappa^p\varepsilon<\kappa^p\,\Vpe(t)$. Now, to estimate the right-hand side of \eqref{eq:upper} let us notice that
	\begin{align}
		\nu p\,(\delta+\eta)\,(\Vpe)^{\frac{p-1}{p}}&\le\nu p\,(c+\kappa)\,\Vpe, \label{eq:aux1} \\
		(1+\tau')R^\alpha\eta^p&\le(1+\tau')R^\alpha\kappa^p\,\Vpe, \label{eq:aux2} \\
		\thetaq{\alpha}{\tau'}\,\delta^\alpha\eta^{-s} & \le\thetaq{\alpha}{\tau'}\,c^\alpha\varepsilon^{\alpha/p}\,\kappa^{-s}\varepsilon^{-s/p}
		=\thetaq{\alpha}{\tau'}\,c^\alpha\kappa^{-s}\,\varepsilon\le\thetaq{\alpha}{\tau'}\,c^\alpha\kappa^{-s}\,\Vpe, \label{eq:aux3}
	\end{align}
	where in the last line $\alpha-s=p$ and $\varepsilon<\Vpe$ were used. Had the
	stopping time been defined through $\delta<c\,\Vpe(t)^{1/p}$ instead of
	$\delta<c\,\varepsilon^{1/p}$, the last term could only be bounded by
	$\thetaq{\alpha}{\tau'}c^\alpha\kappa^{-s}\,(\Vpe/\varepsilon)^{s/p}\,\Vpe$, whose coefficient
	is not uniform on the window. This is why the window is tied to the accuracy
	$\varepsilon$. Thanks to (i)--(iii) and \eqref{eq:aux1}-\eqref{eq:aux2}-\eqref{eq:aux3}, it holds
	\[
	\frac{\dd}{\dd t}\Vpe\le-\Big[C_{p,\alpha} - \big(\Lambda( \tau', R)-1\big)\gamma^\alpha B_{p,\alpha}-\frac{\theta}{4}C_{p,\alpha}-\frac{\theta}{2}C_{p,\alpha}\Big]\Vpe\le-(1-\theta)\,C_{p,\alpha}\,\Vpe .
	\]
	Similarly \eqref{eq:lower} gives $\frac{\dd}{\dd t}\Vpe\ge-\nu p(1+c+\kappa)\Vpe$.
	By Gr\"onwall's inequality, for all $t\in[0,T_\beta]$, it holds
	\begin{equation}\label{eq:gronwall}
		\Vpe(0)\,e^{-\nu p(1+c+\kappa)t}\ \le\ \Vpe(t)\ \le\ \Vpe(0)\,e^{-(1-\theta)C_{p,\alpha}t} \, .
	\end{equation}
	In particular, $\Vpe$ is nonincreasing on $[0,T_\beta]$.
	
	\emph{Step 3: mass bound and Laplace principle on $[0,T_\beta]$.} By Steps
	1--2, $\delta(t)\le c\,\varepsilon^{1/p}=B_\varepsilon$ and $V_p(t)\le\Vpe(t)\le\Vpe(0)$
	for $t\in[0,T_\beta]$. Proposition~\ref{prop:mass} with $r=r_\varepsilon$,
	$B=B_\varepsilon$ and $T=T_\beta\le T^\star$ yields
	$\rho_t(B_{r_\varepsilon}(x^\star))\ge m_0e^{-P(r_\varepsilon,B_\varepsilon)t}\ge m_\star$
	on $[0,T_\beta]$. Both hypotheses of \eqref{eq:deltabound} therefore hold on
	$[0,T_\beta]$, and Step 0 gives $\delta(t)<c\,\varepsilon^{1/p}$ for all
	$t\in[0,T_\beta]$.
	
	\emph{Step 4: case analysis.}
	\emph{Case (a): $T_\beta=T^\star$.} By \eqref{eq:gronwall},
	$\Vpe(T^\star)\le\Vpe(0)e^{-(1-\theta)C_{p,\alpha}T^\star}=\varepsilon$, while
	$\Vpe(T^\star)\ge\varepsilon$ by Step 1. Hence $\Vpe(T^\star)=\varepsilon$ and we
	take $T=T^\star$.
	\emph{Case (b): $T_\beta<T^\star$ and $\Vpe(T_\beta)=\varepsilon$.} Take
	$T=T_\beta$. The lower bound in \eqref{eq:gronwall} gives
	$\varepsilon\ge\Vpe(0)e^{-\nu p(1+c+\kappa)T_\beta}$, i.e.
	$T_\beta\ge\frac{1}{\nu p(1+c+\kappa)}\log\frac{\Vpe(0)}{\varepsilon}=\frac{(1-\theta)C_{p,\alpha}}{\nu p(1+c+\kappa)}T^\star$.
	\emph{Case (c): $T_\beta<T^\star$ and $\Vpe(T_\beta)>\varepsilon$.} By Step 1 the
	second condition is saturated, $\delta( T_\beta) = c \, \varepsilon^{1/p}$, which
	contradicts Step 3 at $t=T_\beta$. Therefore, this case does not occur.
	
	\emph{Step 5: conclusion.} In cases (a) and (b), $\Vpe(T)=\varepsilon$ with $T$
	as in \eqref{eq:Twindow}, and \eqref{eq:gronwall} gives \eqref{eq:decay}. By
	\eqref{eq:Vcomparison}, $V_p(T)\le\Vpe(T)=\varepsilon$ and
	$V_p(t)\le\Vpe(t)\le\Vpe(0)e^{-(1-\theta)C_{p,\alpha}t}$ with
	$\Vpe(0)\le(1+\tau)V_p(0)+\thetaq{p}{\tau}\eta^p$ and $\eta^p=\kappa^p\varepsilon$,
	which is \eqref{eq:decayWp}.
\end{proof}

\begin{remark}[Finite-accuracy result]\label{rem:finite}
	Theorem~\ref{thm:main} is a finite-accuracy statement, i.e. the accuracy
	$\varepsilon$ is fixed first, the regularization scale $\eta=\kappa\varepsilon^{1/p}$
	and the threshold $\beta_0=\beta_0(\varepsilon)$ are chosen afterwards, and the
	conclusion is that $V_p$ reaches $\varepsilon$ in a time comparable to $T^\star$
	with exponential decay up to that time. It does not assert that $V_p(t)\to0$ as
	$t\to\infty$ for fixed $\beta$, since $\beta_0(\varepsilon)\to\infty$ as
	$\varepsilon\to0$ (through $q_\varepsilon\to0$, $m_\star\to0$ and the factor
	$(\Vpe(0)/\varepsilon)^{1/p}$). This is the same
	structure as in \cite{fornasier2024consensus}. Here it is also intrinsic to the
	fractional setting, see Remark~\ref{rem:comparison}. The regularization is
	invisible at scales above $\varepsilon^{1/p}$. In other words, on the whole time window one has
	$\varepsilon\le\Vpe\le(1+\tau)V_p+\thetaq{p}{\tau}\kappa^p\varepsilon$, so
	$\Vpe$ and $V_p$ are comparable up to a factor arbitrarily close to one.
\end{remark}

\begin{remark}[Regularized moments]\label{rem:comparison} Testing \eqref{eq:FFP0} directly against $\varphi=|x-x^\star|^p$ produces, by
	\eqref{eq:powerid}, the fractional contribution
	$\gamma^\alpha B_{p,\alpha}\int|x-x^\star|^{p-\alpha}|\xE-x|^\alpha f\dd x$. When
	$\xE=x^\star$ the diffusion coefficient cancels the negative power exactly and
	one obtains the identity of Remark~\ref{rem:sharp}. When $\xE\ne x^\star$ the
	coefficient does not vanish at $x^\star$, and after splitting one is left with
	$\delta^\alpha V_{p-\alpha}$, where
	$V_{p-\alpha}:=\int|x-x^\star|^{p-\alpha}f\dd x$ is a moment of negative order.
	This quantity is $\infty$ for every $f$ with an atom at $x^\star$, and more
	generally it can be made arbitrarily large while $V_p$ stays bounded
	(concentrate mass $1-\epsilon$ in $B_{b^\#} (x^\star)$ and let $b^\# \to 0$).
	Consequently, no inequality of the form $V_{p-\alpha}\le G(V_p)$ can hold for
	general probability densities. Therefore, a bound on $V_{p-\alpha}$ requires information on
	the local structure of $f$ near $x^\star$. 
	
	A direct estimate of $V_{p - \alpha}$ through the modified Riesz potential estimate of \cite[Lemma~3.1]{harjulehto2018pointwise} requires control of the Hardy-Littlewood maximal function $M f( x^\star, t)$, which does not follow from $\int f \dd x = 1$ and cannot be replaced by a single-ball average. Moreover, the resulting expression in $r$ has no positive critical point. Instead, we followed here a regularization of the test function since this does not require an $L^\infty$ bound on the solution $f$ of \eqref{eq:FFP0}. The only conditions on the solution are that $f \geq 0$, $\int f \dd x = 1$, and the natural bound $V_p < \infty$. In addition, regularizing the test function moves the singularity from $f$ to
	$\psi_\eta$, where it is controlled by the explicit scaling
	$k_\eta=\eta^{-s}k_1(\cdot/\eta)$ of Lemma~\ref{lem:fraclappsi}, i.e.\ the singular factor $|x-x^\star|^{-s}$ is capped at height
	$\sim\eta^{-s}$ inside $B_\eta(x^\star)$ and essentially unchanged outside:
	for $|x-x^\star|\ge R\eta$ one has
	$\big(1+\tfrac1R\big)^{-s}|x-x^\star|^{-s}\le k_\eta(x)\le\big(1-\tfrac1R\big)^{-s}|x-x^\star|^{-s}$,
	the upper bound by Lemma~\ref{lem:fraclappsi}(iii) and the lower bound by the
	same argument, since $|x-x^\star-y|\le\big(1+\tfrac1R\big)|x-x^\star|$ on the
	support of $\rho_\eta$. The price is the explicit remainder
	$\delta^\alpha\eta^{-s}$ in \eqref{eq:upper}, and it is affordable because the
	theorem is a finite-accuracy statement. In other words, with $\eta=\kappa\varepsilon^{1/p}$ and
	$\delta<c\,\varepsilon^{1/p}$ -- which the quantitative Laplace principle
	guarantees on the relevant time window once $\beta$ is large -- one has
	$\delta^\alpha\eta^{-s}<c^\alpha\kappa^{-s}\varepsilon\le c^\alpha\kappa^{-s}\Vpe$, the
	exponents adding up to exactly one (recall that $\alpha - s = p$). The advantages are that only $\int f \dd x = 1$
	and $V_p<\infty$ are used, the argument works in every dimension $d\ge1$, all
	constants are explicit, the sharp rate $C_{p,\alpha}$ is retained up to an
	arbitrarily small loss (see Remark~\ref{rem:sharp}), and the finite-accuracy nature
	of the result becomes structural rather than incidental. 
\end{remark}

\begin{remark}\label{rmk:convergence_issues}
	We observe that the corresponding
	contraction rate in Theorem \ref{thm:main} depends explicitly on the dimension through
	$B_{p,\alpha}(d)\sim C\,d^{\alpha/2}$, for certain constant $C>0$, so that the condition
	$\nu p>\gamma^\alpha B_{p,\alpha}(d)$ becomes increasingly restrictive as
	$d$ grows. This behavior
	is well understood for standard diffusion-based CBO, where isotropic noise leads to dimension-dependent convergence conditions, while componentwise anisotropic
	diffusion yields dimension-independent rates, e.g. see
	\cite{carrillo2021consensus,fornasier2024consensus}.
	This naturally motivates the anisotropic particle formulation, in which the stochastic perturbation is scaled componentwise and which exhibits substantially better performance in high-dimensional numerical experiments. A rigorous convergence analysis of the corresponding anisotropic fractional model, including its dependence on the dimension, is beyond the scope of the present work and is left for future investigation. In the following section, we nevertheless explore this extension numerically and consider different alternative formulations.
\end{remark} 

\section{Numerical methods}\label{sec:numerical_methods}

We now focus on the solution of the dynamics, relying on splitting approaches.
Specifically, in each time interval $[t^n,t^{n+1}]$, with
$t^{n+1}=t^n+\Delta t$, we solve consecutively the following
problems
\begin{equation}\label{eq:subpbs}
	\begin{cases}
		&\partial_s f^* = \mathcal{A}_{drift}(f^*),\\
		&f^*(x,0) = f^n(x),
	\end{cases}
	\qquad
	\begin{cases}
		&\partial_s f^{**}= \mathcal{A}_{diff}(f^{**}),\\
		&f^{**}(x,0) = f^*(x,\Delta t),
	\end{cases}
	\qquad s\in[0,\Delta t],
\end{equation}
where $\mathcal{A}_{drift}$ and $\mathcal{A}_{diff}$ are the operators
describing the drift and diffusion parts of equation \eqref{eq:fFP}.
Let us denote by
\begin{equation}
	\begin{split}
		&\mathcal{S}_1^{\Delta t}(f^n(x)) = f^*(x,\Delta t),\\
		&\mathcal{S}_2^{\Delta t}(f^*(x,\Delta t)) = f^{**}(x,\Delta t),
	\end{split}
\end{equation}
the solutions of the two subproblems in \eqref{eq:subpbs}. Then, in compact
notation, we can write
\begin{equation*}
	f^{n+1}(x)
	=
	f^{**}(x,\Delta t)
	=
	\mathcal{S}_2^{\Delta t}
	\left(\mathcal{S}_1^{\Delta t}(f^n(x))\right).
\end{equation*}
The method provides a first-order approximation of the true solution over the
time interval $\Delta t$, which is realized through the Monte Carlo scheme
detailed in the following subsection.

\subsection{Asymptotic Nanbu-type Monte Carlo scheme}
To approximate the evolution of the density $f(x,t)$, we consider an ensemble
of $N_s$ particles $x_i$, $i=1,\ldots,N_s$, initially sampled from $f_0$.
We discretize the time interval $[0,T]$ as
$t^n=n\Delta t$, $n=0,\ldots,N_t$, and denote by $f^n$ the particle
approximation of $f(\cdot,t^n)$. The next iterate is given by 
\begin{equation}\label{eq:f_splitting}
	\begin{cases}
		&f^* = f^n (1-\Delta t \lambda_1) + \Delta t \lambda_1 \mathcal{Q}_1^{+}(f^n),\\
		&f^{n+1} = f^* (1-\Delta t \lambda_2) + {\color{black}\Delta t \lambda_2 \mathcal{Q}_2^{+}(f^*)},           
	\end{cases}
\end{equation}
for $\lambda_1,\lambda_2\geq0$ and
$\Delta t\,\min\{\lambda_1,\lambda_2\}\leq1$, .
Since the gain operators preserve mass, $\mathcal Q_1^+(f^n)$ and
$\mathcal Q_2^+(f^*)$ are probability densities whenever $f^n$ and $f^*$
are probability densities. Hence, both steps in \eqref{eq:f_splitting}
are convex combinations of probability densities and preserve the total mass. 

From a Monte Carlo viewpoint, we can interpret \eqref{eq:f_splitting} as follows: with probability $(1-\Delta t \lambda_1)$ an agent in position $x$ will not interact with other agents, while with probability $\Delta t \lambda_1$ it will interact according to the interactions given by $\mathcal{Q}_1^+(f^n)$, and hence it will modify its position according to the drift, moving to position $x_*$. Subsequently, the agent $x_*$ with probability $(1-\Delta t \lambda_2)$ will not interact with other individuals, while with probability $\Delta t \lambda_2$ it will modify its position according the stochastic diffusion update. To maximize the number of interactions, we assume $\lambda_1=\lambda_2 = 1/\Delta t$. 
Hence, sampling $x_i\sim f^0(x)$ for $i=1,\ldots,N_s$, we get
\begin{equation}\label{eq:binary_rules}
	\begin{cases}
		x_i^*
		=
		x_i^n+\nu\Delta t
		(\bar{x}_{\mathcal E}^n-x_i^n),
		\\[1mm]
		x_i^{n+1}
		=
		x_i^*
		+\sigma\sqrt{2\Delta t}\,
		D(\bar{x}_{\mathcal E}^n,x_i^*)\,z_i^n
		+\gamma(\Delta t)^{1/\alpha}
		D(\bar{x}_{\mathcal E}^n,x_i^*)\,\widetilde z_i^n,
	\end{cases}
\end{equation}
with $\nu,\sigma,\gamma >0$,   $z^n_i\sim \mathcal{N}(0,I_d)$  and
$\widetilde z_i^n\sim g_\alpha^{\textrm{RS}}$ denotes   a unit-time rotationally symmetric
$\alpha$-stable random vector with $\alpha \in (0,2]$. 

The interaction rules in \eqref{eq:binary_rules} define the underlying
Boltzmann-type kinetic dynamics, which we simulate by means of a Nanbu-type
Monte Carlo method \cite{nanbu1980direct}. Under the assumption of section \ref{sec:FPE} and the grazing-interaction
scaling, this kinetic model formally approaches the fractional Fokker-Planck
equation \eqref{eq:fFP}. Indeed, the particle scheme \eqref{eq:binary_rules} provides an asymptotic Direct Simulation
Monte Carlo (DSMC) approximation of the limiting dynamics. For a broader class of
DSMC methods for interacting particle systems, we
refer to \cite{albi2013binary,pareschi2013interacting}. 

Hence, we report in Algorithm~\ref{alg:KBO} the numerical procedure. The initial particle positions are
sampled from the initial density, typically chosen as a uniform distribution
in space unless prior information on the location of the global minimizer
is available. The parameters $\delta_{\rm stall}$ and $j_{\rm stall}$ are used to monitor the stabilization of the consensus point: the iteration is stopped when $\|\bar{x}_{\mathcal E}^{n+1}-\bar{x}_{\mathcal E}^{n}\|_\infty$
for $j_{\rm stall}$ consecutive iterations, or when the maximum number $N_t$ of iterations is reached.
\begin{algorithm}{$\alpha$-KBO}\label{alg:KBO}
	\begin{enumerate}
		\item[\texttt 1.] Draw $x_i^0$, $i=1,\ldots,N_s$, from $f^0(x)$
		and set $n=0$, $j=0$.
		\item[\texttt 2.] Compute $\bar{x}_\mathcal{E}^0$ as in equation \eqref{eq:x_tot}.
		\item[\texttt 3.] \texttt{while} $n<N_t$ \texttt{and} $j<j_{stall}$
		\begin{enumerate}
			\item  \texttt{for} $i=1$ \texttt{to}  $N_s$ 
			\begin{itemize}
				\item[-] sample $z^n_i\sim \mathcal{N}(0,I_d)$ and $\widetilde z^n_i\sim g_\alpha^{\textrm{RS}}$;
				\item[-] compute $x^{n+1}_i$ according to \eqref{eq:binary_rules}.
			\end{itemize}
			\texttt{end for}
			\item Compute $\bar{x}_\mathcal{E}^{n+1}$ as in equation \eqref{eq:x_tot}.
			\item  \texttt{if} $\Vert \bar{x}_\mathcal{E}^{n+1}-\bar{x}_\mathcal{E}^n\Vert_\infty\leq \delta_{stall}$
			\begin{enumerate}
				\item [] $j\leftarrow j+1$
			\end{enumerate}
			\texttt{else}
			\begin{enumerate}
				\item [] $j\leftarrow 0$
			\end{enumerate}
			\texttt{end if}  
			\item [] $n\leftarrow n+1$
		\end{enumerate}
		\texttt{end while} 
	\end{enumerate}
\end{algorithm}

To generate rotationally symmetric $\alpha$-stable random vectors
$\widetilde z_i$, we rely on the sub-Gaussian representation of multivariate
stable distributions, see \cite{samorodnitsky1994stable}. In this
construction, a positive $\alpha/2$-stable random variable $S_i$, is combined with
an independent $d$-dimensional Gaussian vector $G_i\sim\mathcal{N}(0,I_d)$, as $\tilde z_i=\sqrt{2S_i}G_i$.
In particular, the positive $q$-stable random variable $S_i$, with
$q=\alpha/2$, can be generated using the Kanter representation
\cite{kanter1975stable}. The resulting sampling
procedure is summarized in Algorithm~\ref{alg:stable_sampling}.
\begin{algorithm}{ Sampling $g^{\textrm{RS}}_\alpha$ }
	\label{alg:stable_sampling}
	\begin{enumerate}
		\item[\texttt 1.] Set $q=\alpha/2$.
		\item[\texttt 2.] \texttt{If} $\alpha<2$
		\begin{enumerate}
			\item Draw $U_i\sim\mathcal U(0,\pi)$ and $W_i\sim\mathrm{Exp}(1)$.
			\item Compute the positive $q$-stable random variable
			\begin{equation}\label{eq:Kanter}
				S_i=
				\frac{\sin(qU_i)}
				{(\sin U_i)^{1/q}}
				\left(
				\frac{\sin((1-q)U_i)}{W_i}
				\right)^{(1-q)/q}.
			\end{equation}
		\end{enumerate}
		\texttt{elseif } $\alpha=2$, set $S_i=1$.\\
		\texttt{endif }
		\item[\texttt 3.] Draw $G_i\sim\mathcal N(0,I_d)$.	
		\item[\texttt 4.] Set
		\[
		\widetilde z_i=\sqrt{2S_i}\,G_i.
		\]		
		\item[\texttt 5.] Return $\widetilde z_i$.
	\end{enumerate}
\end{algorithm} 

\begin{remark}
	\label{rmk:componentwise_stable}
	We remark that the construction in Algorithm~\ref{alg:stable_sampling} generates
	a rotationally symmetric $\alpha$-stable vector, characterized by
	\[
	\mathbb E\!\left[\exp\left(i\xi\cdot\widetilde z\right)\right]=
	\int_{\mathbb R^d}
	e^{i\xi\cdot z}\,
	g_\alpha^{\mathrm{RS}}(z)\,dz
	=
	e^{-|\xi|^\alpha}.
	\]
	which is consistent with the
	fractional diffusion term derived in Section \ref{sec:FPE} for the Fokker-Planck model \eqref{eq:fFP} for isotropic diffusion matrix.
	On the other hand, as discussed in Remark~\ref{rmk:convergence_issues}, the
	isotropic formulation exhibits an explicit dependence on the dimension through
	the contraction rate, which becomes increasingly restrictive as $d$ grows. This motivates the use of the anisotropic diffusion matrix where the noise amplitude is scaled independently
	along each coordinate. 
	Nevertheless, retaining a rotationally symmetric
	$\alpha$-stable process leads to a more general matrix-dependent nonlocal
	operator, beyond the classical fractional Laplacian considered above; see,
	e.g. \cite{schilling2010symbol}.
\end{remark}
\paragraph{KBO with componentwise $\alpha$-stable noise.}
Motivated by the previous remark, we consider the particle scheme
\eqref{eq:binary_rules} with the anisotropic choice
\eqref{eq:D_anisotropic}. In this case, the particle update can be equivalently
written as
\begin{equation}\label{eq:binary_rules_anisotropic}
	\begin{cases}
		x_i^*
		=
		x_i^n+\nu\Delta t
		(\bar{x}_{\mathcal E}^n-x_i^n),
		\\[1mm]
		x_i^{n+1}
		=
		x_i^*
		+\sigma\sqrt{2\Delta t}\,
		(\bar{x}_{\mathcal E}^n-x_i^*)\odot z_i^n
		+\gamma(\Delta t)^{1/\alpha}
		(\bar{x}_{\mathcal E}^n-x_i^*)\odot\widehat z_i^n,
	\end{cases}
\end{equation}
where we consider an alternative
componentwise construction of the stable noise,
$\widehat z_i\sim g_\alpha^{\mathrm{CW}}(\cdot)$, in which the components
$\widehat z_{i,k}$, $k=1,\ldots,d$, are sampled independently as
one-dimensional symmetric $\alpha$-stable random variables. Numerically, this
amounts to applying independently to each coordinate the one-dimensional Chambers-Mallows-Stuck (CMS) construction \cite{CMS1976}, which is closely related to the
Kanter representation used in Algorithm \ref{alg:stable_sampling}. The resulting sampling procedure is summarized
in Algorithm~\ref{alg:stable_sampling_CW}.

\begin{algorithm}{Sampling $g_\alpha^{\mathrm{CW}}$}
	\label{alg:stable_sampling_CW}
	\begin{enumerate}
		\item[\texttt 1.] \texttt{for} $k=1$ \texttt{to} $d$
		\begin{itemize}
			\item[-] draw independently
			\[
			U_{i,k}\sim\mathcal U\left(-\frac{\pi}{2},\frac{\pi}{2}\right),
			\qquad
			W_{i,k}\sim\mathrm{Exp}(1);
			\]
			
			\item[-] compute
			\begin{equation}\label{eq:CMS}
				\widehat z_{i,k}
				=
				\begin{cases}
					\displaystyle
					\frac{\sin(\alpha U_{i,k})}
					{(\cos U_{i,k})^{1/\alpha}}
					\left(
					\frac{\cos((1-\alpha)U_{i,k})}{W_{i,k}}
					\right)^{(1-\alpha)/\alpha},
					& \alpha\neq 1,
					\\[4mm]
					\displaystyle
					\tan(U_{i,k}),
					& \alpha=1.
				\end{cases}
			\end{equation}
		\end{itemize}
		\texttt{end for}
		
		\item[\texttt 2.] Set
		\[
		\widehat z_i
		=
		(\widehat z_{i,1},\ldots,\widehat z_{i,d})^T.
		\]
		
		\item[\texttt 3.] Return $\widehat z_i$.
	\end{enumerate}
\end{algorithm}
Hence, the numerical scheme for this approach follows Algorithm \ref{alg:KBO} where the main difference lies in step $(a)$ where $\widehat z_i$ are sampled from $g_\alpha^{\textrm{CW}}$ according to Algorithm \ref{alg:stable_sampling_CW}, while the particles evolve according to \eqref{eq:binary_rules_anisotropic}. 

\begin{remark}
	Note that the multivariate distribution $g^{\textrm{CW}}_\alpha(\cdot)$ differs from the rotationally
	symmetric law $g_\alpha^{\mathrm{RS}}(\cdot)$. Indeed, by independence of the
	components, its characteristic function is
	\[
	\int_{\mathbb R^d}
	e^{i\xi\cdot z}\,
	g_\alpha^{\mathrm{CW}}(z)\,dz
	=
	\exp\left(-\sum_{k=1}^d|\xi_k|^\alpha\right).
	\]
	For $d>1$ and $0<\alpha<2$, this defines a non-rotationally symmetric stable
	process whose generator is given by the sum of one-dimensional fractional
	Laplacians; see, e.g., \cite{chaker2020regularity}. Hence, formally we have that the associated continuous model writes as
	\begin{equation}\label{eq:fFP-CMS}
		\begin{aligned}
			\partial_t f(x,t)
			=&-\nabla_x\cdot
			\Bigl(\nu\bigl(\bar{x}_{\mathcal E}(t)-x\bigr)f(x,t)\Bigr)\cr
			&\quad+\sigma^2\sum_{k=1}^d
			\partial_{x_kx_k}
			\Bigl(
			\bigl(\bar{x}_{\mathcal E,k}(t)-x_k\bigr)^2 f(x,t)
			\Bigr)
			\cr
			&\qquad-\gamma^\alpha\sum_{k=1}^d
			(-\partial_{x_kx_k})^{\alpha/2}
			\Bigl(
			\bigl|\bar{x}_{\mathcal E,k}(t)-x_k\bigr|^\alpha f(x,t)
			\Bigr).
		\end{aligned}
	\end{equation}
	Equation~\eqref{eq:fFP-CMS} provides the formal
	continuous counterpart of the componentwise particle dynamics. 
\end{remark}

\section{Numerical experiments}\label{sec:numerical_experiments} 
We begin with a validation test in a simplified setting, aimed at assessing the accuracy of the numerical method introduced in Section~\ref{sec:numerical_methods}. We then investigate the performance of $\alpha$-KBO on a selection of optimization problems, considering both the success rate and the number of iterations required for convergence, and compare the results with those obtained by the classical diffusion-based KBO method.

In particular, we consider three different $\alpha$-KBO formulations based on Algorithm \ref{alg:KBO}, obtained
by combining the isotropic or anisotropic diffusion matrix with the multivariate constructions of the $\alpha$-stable noise. \begin{itemize}
	\item $\alpha$-KBO$_{\mathrm{iso}}^{\mathrm{RS}}$:~isotropic diffusion matrix \eqref{eq:D_isotropic}, and $\widetilde z_i\sim g^{RS}_\alpha$.
	
	\item $\alpha$-KBO$_{\mathrm{ani}}^{\mathrm{RS}}$:~anisotropic diffusion matrix \eqref{eq:D_anisotropic}, and $\widetilde z_i\sim g^{RS}_\alpha$.
	
	\item $\alpha$-KBO$_{\mathrm{ani}}^{\mathrm{CW}}$:~anisotropic diffusion matrix \eqref{eq:D_anisotropic}, and $\widetilde z_i\sim g^{CW}_\alpha$.
\end{itemize}
\subsection{Validation test}
We will carry out this test in dimension one, in this simplified setting the three approaches coincide. Let us consider the simplified fractional Fokker-Planck equation
\begin{equation}\label{eq:FPE_validation}
	\partial_t f(x,t) = \partial_x (x f(x,t)) - (-\partial_{xx})^{\alpha/2} (x^2 f(x,t)),
\end{equation}
coupled with the initial condition 
\begin{equation}\label{eq:FPE_validation_IC}
	f_0(x) = \frac{1}{\pi}\frac{1}{1+x^2}. 
\end{equation}
Let us notice that equation \eqref{eq:FPE_validation} shares many characteristics with equation \eqref{eq:fFP}.
In the case $\alpha = 1$, it can be proved that equation \eqref{eq:FPE_validation} with initial condition \eqref{eq:FPE_validation_IC} admits an exact solution given by 
\begin{equation}\label{eq:FPE_validation_exact}
	f_{ex}(x,t) = \frac{1}{\pi}\frac{\beta(t)}{\beta^2(t) + x^2}, \qquad \text{with} \quad \beta(t) = \frac{e^{-t}}{2-e^{-t}},
\end{equation}
for any $t\geq 0$. We consider a set of $N=10^6$ particles and we update the particle positions with the rules described in \eqref{eq:binary_rules}, assuming $\nu=1$, $\sigma=0$, $\bar{x}_\mathcal{E} = 0$, $D(\bar{x}_\mathcal{E},x) = x^2$, setting $\alpha = 1$. We let the dynamics evolve up to time $T=2$, with time step $\Delta t = 0.01$. The density $f^N(x,t)$ is reconstructed over a one dimensional grid with $m_x=2^{10}$ points on the interval $[-20,20]$. In Figure \ref{fig:validation}, on the left, three snapshots of the dynamics at time $t=0.1$, $t=1$, and $t=2$.  The numerical and exact solution are in good agreement. On the right, the error between the exact and numerical solution at time $t=2$ computed as 
\begin{equation}\label{eq:validation_error}
	\mathcal{E}^N(T) = \Vert f_{ex}(x,T) - f^N(x,T) \Vert_\infty,
\end{equation}
for different values of $N$. As expected for a Monte Carlo method, the error decreases proportionally to $N^{-1/2}$. 
\begin{figure}[tbhp]
	\centering
	\includegraphics[width=0.34\linewidth]{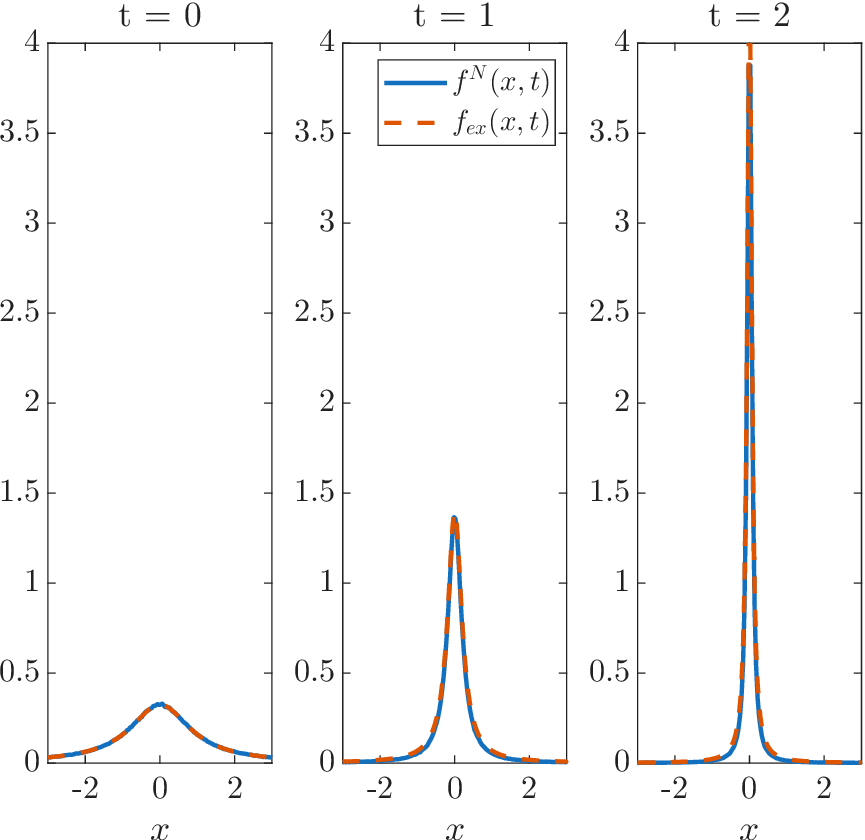}
	\includegraphics[width=0.35\linewidth]{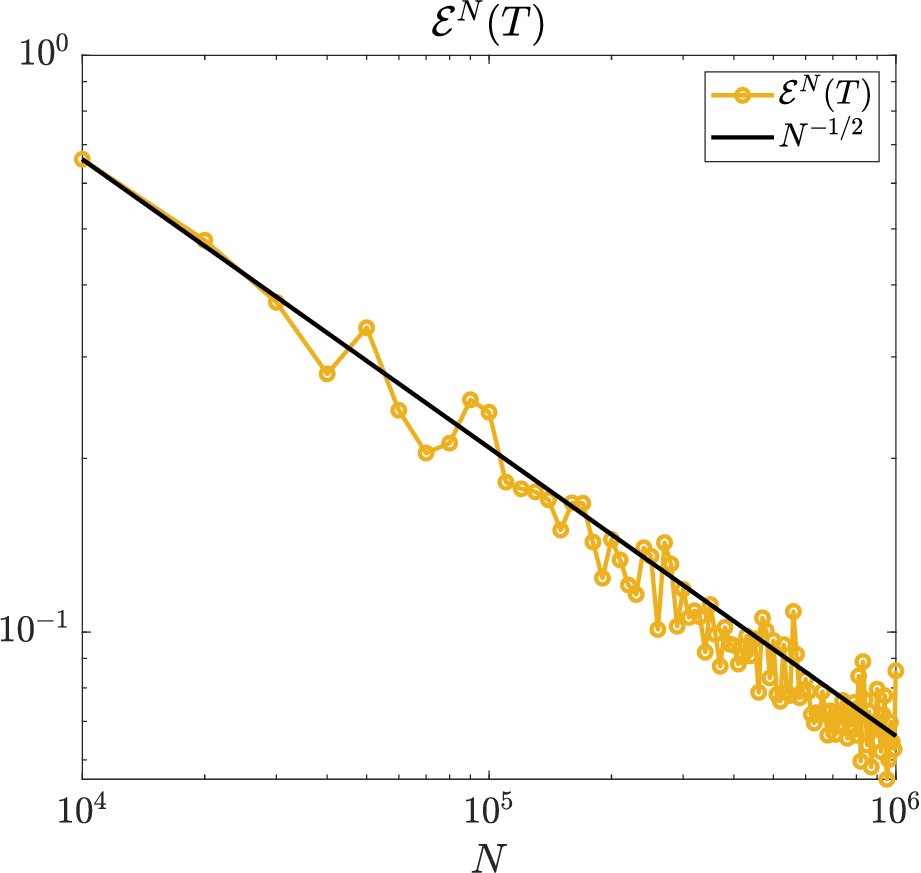}
	\caption{ Validation test: on the left three snapshots of the numerical and exact solution as in \eqref{eq:FPE_validation_exact} for $N=10^6$ taken at time $t=0.1$, $t=1$, and $t=2$. On the right, the error as the number of particles $N$ increases between the exact solution in \eqref{eq:FPE_validation_exact} and its reconstruction, computed according to \eqref{eq:validation_error}.
	}
	\label{fig:validation}
\end{figure}
\subsection{KBO with fractional diffusion}
In this section, we investigate the efficiency of the KBO method driven by a  $\alpha$-stable process. Unless otherwise specified, we consider the Rastrigin function
\begin{equation}\label{eq:Rastrigin}
	\mathcal{E}(x) = 10 + \sum_{k=1}^d (x_k^2 - 10\cos(2\pi x_k)),
\end{equation} 
\noindent as a benchmark problem, which admits a global minimizer at $x^\star = 0$.
We run $M$ simulations, and we consider one successful if 
\begin{equation}\label{eq:succ_estimate}
	\Vert \bar{x}_\mathcal{E}(t) - x^\star \Vert_\infty \leq 0.25,
\end{equation}
where $\Vert \cdot \Vert_\infty $ denotes the $L_\infty$ norm with respect to the dimension. 
Unless specified otherwise, we set the baseline parameters and initial condition as follows. We choose $\beta = 5 \cdot 10^6$, adopting the numerical strategy proposed in \cite{fornasier2021consensus} to accommodate arbitrarily large values of $\beta$. We set $N=200$ and initialize the agents within the hypercube $[-5.12,-2]^d$, which excludes the global minimum. Although these experiments fall outside the initial-support assumption of Theorem \ref{thm:main}, this configuration is chosen intentionally to verify numerically that the proposed method converges even when $x^\star \notin \text{supp}(f_0)$. The dynamics are evolved over $N_t = 10^4$ iterations with a time step of $\Delta t=0.1$, using stalling parameters $j_{\text{stall}} = 10^3$ and $\delta_{\text{stall}} = 10^{-4}$. Across all tests, we fix $\nu = 1$ and $\alpha = 1.5$. Remaining parameters, including the spatial dimension $d$, vary depending on the experiment and are specified in the corresponding sections.
\paragraph{Comparative test.} To assess the impact of the stability parameter $\alpha$ and the influence of the random vector generation scheme on the convergence behavior of the KBO with a pure $\alpha$-stable process, we first perform a preliminary numerical experiment. 
We fix $d=20$, and we consider the non-convex Ackley function defined as 
\begin{equation}\label{eq:ackley}
	\mathcal{E}(x) = -20 \exp\left( -0.2 \sqrt{\frac{1}{d} \sum_{j=1}^d x_j^2} \right) - \exp\left( \frac{1}{d} \sum_{j=1}^d \cos(2\pi x_j) \right) + 20.
\end{equation}
We initialize the agents' positions uniformly in the interval $[-3,3]^d$, ensuring that the support of the initial distribution contains the global minimizer, and we set the parameters to $\nu=1$, $\sigma=0$, and $\gamma = 0.1$, satisfying the requirements for convergence established in Theorem \ref{thm:main}, and recalled in Remark \ref{rmk:convergence_issues}. To evaluate the interplay between spatial diffusion weighting and heavy-tailed vector sampling, we compare three KBO variants, as $\alpha$ varies in the interval $[0.1,2]$. Specifically we consider $\alpha$-KBO$_{\mathrm{iso}}^{\mathrm{RS}}$, which couples an isotropic diffusion matrix \eqref{eq:D_isotropic} with rotationally symmetric vector sampling $\widetilde z_i \sim g^{\mathrm{RS}}_\alpha$ as in Algorithm~\ref{alg:stable_sampling}; $\alpha$-KBO$_{\mathrm{ani}}^{\mathrm{RS}}$, which pairs an anisotropic diffusion matrix \eqref{eq:D_anisotropic} with the same rotationally symmetric sampling; and $\alpha$-KBO$_{\mathrm{ani}}^{\mathrm{CW}}$, which combines an anisotropic matrix \eqref{eq:D_anisotropic} with component-wise sampling $\widetilde z_i \sim g^{\mathrm{CW}}_\alpha$ generated independently along each dimension via the one-dimensional CMS method detailed in Algorithm \ref{alg:stable_sampling_CW}. Initial tests over $M=100$ independent runs yielded statistical results virtually identical to those obtained with $M=20$; hence, we fix $M=20$ in all remaining experiments to optimize computational efficiency.  Figure \ref{fig:comparison_ackley} displays the success rate on the left and the iteration number on the right as $\alpha$ varies.
\begin{figure}[tbhp]
	\centering
	\includegraphics[width=0.327\linewidth]{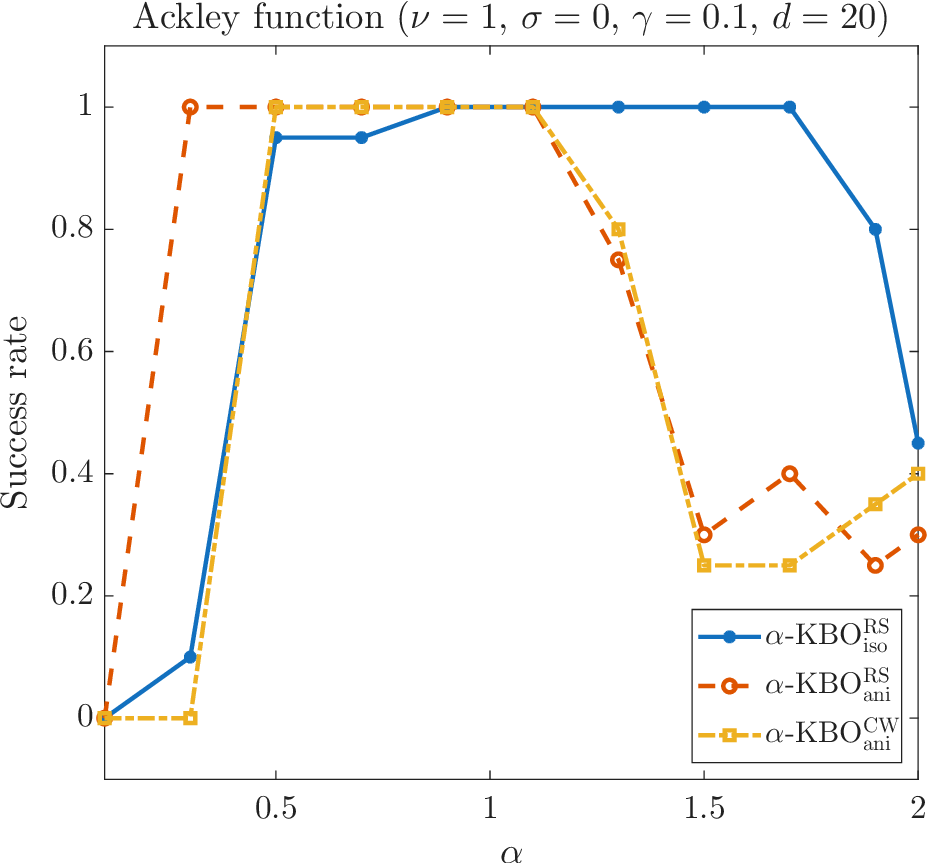} \includegraphics[width=0.335\linewidth]{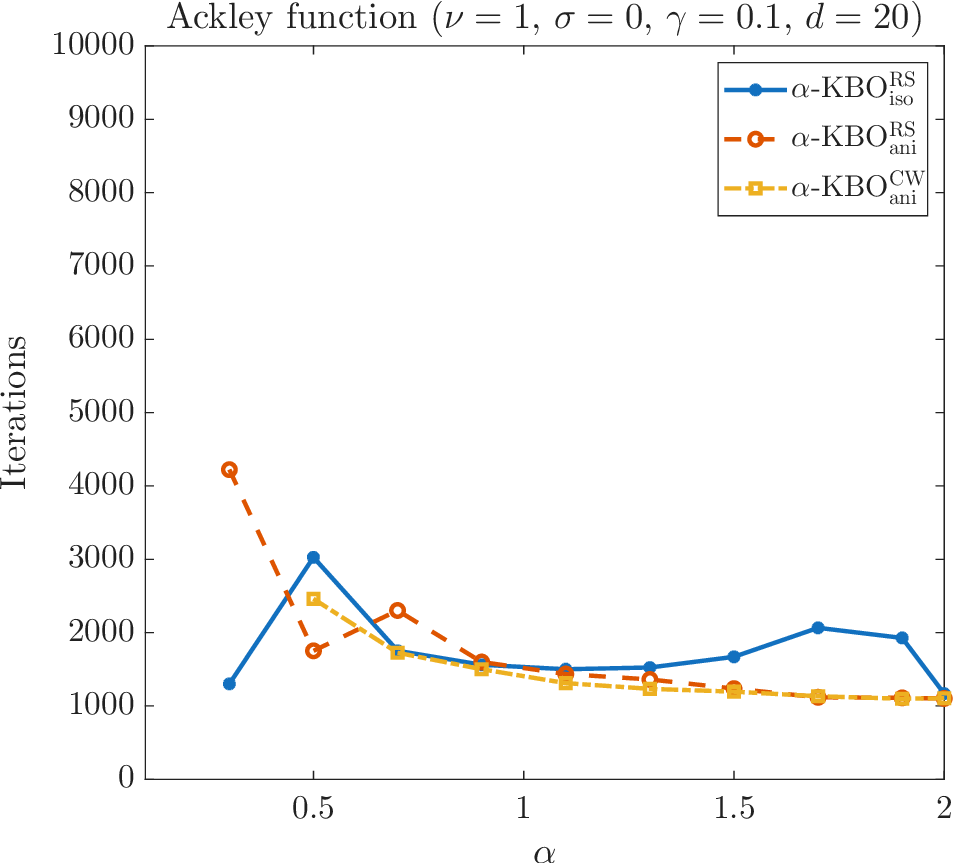}
	\caption{   Comparison of success rate (left) and mean number of iterations (right) for the Ackley function for fixed $d=20$, $\nu=1$, $\sigma = 0$, $\gamma = 0.1$ as the parameter $\alpha$ varies, evaluating the rotationally symmetric sampling procedure described in Algorithm \ref{alg:stable_sampling} on both the isotropic and anisotropic settings ($\alpha$-KBO$_{\mathrm{iso}}^{\mathrm{RS}}$, $\alpha$-KBO$_{\mathrm{ani}}^{\mathrm{RS}}$) against the component-wise CMS scheme in Algorithm \ref{alg:stable_sampling_CW} ($\alpha$-KBO$_{\mathrm{ani}}^{\mathrm{CW}}$) in the anisotropic setting. Markers denote the value of the success rate and the number of iterations.} 
	\label{fig:comparison_ackley}
\end{figure} 
The ability of isotropic diffusion and rotationally symmetric sampling to achieve  convergence on the Ackley benchmark is largely due to its non-separable, macro-radially symmetric structure, allowing isotropic perturbations to guide agents toward the global minimizer. In contrast, on benchmark functions with a strictly separable structure, such as the Rastrigin function, isotropic diffusion suffers in high dimensions because it drives agents away from coordinate-wise consensus. Consequently, handling separable problems necessitates choosing the diffusion matrix $D(\bar{x}_\mathcal{E},x)$ in anisotropic form as in \eqref{eq:D_anisotropic} to decouple coordinate dynamics and avoid dimension-dependent convergence issues (see Remark \ref{rmk:convergence_issues}). We therefore focus on this anisotropic formulation in all remaining numerical experiments.

Accordingly, we apply this anisotropic framework to the Rastrigin benchmark to compare the rotationally symmetric and component-wise sampling schemes across $\alpha \in [0.1,2]$, fixing $d = 20$, $\sigma = 0$, and $\gamma = 2$. We recall that $\alpha = 2$ recovers the classical Brownian diffusion regime, which corresponds to setting $\sigma=2$ and $\gamma=0$ (see Fig. \ref{fig:rastr_sigmarange}). The initial agents positions are sampled uniformly within the hypercube $[-5.12,-2]^d$.   Within the anisotropic diffusion setting, Figure~\ref{fig:CBO_alfavaries} illustrates the performance comparison between rotationally symmetric sampling ($\alpha$-KBO$_{\mathrm{ani}}^{\mathrm{RS}}$, Algorithm~\ref{alg:stable_sampling}) and component-wise CMS sampling ($\alpha$-KBO$_{\mathrm{ani}}^{\mathrm{CW}}$, Algorithm~\ref{alg:stable_sampling_CW}), displaying the success rate on the left and the required iteration count on the right as a function of $\alpha$. The optimal trade-off is achieved by the CMS algorithm for $\alpha \ge 1.5$, where the success rate is maximized and the iteration count is minimized. Consequently, we fix $\alpha = 1.5$ and we adopt the CMS method for all remaining tests.

We emphasize that the numerical parameter choices considered in this last case do not satisfy the sufficient contraction condition of Theorem~\ref{thm:main}, instead this choice explores the effect of $\alpha$-stable perturbations beyond the theoretical regime of Section~3. 
\begin{figure}[tbhp]
	\centering
	\includegraphics[width=0.327\linewidth]{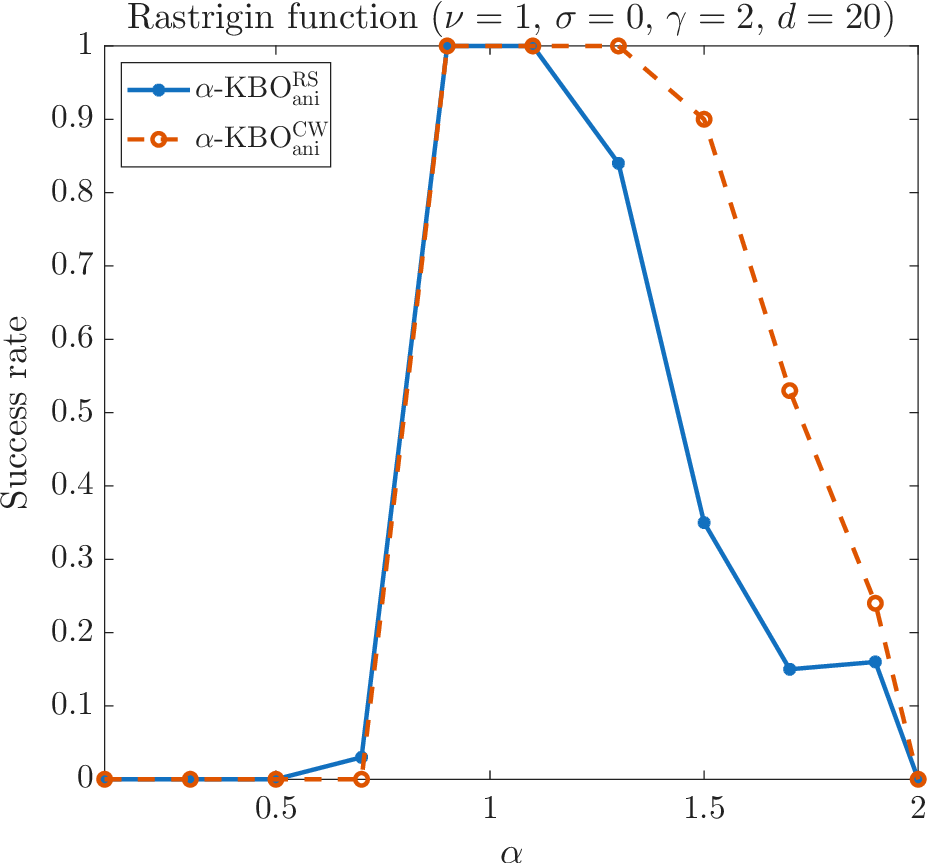} 
	\includegraphics[width=0.335\linewidth]{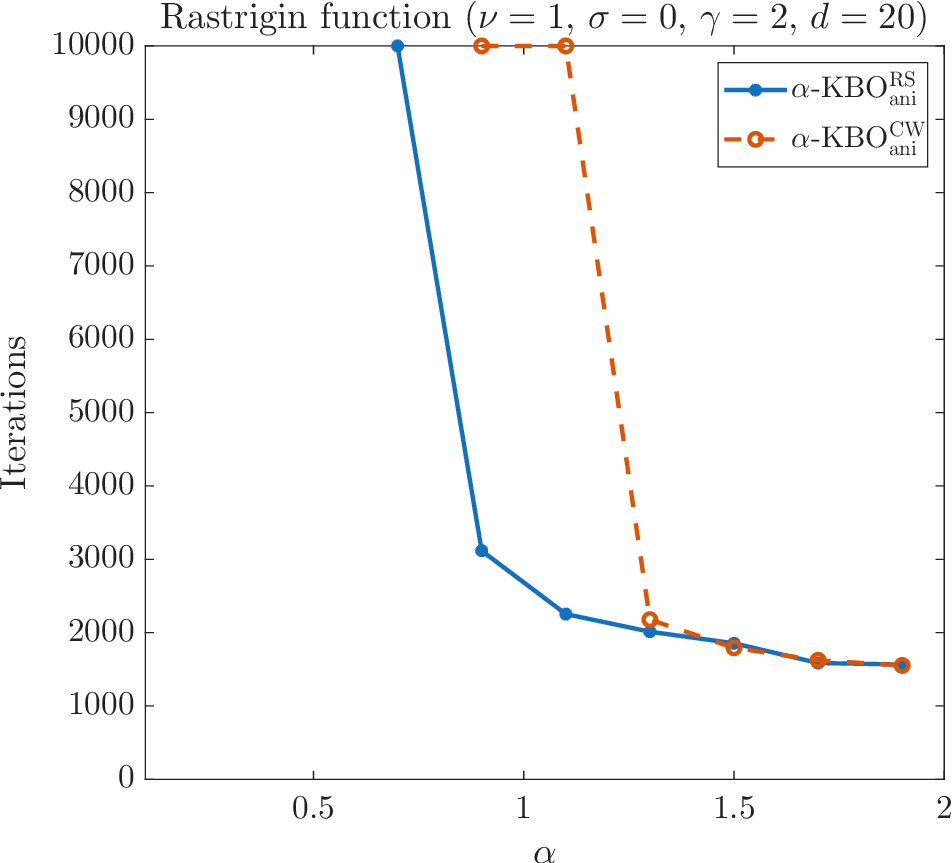}
	\caption{  Comparison of success rate (left) and mean number of iterations (right) for the Rastrigin function for fixed $d=20$, $\nu=1$, $\sigma = 0$, $\gamma = 2$ as the parameter $\alpha$ varies, evaluating the rotationally symmetric sampling procedure described in Algorithm \ref{alg:stable_sampling} ($\alpha$-KBO$_{\mathrm{ani}}^{\mathrm{RS}}$) against the   component-wise CMS scheme in Algorithm \ref{alg:stable_sampling_CW} ($\alpha$-KBO$_{\mathrm{ani}}^{\mathrm{CW}}$). Markers denote the value of the success rate and the number of iterations.}
	\label{fig:CBO_alfavaries}
\end{figure} 
\paragraph{Qualitative test.} We fix $d=2$, and we consider the non-convex non-differentiable function modified Alpine defined as 
\begin{equation}\label{eq:mod_alpine}
	\mathcal{E}(x) = \sum_{k=1}^d(\abs{x_k \sin(x_k)}+0.2\abs{x_k}).
\end{equation}
We consider both the case with purely classical diffusion, corresponding to $\sigma = 0.5$ and $\gamma = 0$, and the case combining classical and fractional diffusion, with $\sigma = \gamma = 0.5$. In dimension $d = 2$, the regime with only jumps ($\sigma = 0$, $\gamma = 0.5$) exhibits a behavior similar to the mixed diffusion case, and is therefore omitted for brevity.
In Figure \ref{fig:alpine_dynamics}, we report three snapshots at iterations $n = 10, 50, 200$ for both the classical KBO dynamics and the KBO dynamics with classical and fractional diffusion. In the latter case, particles exhibit a stronger exploratory behavior and progressively concentrate around the global minimizer. In contrast, in the purely classical setting, they tend to become trapped in a local minimum.
\begin{figure}[tbhp]
	\centering
	\includegraphics[width=0.327\linewidth]{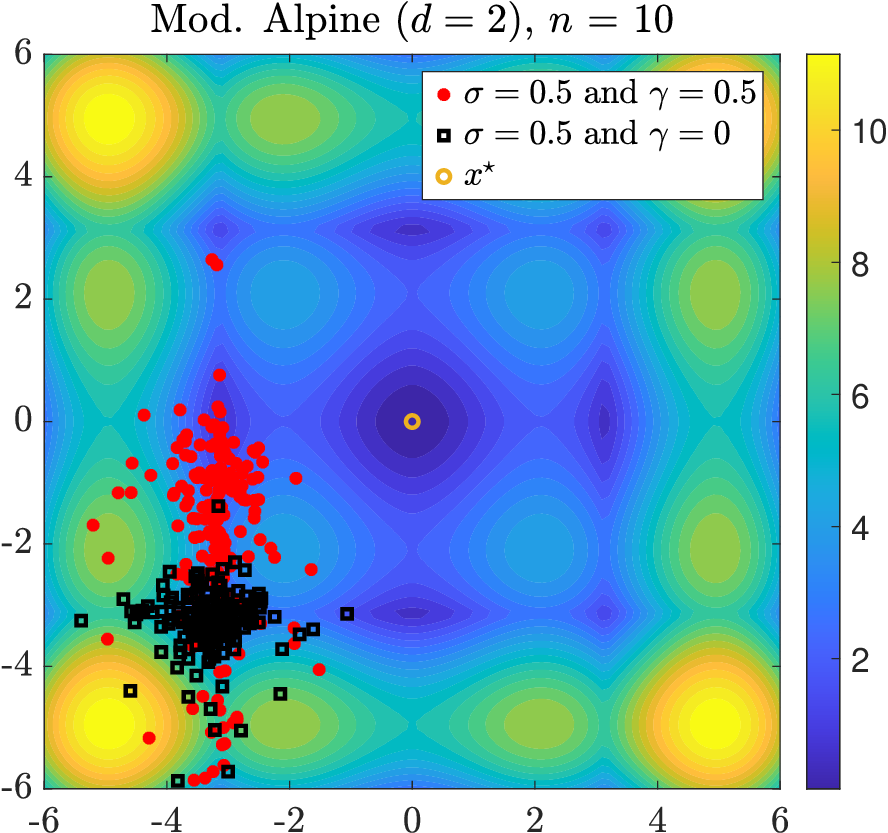}
	\includegraphics[width=0.327\linewidth]{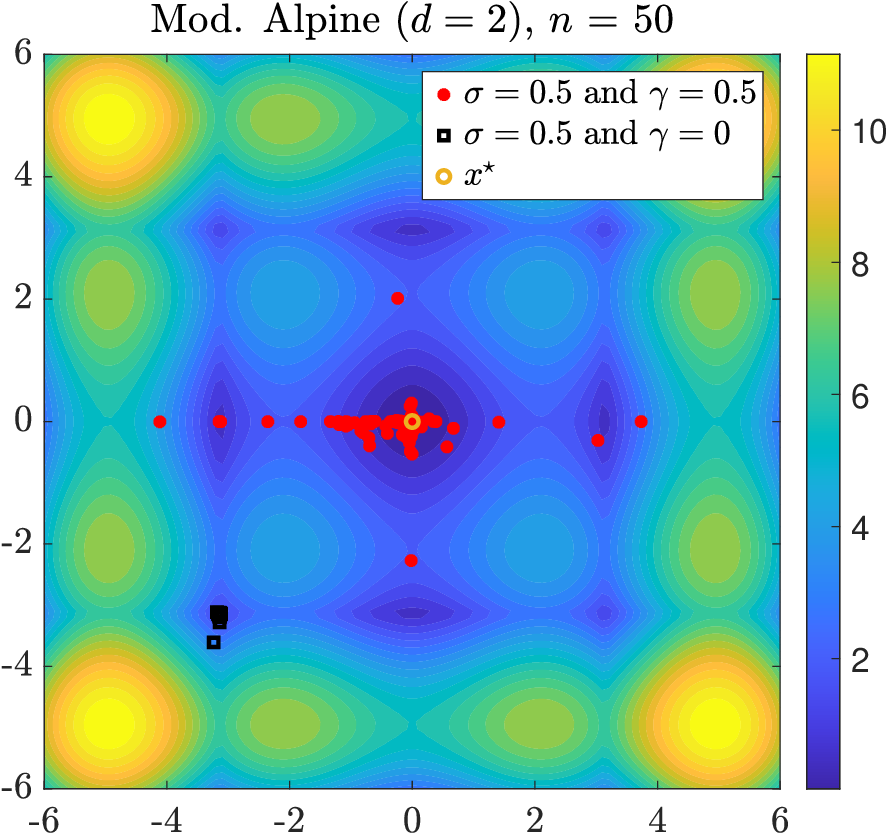}
	\includegraphics[width=0.327\linewidth]{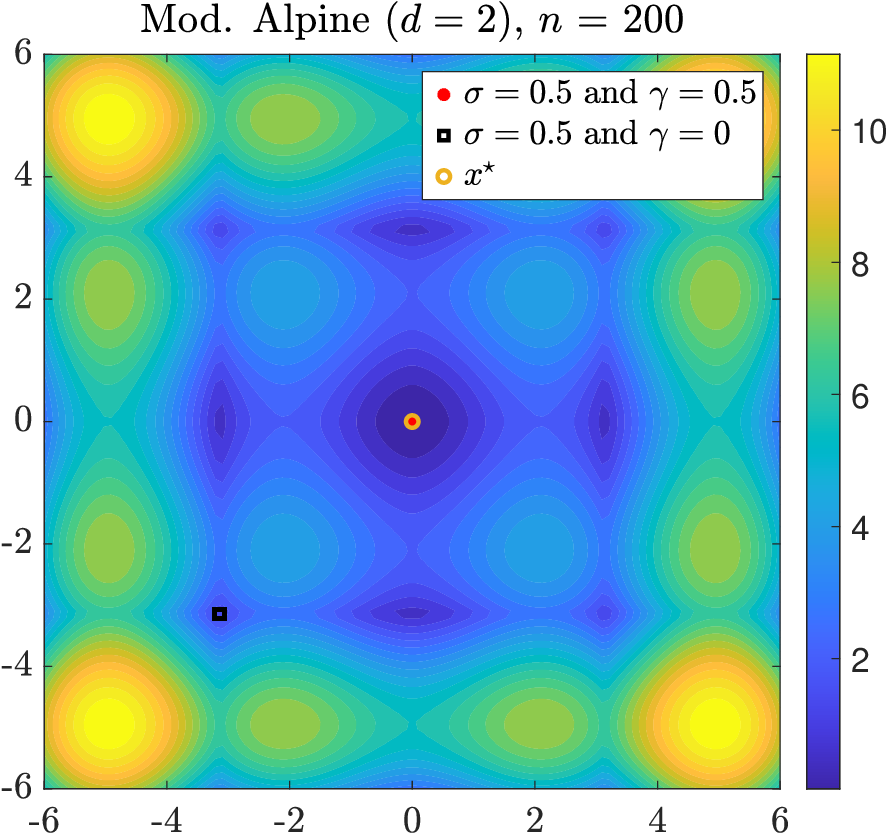}
	\caption{ Modified Alpine function: three snapshots of the KBO evolutions taken at iterations $n=10, 50, 200$ for fixed $d=2$, $\nu=1$. We set either $\sigma=0.5$ and $\gamma=0$ or $\sigma=0.5$ and $\gamma =0.5$ to model the case with pure classical diffusion and the one with both fractional and classical diffusion.  }
	\label{fig:alpine_dynamics}
\end{figure} 

\paragraph{Test 1: Comparison in dimension $d=20$ for varying $\gamma$.}
We fix $d = 20$, and let $\gamma$ vary between $\gamma = 1$ to $\gamma = 5$, testing the method on the Rastrigin function. We consider the KBO with $\alpha$-stable process setting $\sigma = 0$ and $\sigma = 2$ for gaussian process. In Figure \ref{fig:rastr_gammarange} on the left, the plot of the success rate as $\gamma$ varies, and on the right, the plot of the iterations number. The method proved to be particularly effective, especially in the pure-jumps case ($\sigma = 0$) with $\gamma = 2.5$, where it achieves a higher success rate while requiring fewer iterations.

\begin{figure}[tbhp]
	\centering
	\includegraphics[width=0.327\linewidth]{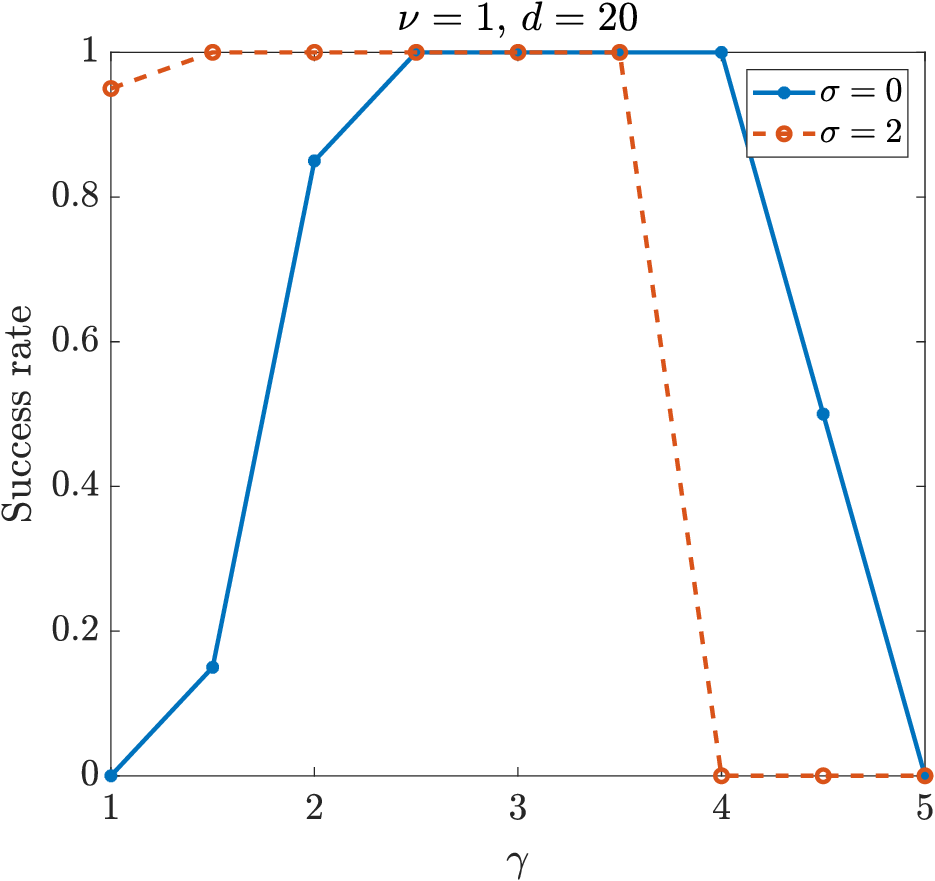}		\includegraphics[width=0.327\linewidth]{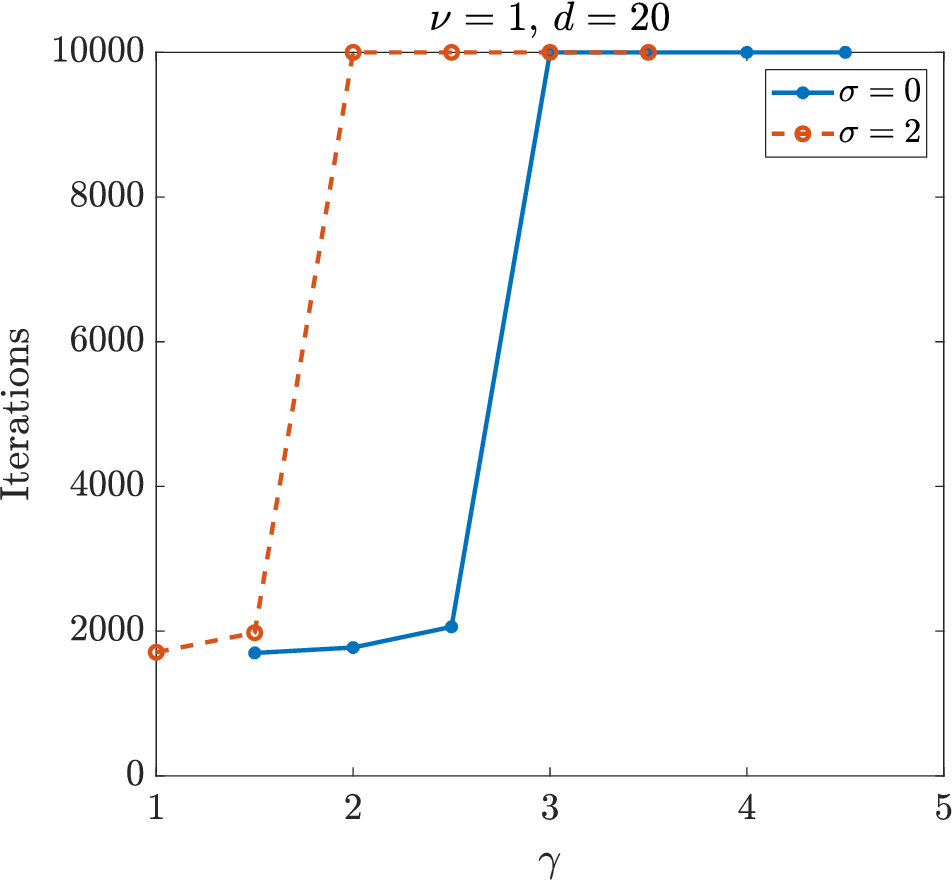}
	\caption{ Success rate and mean number of iterations for the Rastrigin function for fixed $d=20$, $\nu=1$ as the parameter $\gamma$ varies. We set either $\sigma=0$ or $\sigma=2$ to model the case with a pure $\alpha$-stable process and the one with both fractional and classical diffusion.  Markers denote the value of the success rate and the number of iterations.
	}
	\label{fig:rastr_gammarange}
\end{figure}

\paragraph{Test 2: Comparison for varying dimension $d$.}
We let the dimension $d$ vary between $d=1$ and $d=50$, and we make different comparisons changing the values of $\sigma$ and $\gamma$, and testing the methods on the Rastrigin function. In Figure \ref{fig:rastr_drange} on the left, the plot of the success rate as $d$ varies, and on the right, the plot of the number of iterations. The KBO variant incorporating both an $\alpha$-stable process and diffusion ($\gamma=2$ and $\sigma=2$) demonstrates the highest effectiveness in terms of success rate. In contrast, the pure diffusion case ($\gamma=0$ and $\sigma=2$) achieves convergence in fewer iterations but exhibits a smaller success rate area. The best compromise is provided by the case with only the $\alpha$-stable process ($\gamma=2$ and $\sigma=0$), which combines a large success rate area with a reduced number of iterations.  For this choice of the parameters, the KBO with $\alpha$-stable process proves to be more effective than the classical KBO.   

\begin{figure}[tbhp]
	\centering
	\includegraphics[width=0.327\linewidth]{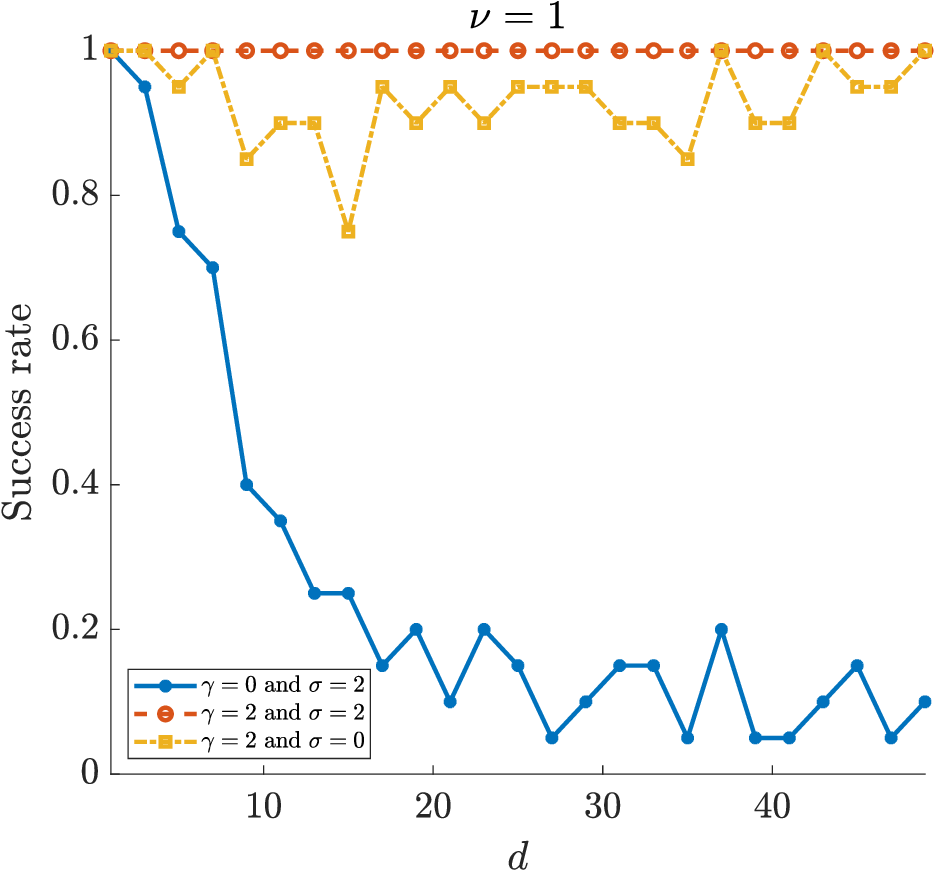}
	\includegraphics[width=0.327\linewidth]{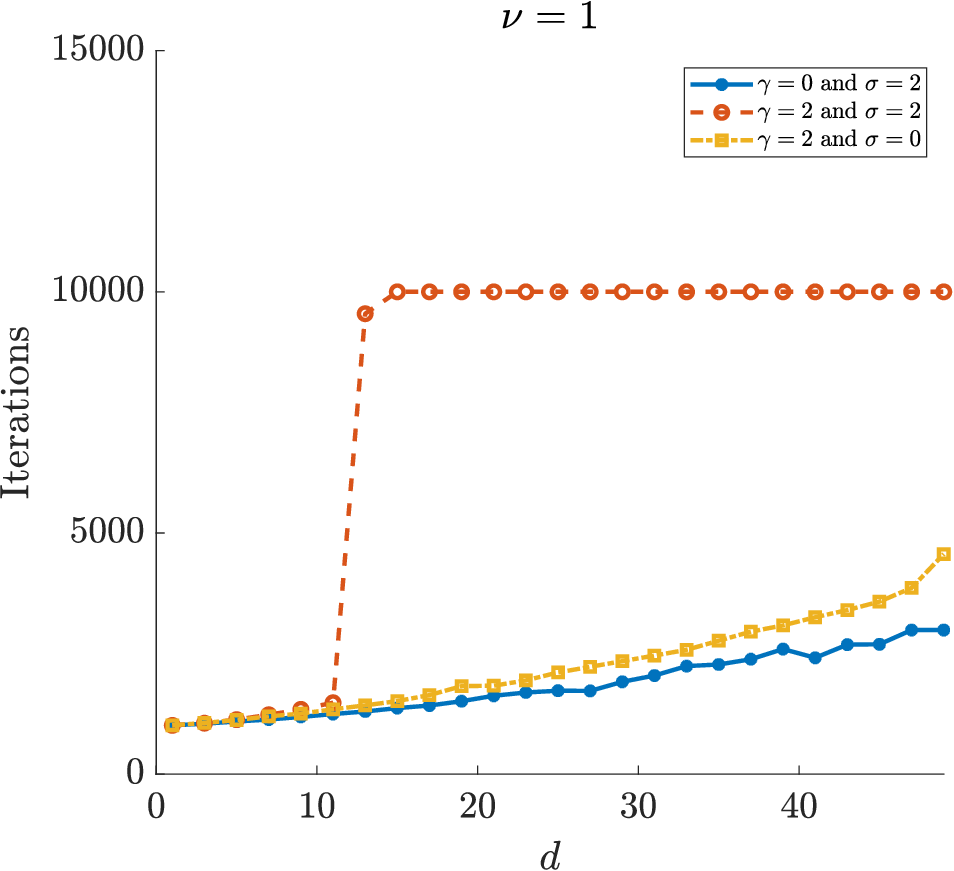} 
	\caption{Success rate and mean number of iterations for the Rastrigin function as the dimension $d$ varies. We consider pure $\alpha$-stable process setting $\gamma =2$ and $\sigma = 0$, pure diffusion setting $\gamma=0$ and $\sigma = 2$, and both diffusion and $\alpha$-stable process, setting $\gamma =2$ and $\sigma = 2$. Markers denote the value of the success rate and the number of iterations.
	}
	\label{fig:rastr_drange}
\end{figure}
\newpage
\paragraph{Test 3: Comparison in dimension $d=20$ for varying $\sigma$.}
We fix $d = 20$, and let $\sigma$ to vary between $\sigma = 0$ to $\sigma = 6$, testing the method on the Rastrigin function. We consider the KBO with pure diffusion setting $\gamma = 0$, and the case with both diffusion and $\alpha$-stable process, setting $\gamma = 2$. In Figure \ref{fig:rastr_sigmarange} on the left, the plot of the success rate as $\sigma$ varies, and on the right, the plot of the number of iterations. The KBO with both diffusion and $\alpha$-stable process proved to be more effective than the one with only diffusion, showing a larger success rate area.

\begin{figure}[tbhp]
	\centering 	
	\includegraphics[width=0.327\linewidth]{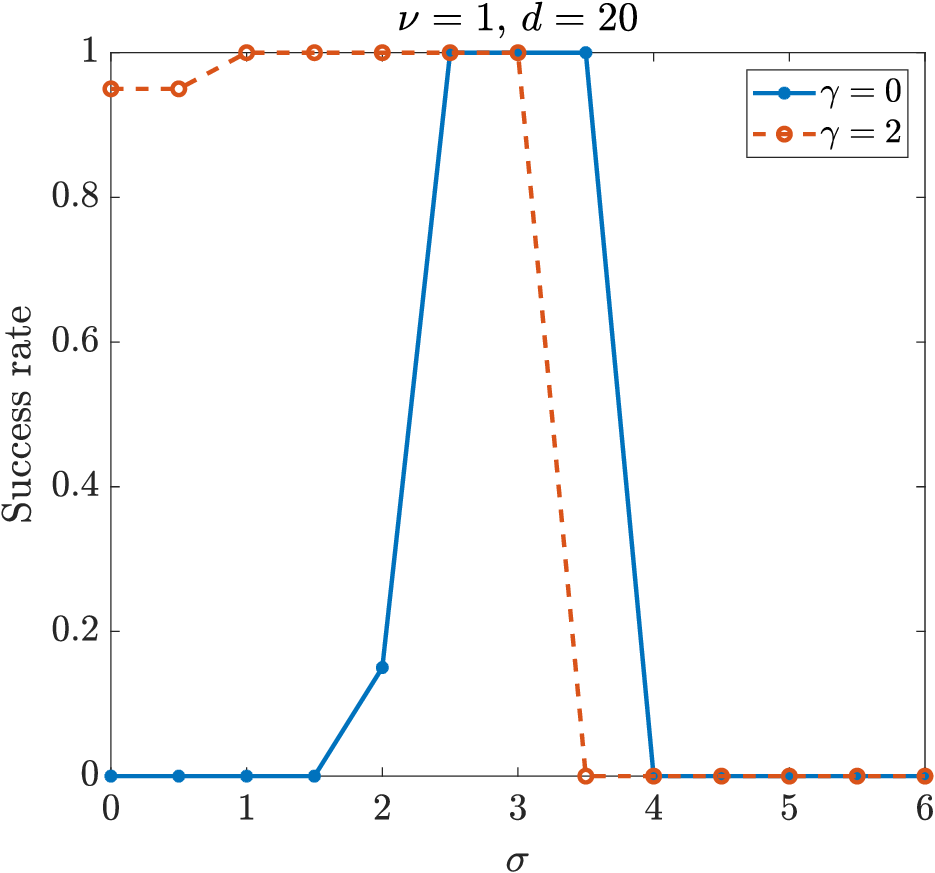}
	\includegraphics[width=0.327\linewidth]{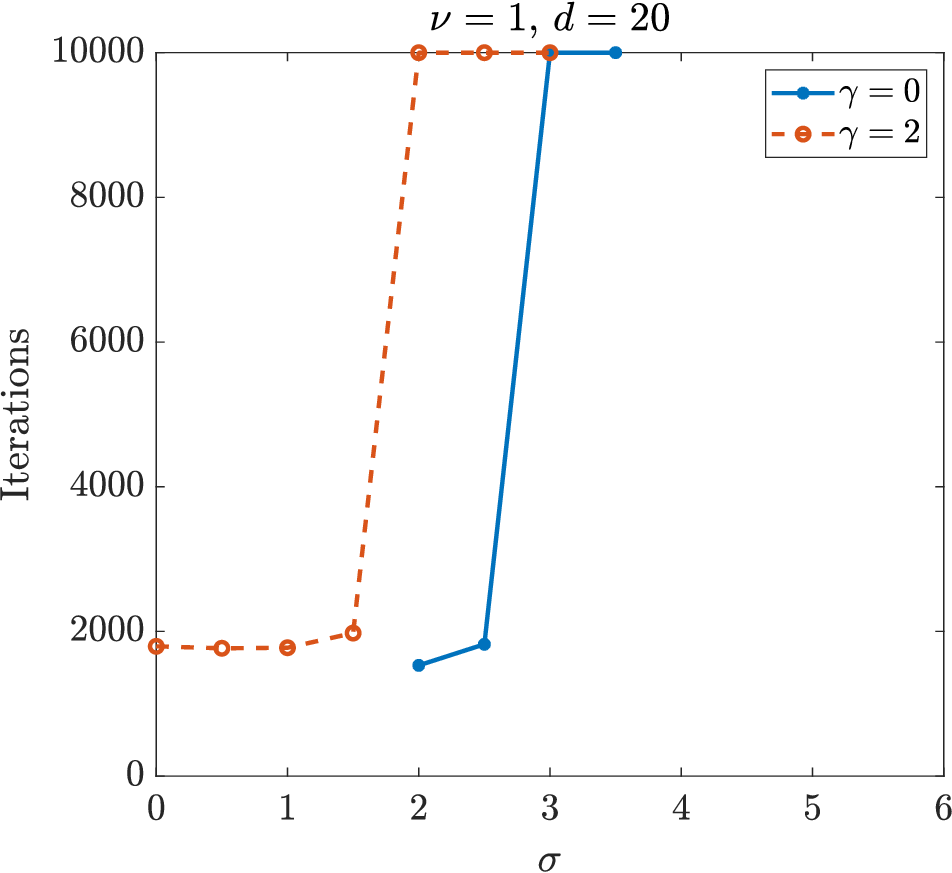} 
	\caption{Success rate and mean number of iterations for the Rastrigin function for fixed $d=20$, $\nu=1$ as the parameter $\sigma$ varies. The KBO with pure diffusion process is considered, setting $\gamma =0$, and the one with both diffusion and $\alpha$-stable, setting $\gamma = 2$.  Markers denote the value of the success rate and the number of iterations.
	}
	\label{fig:rastr_sigmarange}
\end{figure}

\paragraph{Test 4: Comparison of different benchmark functions.}\label{sec:test7}
We fix $d = 20$ and $\nu = 1$. We consider different differentiable and convex benchmark functions (see \cite{jamil2013literature} for a complete list), and we test the KBO method in the case with pure diffusion ($\gamma = 0$ and $\sigma = 2$), in the case with pure $\alpha$-stable process ($\gamma = 2$ and $\sigma = 0$), and in the case with both $\alpha$-stable process and diffusion ($\gamma = 2$ and $\sigma = 2$).  
Figure \ref{fig:funct_in0} reports a comparison in terms of success rate and mean number of iterations. For the different parameter choices, the KBO method with pure classical diffusion, pure jumps, and combined jumps and diffusion achieves a higher success rate. The fractional KBO outperforms the classical version for the Rastrigin function, whereas for the Rosenbrock function, the KBO with classical diffusion attains a higher success rate. Moreover, methods employing either pure jumps or pure diffusion converge faster than the approach combining classical diffusion and jumps. 
\begin{figure}[tbhp]
	\centering
	\includegraphics[width=0.327\linewidth]{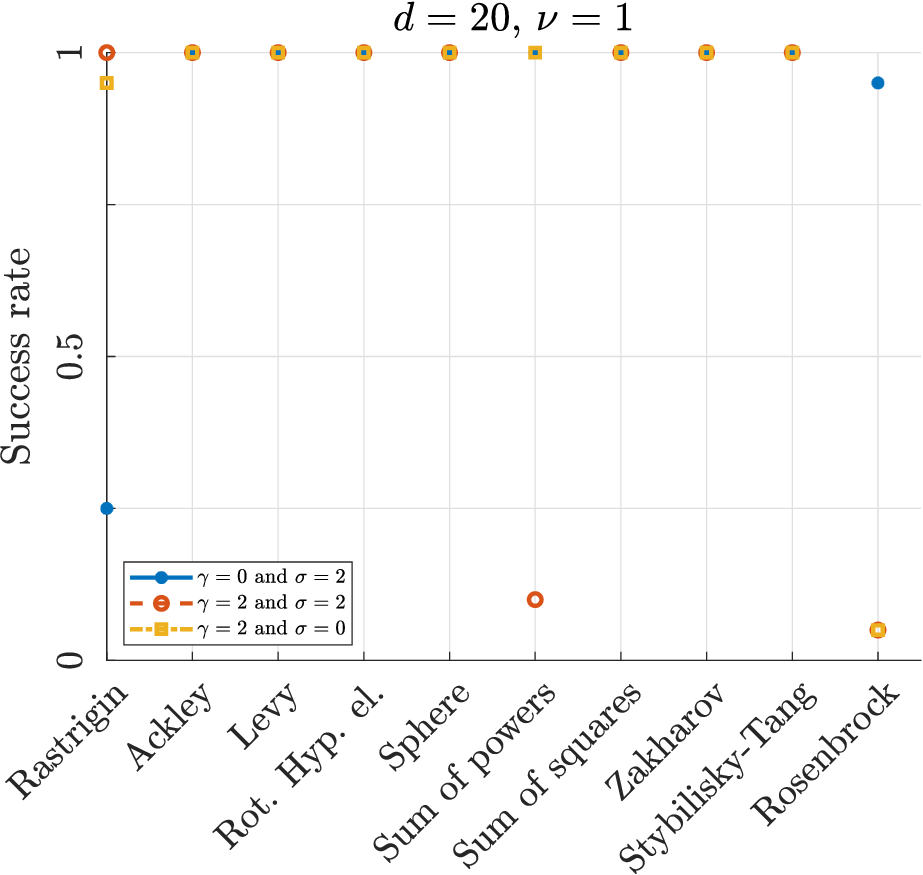}
	\includegraphics[width=0.327\linewidth]{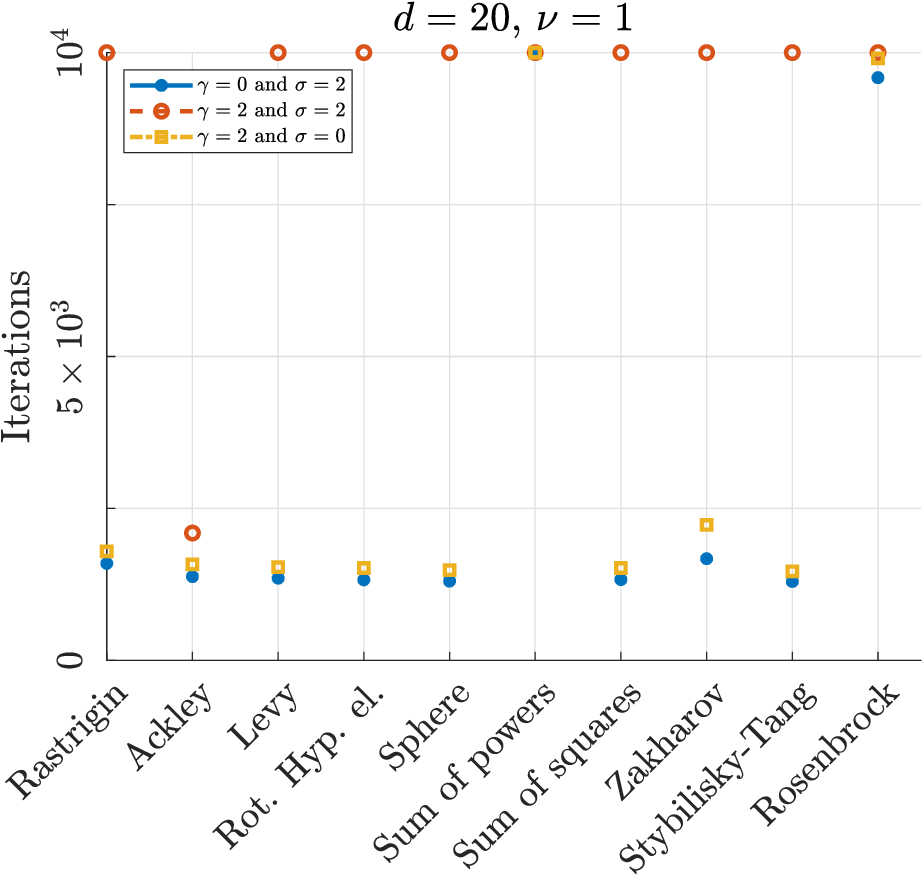} 
	\caption{ Success rate and mean number of iterations for the different benchmark functions for fixed $d=20$.  We consider pure diffusion ($\gamma = 0$, and $\sigma = 2$), pure  $\alpha$-stable process ($\gamma = 2$ and $\sigma = 0$), and both $\alpha$-stable process and diffusion ($\gamma = 2$ and $\sigma = 2$). Markers denote the value of the success rate and the number of iterations for the different benchmark functions.
	}
	\label{fig:funct_in0}
\end{figure}
We now test the method on several non-differentiable functions, some of which are also non-convex (see \cite{karmitsa2007test} for a complete list).  Since the method is gradient-free, convergence is expected. This is confirmed by the results shown in Figure \ref{fig:funct_nondiff}, where the success rate (left) and the number of iterations to convergence (right) are reported. All three variants exhibit a high success-rate region and a comparable number of iterations required to reach convergence.
\begin{figure}[tbhp]
	\centering
	\includegraphics[width=0.327\linewidth]{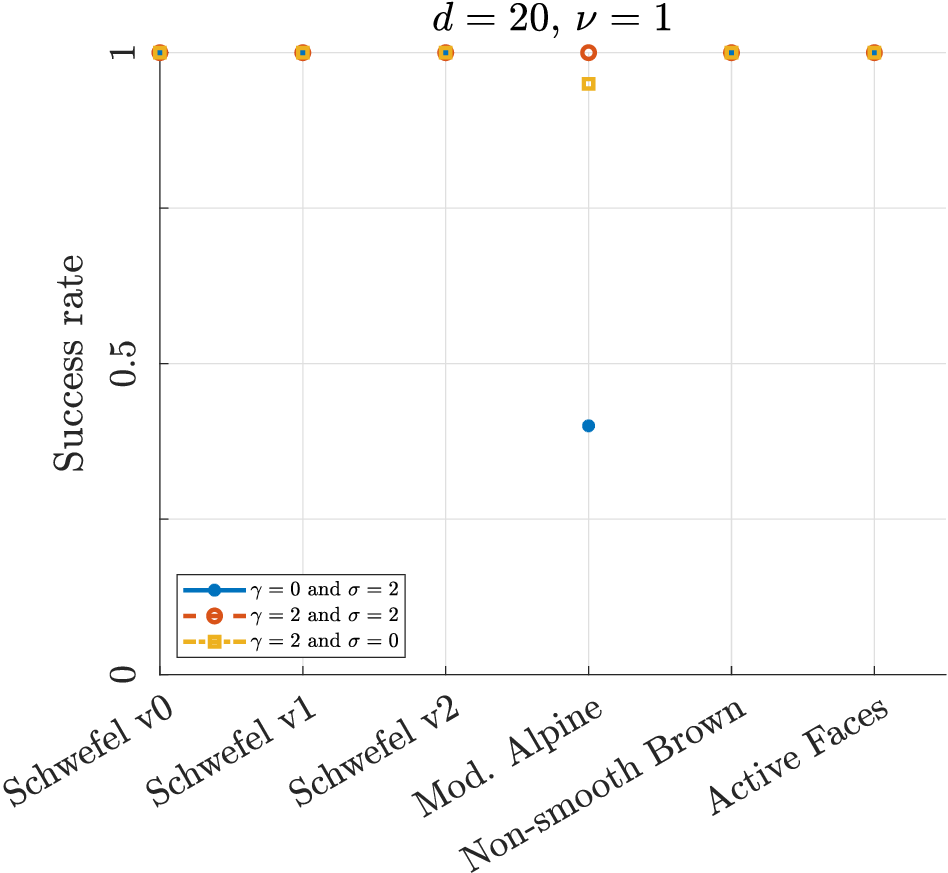}
	\includegraphics[width=0.327\linewidth]{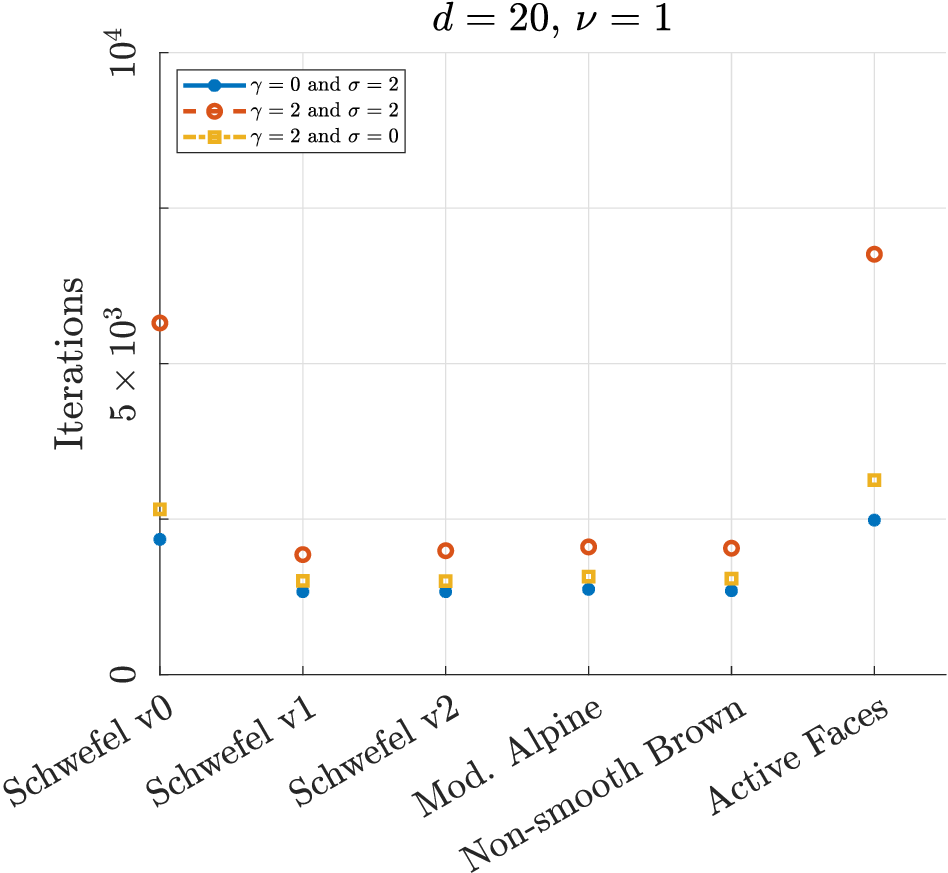} 
	\caption{Success rate and mean number of iterations for the different benchmark functions for fixed $d=20$.  We consider pure diffusion ($\gamma = 0$, and $\sigma = 2$), pure $\alpha$-stable process ($\gamma = 1.5$ and $\sigma = 0$), and both $\alpha$-stable process and diffusion ($\gamma = 2$ and $\sigma = 2$). Markers denote the value of the success rate and the number of iterations for the different benchmark functions.
	}
	\label{fig:funct_nondiff}
\end{figure}

\paragraph{Test 5: Comparison with compound-Poisson dynamics.} We conclude this section by comparing the $\alpha$-stable KBO with the KBO with compound-Poisson dynamics introduced in \cite{kalise2023consensus} and reported in equation \eqref{eq:jump_diffusion}.  We evolve the dynamics up to time $T=100$ that is for $N_t=10^4$ iterations with a time step $\Delta t = 0.01$, considering the Rastrigin function. We set $\beta=20$, $d=20$, and $\sigma = 5.1$, and define the parameter $\gamma(t)$ as\begin{equation}\label{eq:gamma}
	\gamma(t) =
	\begin{cases}\tilde{\gamma}, & \text{if } t \leq t_s, \\
		\tilde{\gamma}\exp(-\frac{t - t_s}{t_s}), & \text{otherwise},
	\end{cases}
\end{equation}
where $t = n \Delta t$, $n$ denotes the iteration index, $t_s \in [20,100]$, and $\tilde{\gamma}$ is a fixed number in the interval $[1,5]$. For the compound Poisson process, we select a jump intensity of $\lambda = 20$. The agents are initially distributed within the hypercube $[-6,6]^d$. The top row of Figure \ref{fig:CBO_jumpsvsalfa} displays the success rate (left) and total iteration count (right) as $N$ varies within $[20,100]$ for $\gamma=\gamma(t)$ defined as in \eqref{eq:gamma} with $\tilde{\gamma}=1$ and $t_s=20$. The KBO driven by an $\alpha$-stable process outperforms the compound Poisson formulation in success rate, whereas both methods achieve convergence within a comparable number of iterations. The bottom row of Figure \ref{fig:CBO_jumpsvsalfa} illustrates the success rate for KBO using an $\alpha$-stable process (left) and a compound Poisson process (right) for  $N=50$, $\gamma=\gamma(t)$ defined as in \eqref{eq:gamma} with $\tilde{\gamma}\in [1,5]$ and $t_s\in [20,100]$. Setting $t_s=100$ corresponds to a constant, time-independent $\gamma = \tilde{\gamma}$. In all configurations, the $\alpha$-stable KBO consistently achieves higher success rates than the compound Poisson formulation. Because the iteration counts between the two methods remain similar, the corresponding plots are omitted for brevity.  
\begin{figure}[tbhp]
	\centering
	\includegraphics[width=0.327\linewidth]{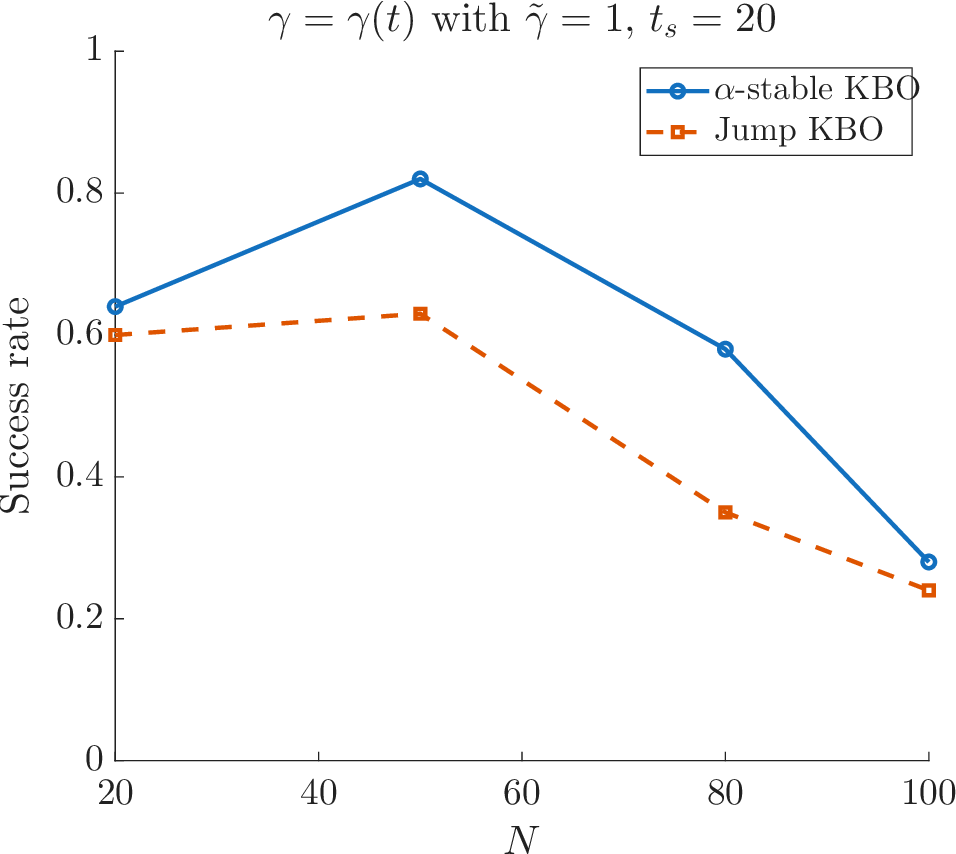} \includegraphics[width=0.327\linewidth]{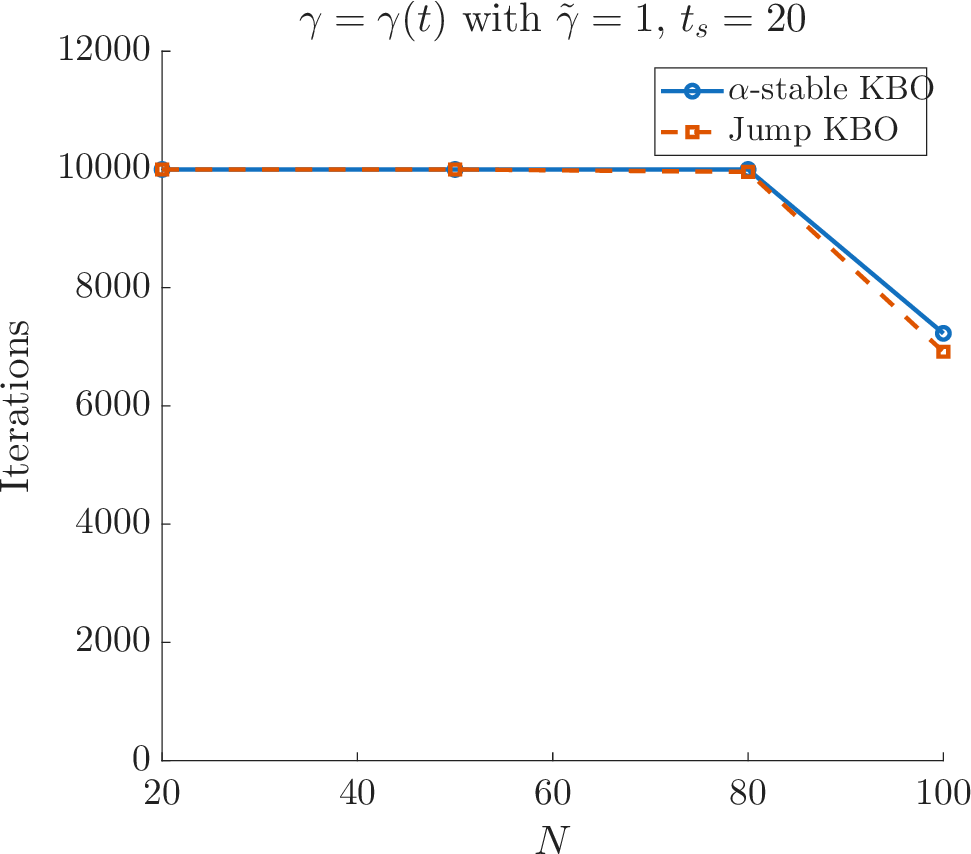}\\
	\includegraphics[width=0.327\linewidth]{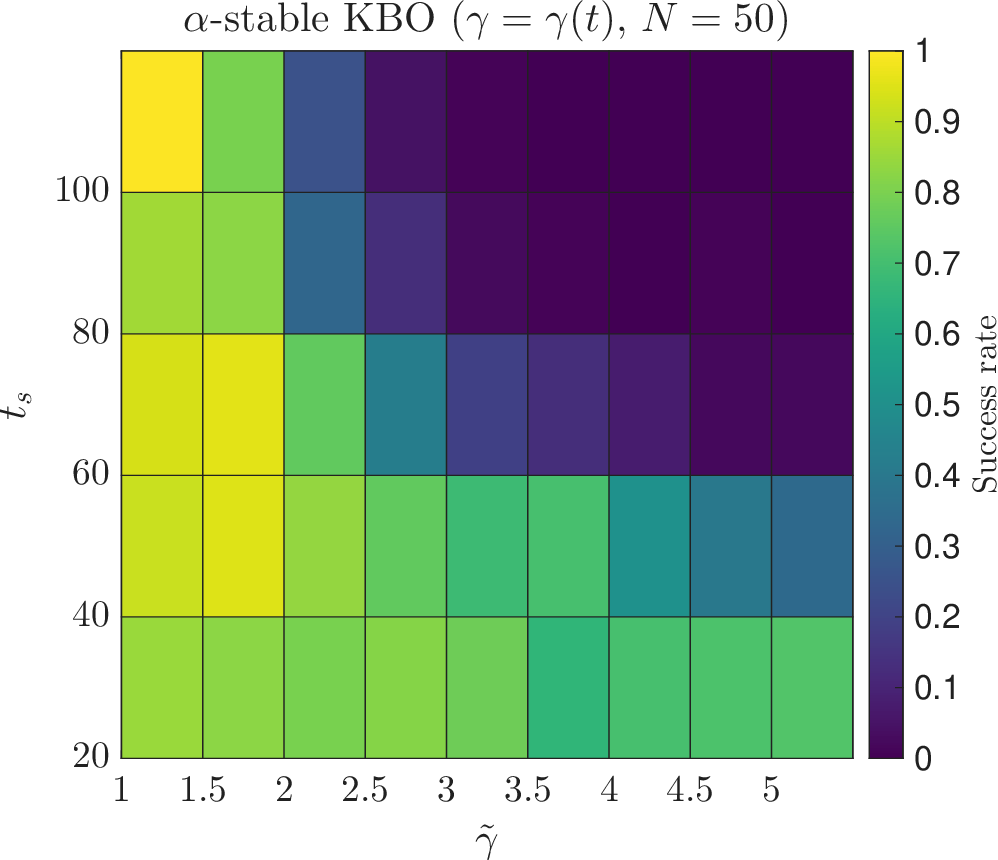} \includegraphics[width=0.327\linewidth]{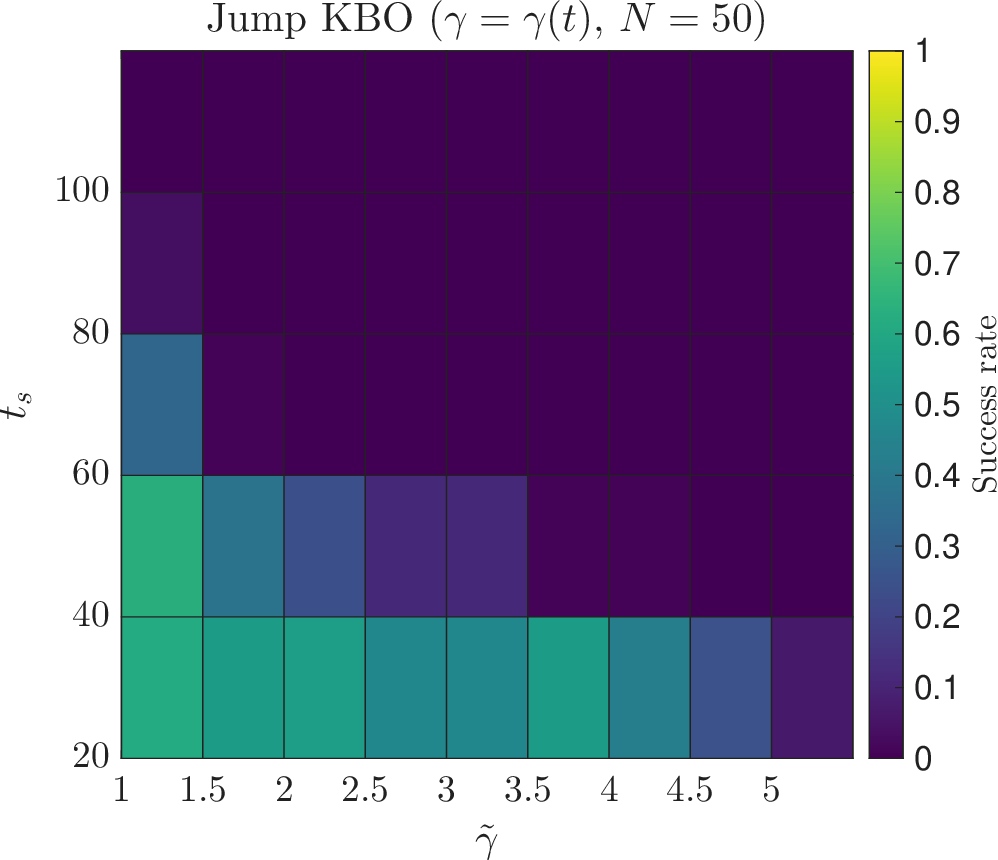}
	\caption{  Performance of KBO driven by $\alpha$-stable diffusion and compound Poisson dynamics \eqref{eq:jump_diffusion} ($\lambda=20$) for the Rastrigin function with $d=20$, $\nu=1$, $\sigma = 5.1$.
		Top row: success rate (left) and mean number of iterations (right) as $N$ varies with $\gamma(t)$ as in \eqref{eq:gamma} setting  $\tilde{\gamma}=1$ and $t_s=20$. Bottom row: success rate for KBO with $\alpha$-stable process (left) and  with compound Poisson dynamics (right) for $N=50$, with  $\gamma(t)$ as in \eqref{eq:gamma}, across varying  $t_s$ and $\tilde{\gamma}$.  Markers denote the value of the success rate and the number of iterations.} 
	\label{fig:CBO_jumpsvsalfa}
\end{figure} 

\section{Conclusions}
\label{sec:conclusions}

In this work, we have introduced a consensus based method where jumps are driven by an $\alpha$-stable stochastic process. By incorporating nonlocal stochastic effects, the proposed approach enhances the exploration of complex, high-dimensional, and non-convex landscapes. Starting from the particle-level formulation, we derived the corresponding fractional Fokker–Planck equation via a Fourier representation and provided a rigorous convergence analysis, thereby offering both modeling and theoretical support for the method.

The numerical experiments confirm that the fractional diffusion mechanism can improve performance with respect to standard diffusion-based approaches, particularly in challenging multimodal scenarios, where avoiding premature convergence is crucial. In addition to the isotropic diffusion matrix considered in the theoretical analysis, the numerical experiments highlight the potential of anisotropic diffusion and componentwise $\alpha$-stable variants in higher-dimensional settings. These formulations lead to different fractional Fokker-Planck models, whose rigorous analysis represents a natural direction for future work, in particular to establish corresponding convergence guarantees.
Several other directions for future research remain open. In particular, a promising perspective lies in exploring the bidirectional connection between the proposed framework and Scientific Machine Learning (SciML). On one hand, fractional Physics-Informed Neural Networks (fPINNs) could serve as a mesh-free numerical tool to evaluate the continuous, high-dimensional fractional mean-field limit of our particle system, bypassing the curse of dimensionality inherent to traditional grid-based PDE methods \cite{fPINNs}. On the other hand, the nonlocal exploratory features of the $\alpha$-stable process offer distinct advantages as a derivative-free global optimizer for specific SciML tasks such as parameter identification in physical inverse problems or loss-weight tuning where landscapes are notoriously non-convex and gradient-based methods suffer from bad local minima. 

\section*{Appendix}

In this appendix we prove Lemma~\ref{lem:extension}, i.e.\ that the weak
formulation \eqref{eq:weakc}, assumed for test functions in $C^2_c(\R^d)$,
extends to the class $\Ccal_p$ defined in \eqref{eq:classCp}, and that
$t\mapsto\int\varphi f(\cdot,t)\dd x$ is Lipschitz for $\varphi\in\Ccal_p$. We use the
notation of Section~\ref{sec:convergence}.

\begin{proof}
	Fix $\chi\in C^\infty_c(\R^d)$ with $0\le\chi\le1$, $\chi\equiv1$ on $B_1(0)$ and
	$\supp\chi\subset B_2(0)$, and set $\chi_n:=\chi(\cdot/n)$ and
	$\varphi_n:=\chi_n\varphi\in C^2_c(\R^d)$ for $\varphi\in\Ccal_p$. Let
	$M_p:=\sup_{[0,T]}V_p$; then $\int(1+|x|)^pf(\cdot,\tau)\dd x\le C(1+|x^\star|^p+M_p)$
	uniformly in $\tau$. By \eqref{eq:growthCp}, $|\varphi(x)|\le C(1+|x|)^p$.
	We apply \eqref{eq:weakc} to $\varphi_n$ and let
	$n\to\infty$. Each of the three terms converges by dominated convergence -- first
	in $x$ with respect to $f(\cdot,\tau)$, then in $\tau$ -- provided we show that
	$|\varphi_n|$, $|\nabla\varphi_n|\,|\xE-x|$ and $|(-\Delta)^{\alpha/2}\varphi_n|\,|\xE-x|^\alpha$
	are bounded by $C(1+|x|)^p$ uniformly in $n$ and $\tau$, and converge pointwise
	to the corresponding quantities for $\varphi$.
	
	\medskip
	
	\emph{(i)} $|\varphi_n|\le|\varphi|\le C(1+|x|)^p$ and $\varphi_n\to\varphi$ pointwise.
	
	\medskip
	
	\emph{(ii)} $\nabla\varphi_n=\chi_n\nabla\varphi+\varphi\nabla\chi_n$. Since
	$|\nabla\chi_n|\le\|\nabla\chi\|_\infty n^{-1}\mathbf 1_{\{n\le|x|\le2n\}}$ and
	$(1+|x|)/n\le3$ on that set, $|\varphi\nabla\chi_n|\le3C\|\nabla\chi\|_\infty(1+|x|)^{p-1}$,
	and $\varphi\nabla\chi_n\to0$ pointwise. Hence
	$|\nabla\varphi_n|\,|\xE-x|\le C(1+\bar B)(1+|x|)^p$ and $\nabla\varphi_n\to\nabla\varphi$.
	
	\medskip
	
	\emph{(iii)} From the algebraic identity
	$u(x)v(x)-u(y)v(y)=u(x)(v(x)-v(y))+v(x)(u(x)-u(y))-(u(x)-u(y))(v(x)-v(y))$
	and the principal-value form of \eqref{eq:fraclap} one obtains, for
	$u=\varphi$ and $v=\chi_n$,
	\begin{align*}
		(-\Delta)^{\alpha/2}(\chi_n\varphi)(x)
		={}&\chi_n(x)\,(-\Delta)^{\alpha/2}\varphi(x)+\varphi(x)\,(-\Delta)^{\alpha/2}\chi_n(x)\\
		&-c_{d,\alpha}\int_{\R^d}\frac{(\varphi(x)-\varphi(y))(\chi_n(x)-\chi_n(y))}{|x-y|^{d+\alpha}}\dd y,
	\end{align*}
	the last integral being absolutely convergent. We bound the three terms by
	$C(1+|x|)^{-s}$ uniformly in $n$.
	The first term is less or equal than $|(-\Delta)^{\alpha/2}\varphi(x)|\le C(1+|x|)^{-s}$.
	For the second, $(-\Delta)^{\alpha/2}\chi_n(x)=n^{-\alpha}\big((-\Delta)^{\alpha/2}\chi\big)(x/n)$
	by scaling, and $|(-\Delta)^{\alpha/2}\chi(z)|\le C_\chi\min\{1,|z|^{-d-\alpha}\}$
	since $\chi\in C^\infty_c$. If $|x|\le n$ then $1+|x|\le2n$ and
	$|\varphi(x)(-\Delta)^{\alpha/2}\chi_n(x)|\le CC_\chi(1+|x|)^pn^{-\alpha}\le2^\alpha CC_\chi(1+|x|)^{p-\alpha}$;
	if $|x|>n$ then
	$|\varphi(x)(-\Delta)^{\alpha/2}\chi_n(x)|\le CC_\chi(1+|x|)^p n^{d}|x|^{-d-\alpha}\le2^\alpha CC_\chi(1+|x|)^{p-\alpha}$.
	For the third, write $h:=y-x$ and use
	$|\varphi(x)-\varphi(y)|\le C\big[(1+|x|)^{p-1}+|h|^{p-1}\big]|h|$ (mean value
	theorem and $p-1\in(0,1)$) together with
	$|\chi_n(x)-\chi_n(y)|\le\min\{2,\|\nabla\chi\|_\infty|h|/n\}$.
	If $|x|<3n$, splitting at $|h|=n$ gives
	\begin{align*}
		\int_{\R^d} \frac{|\varphi(x)-\varphi(y)||\chi_n(x)-\chi_n(y)|}{|h|^{d+\alpha}} \dd y & \\
		& \hspace{-3cm}  \le C\Big[\frac{(1+|x|)^{p-1}n^{2-\alpha}+n^{p+1-\alpha}}{n}+(1+|x|)^{p-1}n^{1-\alpha}+n^{p-\alpha}\Big]\\
		& \hspace{-3cm} \le C\big[(1+|x|)^{p-1}n^{1-\alpha}+n^{p-\alpha}\big],
	\end{align*}
	where $\alpha>1$ and $\alpha>p$ were used for the integrals over $\{|h|\ge n\}$;
	since $1+|x|<4n$, this is $\le C(1+|x|)^{p-\alpha}$. If $|x|\ge3n$ then
	$\chi_n(x)=0$ and the integrand vanishes unless $y\in B_{2n}(0)$, where
	$|h|\ge|x|-2n\ge|x|/3$ and $|\varphi(x)-\varphi(y)|\le2C(1+|x|)^p$; hence the integral
	is $\le C(1+|x|)^p|B_{2n}|\,(|x|/3)^{-d-\alpha}\le C(1+|x|)^pn^d|x|^{-d-\alpha}\le C(1+|x|)^{p-\alpha}$.
	Altogether $\sup_n|(-\Delta)^{\alpha/2}\varphi_n(x)|\le C(1+|x|)^{-s}$, so that
	$|(-\Delta)^{\alpha/2}\varphi_n|\,|\xE-x|^\alpha\le C(1+\bar B)^\alpha(1+|x|)^p$ uniformly in
	$n$ and $\tau$. Finally $(-\Delta)^{\alpha/2}\varphi_n(x)\to(-\Delta)^{\alpha/2}\varphi(x)$
	for every $x$: for $n>|x|+1$ one has $\varphi_n=\varphi$ on $B_1(x)$, so the second
	differences satisfy $\Delta_h\varphi_n(x)\to\Delta_h\varphi(x)$ for every $h$ with
	$|\Delta_h\varphi_n(x)|\le C_x\min\{|h|^2,(1+|h|)^p\}$ uniformly in $n$, and dominated
	convergence applies in \eqref{eq:fraclap} because $p<\alpha$.
	
	Passing to the limit in \eqref{eq:weakc} written for $\varphi_n$ yields
	\eqref{eq:weakc} for $\varphi$. The same bounds show that the $\tau$-integrand
	is bounded by $C(1+\bar B)^\alpha(1+|x^\star|^p+M_p)$, whence
	\[
	\bigg| \int_{\R^d} \varphi(x) f( x,t) \dd x - \int_{\R^d} \varphi(x) f( x,t') \dd x\bigg| \le C|t-t'|,
	\]
	and \eqref{eq:weak}
	follows by differentiating the Lipschitz function $t\mapsto\int\varphi(x) f( x, t) \dd x$.
\end{proof}

\section*{Acknowledgments}
This work has been written within the activities of GNCS group of INdAM (Italian National Institute of High Mathematics). The research of FF was supported by the Italian Ministry of University and Research (MUR) through the Fondo Italiano per la Scienza 2022–2023 (Bando FIS 2), under Decree No. 1236 of 01/08/2023, within the FIS2023-01334 Advanced Grant “Tackling Complexity: Advanced Numerical Approaches for Multiscale Systems with Uncertainties” (ADAMUS)- CUP F53C25000210001. The research of MH has been supported by the Deutsche Forschungsgemeinschaft (German Research Foundation) through HE5386/33-1 Control of Interacting Particle Systems, and Their Mean-Field, and Fluid-Dynamic Limits (560288187) and 
HE5386/34-1 Partikelmethoden für unendlich dimensionale Optimierung (561130572). PAS acknowledges support from the Simons Foundation through the Travel Support for Mathematicians program.  GA acknowledges support by MUR project PRIN 2022 PNRR nr.~P2022JC95T, CUP: B53D23027840001.  We thank the two anonymous referees for their valuable comments, which helped improve this work. 
\nocite{*}

\bibliographystyle{abbrv}
\bibliography{references}

\end{document}